\documentclass[11pt]{amsart}

\ExplSyntaxOn
\msg_redirect_name:nnn { kernel } { variant-same-as-base } { info }
\ExplSyntaxOff

\usepackage{amsmath,amssymb,amsthm,url,scalerel}
\usepackage[utf8]{inputenc}
\usepackage[T1]{fontenc}
\usepackage{libertine}
\usepackage[libertine,cmintegrals,cmbraces]{newtxmath}
\usepackage{verbatim}
\usepackage[shortlabels]{enumitem}
\usepackage{stmaryrd}
\usepackage{mathtools}
\usepackage{microtype}
\usepackage{tabto}
\usepackage{xcolor}
\usepackage{tikz}
\usepackage{tikz-cd}
\tikzcdset{arrow style=tikz,
           diagrams={>=Straight Barb}
           }
\usepackage{nicefrac}
\usepackage{tabularx}
\usepackage{dsfont}
\usepackage{bbold}
\usepackage{todonotes}
\usepackage{quiver}
\usepackage{a4wide}
\usepackage{euscript}
\usepackage[all]{xy}
\usepackage{mathrsfs,comment}
\numberwithin{equation}{subsection}
\usepackage{adjustbox}
\usepackage{multicol}
\usepackage{listings}
\usepackage[table]{xcolor}
\usepackage[dvipsnames]{xcolor}
\usepackage{BOONDOX-calo}
\usepackage{xfrac}
\usepackage{booktabs}
\usepackage{float}
\usepackage{epigraph}

\usepackage[pagebackref]{hyperref}
  
\hypersetup{%
  bookmarksnumbered=true,
  colorlinks=true,%
  linkcolor=black,%
  citecolor=black,%
  filecolor=blue,%
  menucolor=black,%
  urlcolor=blue,%
  pdfnewwindow=true,%
  pdfstartview=FitBH}
\usepackage[capitalise]{cleveref}

\usepackage{etoolbox}

\makeatletter
\apptocmd{\thebibliography}{%
  \item[]\hspace*{-\leftmargin}%
  \begin{minipage}{\dimexpr\linewidth+\leftmargin\relax}
  \small
  \textbf{Note on publication data.}
  In the spirit of the \textit{Cost of Knowledge} campaign, journal information is included in this bibliography only as metadata identifying published versions of the cited works. No credit is intended to accrue to the journals themselves: the research cited here is due to the authors, and the labour of peer review belongs to the research community.
  \end{minipage}
  \vspace{1em}
}{}{}
\makeatother

\newtheorem{thm}{Theorem}[subsection]

\newtheorem{lem}[thm]{Lemma}

\newtheorem*{thm*}{Theorem}

\theoremstyle{definition}
\newtheorem{definition}[thm]{Definition}
\newtheorem{ex}[thm]{Example}

\newtheorem{rem}[thm]{Remark}
\newtheorem{ques}[thm]{Question}

\newcommand{\conn}{\operatorname{conn}}

\title{Metastable homotopy theory via Tate coalgebras}
\author{Marco Nervo}
\date{}

\begin{document}

\begin{abstract}
We develop foundations for metastable homotopy theory using Heuts' $2$-truncated Tate coalgebras. 
\end{abstract}

\maketitle
\tableofcontents

\section*{Introduction}

\begin{flushright}
\begin{minipage}{0.65\textwidth}
\small\itshape
But it is not real, you know. It is not stable, not solid -- nothing is. Things change, change.

\medskip
\hfill --- Ursula K. Le Guin, \textit{The Dispossessed}
\end{minipage}
\end{flushright}

\medskip

In recent years, there has been growing interest in building algebraic models for the category of spaces. Here by "algebraic" we mean, roughly, describing the category of spaces in terms of extra structure on a stable category.

Examples include Mandell's \textit{$\mathbb{E}_\infty$-algebras over $\overline{\mathbb{F}}_p$} \cite{Mandell}, Heuts' \textit{Tate coalgebras} \cite{Gijsapprox}, Yuan's \textit{Frobenius-fixed $\mathbb{E}_\infty$-rings} \cite{Yuan}, Horel's \textit{cosimplicial binomial rings} \cite{Horel}, Kubrak--Shuklin--Zakharov's \textit{derived binomial rings} \cite{KubrakShuklinZakharov}, and Antieau's \textit{perfect derived $\lambda$-rings with trivialized Adams operations} \cite{Antieau}. Other related algebraic models or approximations appear in work of Bachmann--Burklund \cite{BachmannBurklund}, Ekedahl \cite{Ekedahl}, Blomquist--Harper \cite{BlomquistHarper}, Riedel \cite{Riedel}, Lucio \cite{Lucio}, Nikolaus \cite{NikolausNotes,NikolausFrobenius}, Rubio--Sergeraert \cite{RubioSergeraert}, and Toën \cite{Toen}, among others.

It seems to us that, despite the proliferation of models, comparatively little has been done to \textit{exploit} them in the study of spaces. We focus on Heuts' Tate coalgebras because, among the available algebraic models, this one comes with a natural filtration
$$\mathcal{S}_*^{\geq 2} \simeq \mathrm{coAlg}_\mathrm{Tate}\mathcal{Sp}^{\geq 2} \longrightarrow \cdots \longrightarrow  P_m\mathcal{S}_* \longrightarrow P_{m-1}\mathcal{S}_*  \longrightarrow \cdots \longrightarrow P_2\mathcal{S}_* \longrightarrow P_1\mathcal{S}_* \simeq \mathcal{Sp}$$

From this point of view, the category of spectra appears as the first approximation to the category of pointed spaces. This approximation has already been studied extensively by stable homotopy theorists. The goal of this paper is to lay the foundations for studying the next stage, namely the category $P_2\mathcal{S}_*$. Even though the definition of this category is only about a decade old, it captures a variety of phenomena that were once grouped under the heading of \textit{metastable homotopy theory}. For this reason, we refer to $P_2\mathcal{S}_*$ as the metastable category.

The objects of the metastable category are spectra $X$ equipped with a metastable structure (def. \ref{metstructure}), namely a lift

\[\begin{tikzcd}
	& {(X^{\otimes 2})^{hC_2}} \\
	X & {(X^{\otimes 2})^{tC_2}}
	\arrow["{\mathrm{can}}", from=1-2, to=2-2]
	\arrow[dashed, from=2-1, to=1-2]
	\arrow["\Delta"', from=2-1, to=2-2]
\end{tikzcd}\]
of the Tate diagonal $\Delta$ along the canonical map from homotopy fixed points to the Tate construction.

The main example to keep in mind is that of suspension spectra, whose metastable structure arises from the space-level diagonal (ex. \ref{suspensionspectra}). On the other hand, if a spectrum admits a metastable structure, then its $\mathbb{F}_2$-cohomology must be unstable (lem. \ref{unstableobs}). In particular, non-connective bounded-below $2$-complete spectra do not admit a metastable structure (lem. \ref{connectobs}). The first section is devoted to further examples and non-examples, as well as to the classification of metastable structures.

There is also another equivalent description that is useful to keep in mind. A metastable structure is equivalent to giving a section of the canonical map $$(X^{\otimes 2})^{C_2} \longrightarrow (X^{\otimes 2})^{\Phi C_2} \simeq X$$ Hence spectra with a metastable structure are the same as spectra $X$ equipped with a tom Dieck splitting for the $C_2$-equivariant spectrum $X^{\otimes 2}$ (lem. \ref{metsect}). This perspective is closely related to Klein's work \cite{Kleinmoduli}, whose content overlaps with ours in some places.

Section 2 shows that these objects do not merely form a collection: they fit naturally into a category, namely the metastable category $P_2\mathcal{S}_*$ described above. There we present the construction of this category, which is taken from Heuts' \cite{Gijsapprox}. We then upgrade this construction to a symmetric monoidal category (lem. \ref{forgetmonoidal}), which is novel, and obtain symmetric monoidal colimit-preserving functors (lem. \ref{monoidalfromspaces})
\[\begin{tikzcd}
	{(\mathcal{S}_*, \wedge)} & {(P_2\mathcal{S}_*, \otimes)} & {(\mathcal{Sp}, \otimes)}
	\arrow["{\Sigma^\infty_2}", from=1-1, to=1-2]
	\arrow["{\Sigma^\infty}"', curve={height=18pt}, from=1-1, to=1-3]
	\arrow["{\Sigma^\infty_{}}", from=1-2, to=1-3]
\end{tikzcd}\]
This diagram exhibits the category of spectra as the stabilization of the metastable category (lem. \ref{forgetstabilization}). To complete the picture, we also compute its costabilization, which we prove to be equivalent to the category of spectra with trivial $C_2$ Tate diagonal (lem. \ref{cosp}). Along the way, we also compute the Picard group (lem. \ref{pic}), which we show to be trivial.

Once this category is available, one can begin to do mathematics inside it. In particular, one can form limits and colimits of metastable spectra and ask which familiar unstable phenomena survive in this setting. We end Section 2 by studying coproducts (lem. \ref{coprodp2}) and suspensions (lem. \ref{suspensionp2}). In particular, regarding suspensions, we prove (lem. \ref{suspensionp2gen}) that for every metastable spectrum $X$ the following are equivalent:

\begin{enumerate}
		\item $X$ is a suspension.
		\item $X$ is a co-H object.
		\item $X \to (X \otimes X)^{hC_2} \to X \otimes X$ is null.
	\end{enumerate}

recovering an old result of Berstein--Hilton \cite[thm.~A]{coHsusp}. 

Section 3 is probably the most fun and useful part, because it is where the theory starts producing computations. There we show that for every metastable spectrum $X$ and every natural number $n$ there is a fiber sequence (lem. \ref{intermediateres})
$$
X \xlongrightarrow{E^n} \Omega^n\Sigma^n X \xlongrightarrow{H^n} \Omega^\infty\Sigma^\infty(S(n\rho)_+ \otimes X^{\otimes 2})_{hC_2}
$$
where $\rho$ denotes the sign representation. This has several important specializations.

The case $n=\infty$ gives the canonical resolution (lem. \ref{canonicalres})
$$
X \xlongrightarrow{E^\infty} \Omega^\infty\Sigma^\infty X \xlongrightarrow{H^\infty} \Omega^\infty\Sigma^\infty X^{\otimes 2}_{hC_2}
$$
This should be thought of as the analogue of a resolution of a $2$-step nilpotent group in terms of abelian groups. We refer to the map $H^\infty$ as the James--Hopf map, since it induces the classical James--Hopf map after passing to pointed spaces.

The case $n=2$ gives a description of the fiber of the double suspension in terms of the swap map $\tau$ (lem. \ref{doublesuspp2})
$$
X \xlongrightarrow{E^2} \Omega^2\Sigma^2 X \xlongrightarrow{H^2} \Omega^\infty\Sigma^\infty X^{\otimes 2}/(1-\tau)
$$
This will be used to prove an exponent theorem (lem. \ref{exponentp2}): if the swap map on $X^{\otimes 2}$ is given by $-1$, then multiplication by $4^n$ on $\Omega^{2n}X$ factors through $X$.

The case $n=1$ gives the metastable EHP sequence (lem. \ref{EHP})
$$
X \xlongrightarrow{E} \Omega\Sigma X \xlongrightarrow{H} \Omega^\infty\Sigma^\infty X^{\otimes 2}
$$
This will be used to construct a spectral sequence for computing homotopy groups (lem. \ref{metSS}).

In order to state these computations, one first needs to define the spheres in the metastable category. For every $n$, we define $S^n$ to be the image of the classical sphere under the functor $\Sigma^\infty_2$ (def. \ref{spheresp2}). For $n > 0$, this is the only metastable structure on the spectrum $\mathbb{S}^n$, but the case $n = 0$ is different. In that case, $\mathbb{S}$ also admits exotic metastable structures (lem. \ref{exotics0}, def. \ref{spacelike}). This leads us to consider exotic $0$-spheres $S^{0,k}$, where $k$ is an odd natural number (def. \ref{exoticspheresp2}). To make the notation uniform for $n > 0$, we sometimes write $S^{n,0}$ for $S^n$. Having these spheres at hand, we define the bigraded homotopy groups by
$$
\pi_{n,j}X := [S^{n,j},X]
$$

The main interest is then to compute these bigraded homotopy groups for the spheres themselves. We compute them explicitly for the spheres $S^{0,k}$ and $S^1$, while for the spheres $S^n$ with $n>1$ the answer is organized by a spectral sequence.

We show that the bigraded homotopy groups of the exotic $0$-spheres $S^{0,k}$ are given by (lem. \ref{homgroupsS0})
$$
\pi_{n,j}S^{0,k} \cong
\begin{cases}
	\mathbb{Z}/2 & \text{if } n=0 \text{ and } j \nmid k \\
	\mathbb{Z}/2 \times \{0,k/j\} & \text{if } n=0 \text{ and } j \mid k \\
	\pi_{n+1}\mathbb{RP}^\infty_{-1} \oplus \pi_{n+1}(\mathbb{S}/k) & \text{if } n>0
\end{cases}
$$

where $\mathbb{RP}^\infty_{-1}$ is the spectrum $(\mathbb{S}^{-\rho})_{hC_2}$. 

Next, the bigraded homotopy groups of $S^1$ are given by (lem. \ref{homgroupsS1})
$$
\pi_{n,j}S^1 \cong
\begin{cases}
	0 & \text{if } n=0 \\
	\mathbb{Z} & \text{if } n=1 \\
	\pi_{n+1}\mathbb{Sp}^2 \oplus \pi_{n-1}\mathbb{S}[\sfrac{1}{2}] & \text{if } n>1
\end{cases}
$$

where $\mathbb{Sp}^2$ is the $2$-local symmetric square spectrum. 

For the spheres $S^n$ with $n > 1$ we use the metastable EHP sequence above to obtain a spectral sequence whose $E_1$-page is given by
$$
E^1_{t,m} :=
\begin{cases}
	\mathbb{Z} & m = 0,\ t = 0 \\
	\pi_{t+1}\mathbb{Sp}^2 & m = 0,\ t > 0 \\
	\pi_{t-m}^s & 0 < m < n \\
	0 & m > n
\end{cases}
$$
and which converges to $\pi_{n+t}S^n$. A companion paper displays and analyzes this spectral sequence in the setting of the Goodwillie tower of the identity \cite[sec. 3.2]{NervoCircle}. As a nice application of the computations developed here, we prove that no metastable sphere is a loop space, except possibly $S^7$ (lem. \ref{noloopspaces}). This concludes Section 3.

Section 4 collects a mix of further results and future directions. For example, we develop a notion of Mahowald invariant that works beyond maps on spheres (def. \ref{newmahowald}). In particular, we prove a generalization of Jones' theorem (lem. \ref{newjones}) which says the following: if $\alpha: X \to Y$ is a map of spectra and $X$ admits a metastable structure under which it is an $n$-suspension, then the Mahowald invariant raises the degree by at least $n$.

We also prove a Hilton--Milnor-like theorem (lem. \ref{hiltonmilnor}), namely that for any two metastable spectra $X_1$ and $X_2$, there is an equivalence
	$$
	\Omega (X_1 \vee X_2) \simeq \Omega X_1 \times \Omega X_2 \times \Omega^{\infty + 2}\Sigma^\infty(X_1 \otimes X_2)
	$$

We use this to compute the homotopy groups of $S^1 \vee S^1$ as (ex. \ref{homgroupsS1vS1})
$$
	\pi_{n, j}(S^1 \vee S^1) \cong
	\begin{cases}
		0 & n = 0 \\
		\mathrm{Heis}(\mathbb{Z}) & n = 1 \\
		\bigl(\pi_{n+1}\mathbb{Sp}^2 \oplus \pi_{n-1}\mathbb{S}[\sfrac{1}{2}]\bigr)^{\oplus 2} \oplus \pi_{n-1}\mathbb{S} & n > 1
	\end{cases}
	$$

where $\mathrm{Heis}(\mathbb{Z})$ denotes the integral Heisenberg group.
	
We conclude by discussing generalizations to an odd-primary analogue and by wondering about an internal metastable Adams spectral sequence. More generally, we hope that this work can serve as a modern and accessible foundation for metastable homotopy theory, and that others may build on it in the future.

\subsection*{Acknowledgments}

I would like to thank my supervisor Gijs Heuts for his support, trust, and for giving me the freedom to pursue my own ideas and interests, as well as my colleagues at Utrecht University and in the broader homotopy theory community around Utrecht, Amsterdam, and Nijmegen.

This research was supported by the European Research Council through the grant \textit{Chromatic homotopy theory of spaces} (grant no.~950048). It also benefited from Hood Chatham's Steenrod algebra and Adams spectral sequence calculators, from Nathanael Arkor's \textit{Quiver}, and from tools developed by OpenAI, which helped improve the exposition of the paper.

\section*{Background}

\subsection*{Conventions}

All categories in this paper are understood to be $\infty$-categories. Our standard references for this material are \cite{HTT} and \cite{HA}. We write $\mathrm{Cat}$ for the category of small categories, $\mathrm{Cat}_*$ for the category of pointed categories, $\mathrm{Cat}_*^\omega$ for the category of compactly generated pointed categories, and $\mathrm{Pr}^L$ for the category of \textit{presentable categories} and colimit-preserving functors. A basic feature of presentable categories is the \textit{adjoint functor theorem}: every colimit-preserving functor between presentable categories admits a right adjoint \cite[cor. 5.5.2.9]{HTT}. We will use this repeatedly throughout the paper.

If $\mathcal{C}$ is a category and $X,Y$ are objects of $\mathcal{C}$, we denote by $\mathrm{Map}_{\mathcal{C}}(X,Y)$ the mapping space from $X$ to $Y$. When $\mathcal{C}$ is stable, we denote by $\mathrm{map}_{\mathcal{C}}(X,Y)$ the corresponding mapping spectrum, so that
$$
\Omega^\infty \mathrm{map}_{\mathcal{C}}(X,Y) \simeq \mathrm{Map}_{\mathcal{C}}(X,Y)
$$

The main categories appearing throughout the paper are the category of spaces $\mathcal{S}$, the category of pointed spaces $\mathcal{S}_*$, and the category of spectra $\mathcal{Sp}$. These are symmetric monoidal with respect to, respectively, the cartesian product $\times$, the smash product $\wedge$, and the tensor product $\otimes$. Among them, $\mathcal{Sp}$ is stable.

We write $S^n$ for spheres in $\mathcal{S}$, and we will use the same notation for the corresponding spheres in the metastable category. More generally, since the metastable category should be thought of as an approximation to the category of spaces, we will often carry over notation from spaces to the metastable setting.

If $E$ and $X$ are spectra, we write $E_*X := \pi_*(E \otimes X)$
for the $E$-homology of $X$, and $E^*X := \pi_{-*}\mathrm{map}(X,E)$
for the $E$-cohomology of $X$. We will use this notation also when $E$ is an Eilenberg--MacLane spectrum, for instance $E = \mathbb{F}_p$, in place of the more classical notation $H_*(X;\mathbb{F}_p)$ and $H^*(X;\mathbb{F}_p)$.

Among the spectra that will play a role in this paper are the sphere spectrum $\mathbb{S}$, Eilenberg--MacLane spectra such as $\mathbb{Z}$ and $\mathbb{F}_p$, the stunted projective spectra $\mathbb{RP}^\infty_n := \Sigma^\infty RP^\infty / RP^{n-1}$, Morava $K$-theories $K(n)$, and the telescope spectra $T(n)$. We write $L_n$ for localization at $K(0) \oplus \cdots \oplus K(n)$ and $L_n^f$ for localization at $T(0) \oplus \cdots \oplus T(n)$. In either setting, we will say that a spectrum is \textit{monochromatic} of height $n$ if it is $L_n$-local and $L_{n-1}$-acyclic, or $L_n^f$-local and $L_{n-1}^f$-acyclic, respectively.

At several points we will work implicitly $p$-locally for a fixed prime $p$. In tables and spectral sequences, $\square$ denotes a copy of $\mathbb{Z}$, while $\bullet_r$ denotes a copy of $\mathbb{Z}/p^r$. When $r=1$, we omit the subscript.

\subsection*{Equivariant tools}

Let $G$ be a finite group. By \textit{Elmendorf's theorem} \cite{Gspaces}, the homotopy theory of topological $G$-spaces can be modeled by the category
$$
\mathrm{Fun}(\mathcal{O}(G)^{op}, \mathcal{S})
$$
where $\mathcal{O}(G)$ denotes the \textit{orbit category} of $G$, whose objects are the transitive $G$-sets $G/H$.

There are several equivariant versions of the category of spectra. The one most relevant for us is the naive category of \textit{spectra with $G$-action}, namely the functor category
$$
\mathcal{Sp}^{BG} \simeq \mathrm{Fun}(BG,\mathcal{Sp})
$$
We will also occasionally refer to the category $\mathcal{Sp}^G$ of \textit{genuine $G$-spectra}, obtained from pointed $G$-spaces by inverting all representation spheres. In this paper, however, all equivariant constructions will take place in the naive setting unless explicitly stated otherwise.

If $X$ is a spectrum with $G$-action, we write
$$
X_{hG} := \mathrm{colim}_{BG} X
\qquad\qquad
X^{hG} := \mathrm{lim}_{BG} X
$$
for its \textit{homotopy orbits} and \textit{homotopy fixed points}. These come equipped with natural maps
$$
X \xlongrightarrow{\pi} X_{hG}
\qquad\qquad
X^{hG} \xlongrightarrow{i} X
$$
There is also a canonical map
$$
\mathrm{nm} \colon X_{hG} \longrightarrow X^{hG}
$$
called the \textit{norm map} \cite[sec. 6.1.6]{HA} \cite[ex. I.1.11]{NikolausScholze}. It is characterized by the property that the composite
$$
X \xlongrightarrow{\pi} X_{hG} \xlongrightarrow{\mathrm{nm}} X^{hG} \xlongrightarrow{i} X
$$
agrees with the sum of the automorphisms $\tau_g$ induced by the $G$-action \cite[rem. 6.1.6.23]{HA}. The cofiber of the norm map is called the \textit{Tate construction} and is denoted by $X^{tG}$ \cite[def. 6.1.6.24]{HA}. Thus one has a canonical cofiber sequence
$$
X_{hG} \xlongrightarrow{\mathrm{nm}} X^{hG} \xlongrightarrow{\mathrm{can}} X^{tG}
$$

The Tate construction will play an important role throughout the paper. One useful feature is that it vanishes on spectra built out of free $G$-cells \cite[lem. I.3.8]{NikolausScholze}. Another is \textit{chromatic blueshift}: if $X$ is $L_n$-local, then $X^{tG}$ is $L_{n-1}$-local \cite{HoveySadofsky}. We will also use the analogous \textit{telescopic blueshift}, which asserts that if $X$ is $L_n^f$-local, then $X^{tG}$ is $L_{n-1}^f$-local \cite[thm. 1.5]{telescopicblue}.

In the special case $G = C_p$, we will frequently consider the \textit{Tate power construction}
$$
((-)^{\otimes p})^{tC_p}
$$
where $C_p$ acts on $(-)^{\otimes p}$ by cyclic permutation. This defines an exact functor \cite[prop. III.1.1]{NikolausScholze}, and it comes equipped with a natural transformation
$$
1 \xlongrightarrow{\Delta} ((-)^{\otimes p})^{tC_p}
$$
called the \textit{Tate diagonal} \cite[def. III.1.4]{NikolausScholze}. By the spectral Yoneda lemma \cite[prop. III.1.2]{NikolausScholze}, this transformation is determined by its value on the sphere spectrum. The resulting map
$$
\mathbb{S} \longrightarrow \mathbb{S}^{tC_p}
$$
is the composite
$$
\mathbb{S} \longrightarrow \mathbb{S}^{hC_p} \xrightarrow{\mathrm{can}} \mathbb{S}^{tC_p}
$$
where the first map is induced by equipping $\mathbb{S}$ with the trivial $C_p$-action. We will use the \textit{Segal conjecture} in the form stating that, for bounded-below spectra, the Tate diagonal identifies with $p$-completion \cite[thm. III.1.7]{NikolausScholze}.

We will also use representation spheres. If $V$ is a finite-dimensional real $G$-representation, we denote by $S^V$ its one-point compactification and by $S(V)$ the unit sphere in $V$. These fit into a cofiber sequence
$$
S(V)_+ \longrightarrow S^0 \longrightarrow S^V
$$
which is often useful in calculations.

The case $G = C_2$ will occur repeatedly. We write $\rho$ for the \textit{sign representation}. Then $S(n\rho)$ is the ordinary sphere $S^{n-1}$ equipped with the antipodal $C_2$-action, and in particular it is a free $C_2$-space. It follows that its ordinary orbits agree with its homotopy orbits. Passing to homotopy orbits in the cofiber sequence above, one obtains an identification
$$
(S^{n\rho})_{hC_2} \simeq RP^\infty/RP^{n-1} =: RP^\infty_n
$$
Motivated by this, for $n < 0$ we define the spectra
$$
\mathbb{RP}^\infty_n := (\mathbb{S}^{-n\rho})_{hC_2}
$$
These spectra are no longer suspension spectra.

\section{Metastable structures}

This section develops the basic theory of \emph{metastable structures} on spectra.  A metastable structure on a spectrum $X$ is, by definition, a lift of the Tate diagonal
\[
X \xlongrightarrow{\Delta} (X^{\otimes 2})^{tC_2}
\]
to the homotopy fixed points $(X^{\otimes 2})^{hC_2}$.  Suspension spectra provide the basic source of examples, and the section
examines to what extent the existence of a metastable structure forces a
spectrum to share structural features with suspension spectra.

In the first subsection we introduce the definition, discuss fundamental examples, and establish basic closure properties of the class of spectra admitting a metastable structure.

In the second subsection we develop obstructions to the existence of metastable structures. In particular, we obtain strong $2$-local constraints, leading to a range of explicit nonexamples.

In the third subsection we study classification.  We define the space of metastable structures and compute its connected components in several cases, exhibiting in particular infinitely many exotic metastable structures on the sphere spectrum.

In the fourth subsection we give an equivariant reformulation, identifying metastable structures with a tom Dieck splitting for the genuine $C_2$-spectrum $X^{\otimes 2}$ equipped with the swap action.

\subsection{Definition and examples}

In this subsection we introduce the notion of a metastable structure on a spectrum. This consists of a choice of lift of the Tate diagonal to the homotopy fixed points of the tensor square.

We then discuss examples of spectra admitting a metastable structure. The basic class of examples to keep in mind is given by suspension spectra, where the metastable structure is induced by the space-level diagonal. We also record further examples, including degenerate cases arising from triviality of the Tate diagonal.

Finally, we establish several closure properties of spectra admitting a metastable structure, showing in particular that this class is closed under suspensions, sums, and retracts, but not under arbitrary finite colimits.

Let $C_2$ act on the tensor square $X^{\otimes 2}$ by permutation of the factors. Associated to this action are the homotopy orbits $(X^{\otimes 2})_{hC_2}$, the homotopy fixed points $(X^{\otimes 2})^{hC_2}$, and the Tate construction $(X^{\otimes 2})^{tC_2}$. These fit into a canonical cofiber sequence
\[
(X^{\otimes 2})_{hC_2}
\xlongrightarrow{\mathrm{nm}}
(X^{\otimes 2})^{hC_2}
\xlongrightarrow{\mathrm{can}}
(X^{\otimes 2})^{tC_2}
\]
where $\mathrm{nm}$ is the \emph{norm map} \cite[ex. I.1.11]{NikolausScholze} and $\mathrm{can}$ is the
\emph{canonical map}.

Moreover, the Tate construction on the tensor square is an exact functor and it admits a natural map
\[
X \xlongrightarrow{\Delta} (X^{\otimes 2})^{tC_2}
\]
called the \emph{Tate diagonal} \cite[5.1]{Gijsapprox} \cite[def. III.1.4]{NikolausScholze} or, originally, the \emph{topological Singer $\epsilon$-map} \cite[def. 5.10]{Singercons}.
With these constructions in place, we are now ready to give our main definition.

\begin{definition}
	\label{metstructure}
A \emph{spectrum with a metastable structure} consists of a spectrum $X$
together with a map $X \to (X^{\otimes 2})^{hC_2}$ - called the \emph{metastable structure} - such that the following diagram commutes:
\[\begin{tikzcd}
	& {(X^{\otimes 2})^{hC_2}} \\
	X & {(X^{\otimes 2})^{tC_2}}
	\arrow["{\mathrm{can}}", from=1-2, to=2-2]
	\arrow[from=2-1, to=1-2]
	\arrow["\Delta"', from=2-1, to=2-2]
\end{tikzcd}\]
\end{definition}

\begin{rem}
The notion of a metastable structure coincides with that of a
\emph{$2$-truncated Tate coalgebra} in the sense of \cite[ex. 3.1]{Gijsapprox}.
\end{rem}

\begin{ex}
	\label{suspensionspectra}
Every suspension spectrum $\Sigma^\infty X$ admits a canonical metastable structure,
induced by the space-level diagonal map \cite[ex. 3.1]{Gijsapprox} \cite[lem. IV.1.3]{NikolausScholze} 
\[
X \longrightarrow X \times X
\longrightarrow X \wedge X
\]
In particular, the $n^\mathrm{th}$ suspension of the sphere spectrum $\mathbb{S}^n$
admits a metastable structure for every $n \geq 0$.
\end{ex}

From the definition, one immediately obtains the following simple criterion
for the existence of a metastable structure.

\begin{lem}
	\label{metnull}
A spectrum $X$ admits a metastable structure if and only if the composite
\[
X \longrightarrow (X^{\otimes 2})^{tC_2}
\longrightarrow \Sigma (X^{\otimes 2})_{hC_2}
\]
is null-homotopic.
\end{lem}

\begin{rem}
	The composite map above is precisely the map $\delta_X$ appearing in Klein's work \cite[def.~3.1, thm.~A]{Kleinmoduli}. 
	Thus, the obstruction to the existence of a metastable structure considered here agrees with Klein's obstruction.
	A detailed comparison will be given later.
\end{rem}

This lemma provides an alternative proof that suspensions of the sphere spectrum
admit metastable structures.

\begin{ex}
For $X = \mathbb{S}^n$, the composite above takes the form
\[
\mathbb{S}^n \longrightarrow (\mathbb{S}^n)^{\wedge}_2
\longrightarrow \Sigma^{n + 1}\mathbb{RP}^\infty_n
\]
which is trivial for $n \geq 0$ since $\mathbb{RP}^\infty_n$ is $n$-connective.
\end{ex}

The lemma may also be used to produce a class of metastable structures that are,
in a sense, degenerate, arising whenever the Tate diagonal is trivial.
The following are examples of this flavour.

\begin{ex}
	\label{2inv}
Every spectrum on which multiplication by $2$ is invertible admits a (trivial)
metastable structure. In particular, every rational spectrum admits a metastable
structure.
Indeed, if $X$ is an $\mathbb{S}/2$-acyclic spectrum, then $X^{\otimes 2}$ is also
$\mathbb{S}/2$-acyclic. Since the Tate construction kills
spectra on which 2 acts invertibly, it follows that $(X^{\otimes 2})^{tC_2} \simeq *$.
\end{ex}

\begin{ex}
	\label{0square}
Every spectrum with vanishing tensor square admits a (trivial) metastable
structure. In particular, the Brown--Comenetz dual $I_{\mathbb{Q}/\mathbb{Z}}$
admits a metastable structure.
Indeed, if $X$ is a spectrum such that $X^{\otimes 2} \simeq *$, then
$(X^{\otimes 2})^{tC_2} \simeq *$ as well.
\end{ex}

\begin{ex}
	\label{monochrom}
Every monochromatic spectrum admits a metastable structure as a consequence of
chromatic blueshift. In particular, $K(n)$ and $T(n)$ admit a metastable structure for every $n$.
Indeed, if $X$ is an $L_n$-local spectrum, then $X^{\otimes 2}$ is also $L_n$-local,
and the Tate construction $(X^{\otimes 2})^{tC_2}$ is $L_{n-1}$-local by blueshift \cite{HoveySadofsky}.
Consequently, the Tate diagonal factors through $L_{n-1}X$, which is trivial if
$X$ is assumed to be monochromatic.
The same conclusion holds in the telescopic setting, by telescopic blueshift \cite[thm. 1.5]{telescopicblue}\cite[cor. 2.7]{telescopicblueshort}.
\end{ex}

The following lemma records a simple dimension-connectivity criterion for the existence of a metastable structure.

\begin{lem}
	\label{existencemeta}
	If $X$ is a spectrum and $k \geq \dim X - 2\conn X$, then $\Sigma^k X$ admits a metastable structure. In particular, if $\dim X \leq 2\conn X$, then $X$ admits a metastable structure.
\end{lem}

\begin{proof}
	We first prove the special case $k = 0$, namely the case $\dim X \leq 2\conn X$. The general statement then follows from the identities
	$$\conn \Sigma^k X = k + \conn  X \qquad \dim \Sigma^k X = k + \dim X$$
	Assume that $X$ is an $n$-dimensional spectrum. By cellular approximation, we obtain the following diagram:

\[\begin{tikzcd}
	X & {(X^{\otimes 2})^{tC_2}} & {\Sigma(X^{\otimes 2})_{hC_2}} \\
	&& {\mathrm{sk}_n\Sigma(X^{\otimes 2})_{hC_2}}
	\arrow[from=1-1, to=1-2]
	\arrow[dashed, from=1-1, to=2-3]
	\arrow[from=1-2, to=1-3]
	\arrow[from=2-3, to=1-3]
\end{tikzcd}\]
	If $n \leq 2\conn X$, then $$n \leq 2\conn X \leq \conn  (X^{\otimes 2})_{hC_2} < \conn  \Sigma (X^{\otimes 2})_{hC_2}$$	In particular, the $n$-skeleton of $\Sigma(X^{\otimes 2})_{hC_2}$ is contractible, so the vertical map in the diagram is null-homotopic. Consequently, the horizontal composite is also null-homotopic, and $X$ admits a metastable structure.
\end{proof}

\begin{rem}
	The previous lemma does not lead to new examples beyond suspension spectra. Indeed, if a spectrum $X$ satisfies $\dim X \leq 2\conn X$, then $X$ is automatically a suspension spectrum by a result of Freudenthal \cite[cor.~2.2]{Kleinmoduli}. We include the lemma for completeness, as its proof is self-contained and does not use this observation.
\end{rem}

\begin{ques}
	Can one give an explicit example of a finite spectrum which is not a suspension spectrum but nevertheless admits a metastable structure?
\end{ques}

\begin{rem}
	The previous lemma implies that every finite-dimensional spectrum admits a metastable structure after a sufficiently large suspension. This conclusion no longer holds if the finite-dimensional hypothesis is dropped. For instance, as we shall see later, the spectrum
	\[
		\bigvee_{n \in \mathbb{Z}} \mathbb{S}^n \;\simeq\; \Sigma^k \bigvee_{n \in \mathbb{Z}} \mathbb{S}^n
	\]
	does not admit a metastable structure.
\end{rem}

We now record some basic closure properties of spectra admitting a metastable
structure.

\begin{lem}
The class of spectra admitting a metastable structure is closed under suspensions.
\end{lem}

\begin{proof}
	Let $X$ be a spectrum with a metastable structure. Then the composite \[
X \longrightarrow (X^{\otimes 2})^{tC_2}
\longrightarrow \Sigma (X^{\otimes 2})_{hC_2}
\]
	is null-homotopic. Suspending, we obtain a null-homotopic composite \[
\Sigma X \longrightarrow \Sigma(X^{\otimes 2})^{tC_2}
\longrightarrow \Sigma^2 (X^{\otimes 2})_{hC_2}
\]
	By the colimit comparison map, we have the following commutative diagram
\[\begin{tikzcd}
	{\Sigma X} & {\Sigma (X^{\otimes 2})^{tC_2}} & { \Sigma^2(X^{\otimes 2})_{hC_2}} \\
	{\Sigma X} & {((\Sigma X)^{\otimes 2})^{tC_2}} & {\Sigma((\Sigma X)^{\otimes 2})_{hC_2}}
	\arrow[from=1-1, to=1-2]
	\arrow[equals, from=1-1, to=2-1]
	\arrow[from=1-2, to=1-3]
	\arrow[equals, from=1-2, to=2-2]
	\arrow[from=1-3, to=2-3]
	\arrow[from=2-1, to=2-2]
	\arrow[from=2-2, to=2-3]
\end{tikzcd}\]
	Since the top horizontal composite is null-homotopic, so is the bottom one. Hence $\Sigma X$ admits a metastable structure. 
\end{proof}

\begin{lem}
	The class of spectra admitting a metastable structure is closed under sums.
\end{lem}

\begin{proof}
	Let $X_i$ be spectra with a metastable structure. Then the compositions \[
X_i \longrightarrow (X_i^{\otimes 2})^{tC_2}
\longrightarrow \Sigma (X_i^{\otimes 2})_{hC_2}
\]
	are null-homotopic. Taking their direct sum yields a null-homotopic composite
	$$
\textstyle\bigoplus\nolimits_i X_i \longrightarrow \textstyle\bigoplus\nolimits_i (X_i^{\otimes 2})^{tC_2} 
\longrightarrow \textstyle\bigoplus\nolimits_i \Sigma (X_i^{\otimes 2})_{hC_2} 
$$
	By the colimit comparison map, we have the following commutative diagram
\[\begin{tikzcd}
	{\bigoplus_i X_i} & {\bigoplus_i (X_i^{\otimes 2})^{tC_2} } & { \bigoplus_i \Sigma(X_i^{\otimes 2})_{hC_2} } \\
	{\bigoplus_i X_i} & {((\bigoplus_i X_i)^{\otimes 2})^{tC_2}} & {\Sigma((\bigoplus_i X_i)^{\otimes 2})_{hC_2}}
	\arrow[from=1-1, to=1-2]
	\arrow[equals, from=1-1, to=2-1]
	\arrow[from=1-2, to=1-3]
	\arrow[from=1-2, to=2-2]
	\arrow[from=1-3, to=2-3]
	\arrow[from=2-1, to=2-2]
	\arrow[from=2-2, to=2-3]
\end{tikzcd}\]
	Since the top horizontal composite is null-homotopic, so is the bottom one. Hence $\bigoplus_i X_i$ admits a metastable structure. 
\end{proof}

\begin{rem}
	The class of spectra admitting a metastable structure is \textit{not} closed under arbitrary finite colimits. For example, as we are going to see afterwards, the Moore spectrum $\mathbb{S}/2$ does not admit a metastable structure. This may seem surprising at first, since one can formally write down a diagram
\[\begin{tikzcd}
	{\mathrm{colim}_i X_i} & {\mathrm{colim}_i (X_i^{\otimes 2})^{tC_2}} & {\mathrm{colim}_i \Sigma(X_i^{\otimes 2})_{hC_2}} \\
	{\mathrm{colim}_i X_i} & {((\mathrm{colim}_i X_i)^{\otimes 2})^{tC_2}} & {\Sigma((\mathrm{colim}_i X_i)^{\otimes 2})_{hC_2}}
	\arrow[from=1-1, to=1-2]
	\arrow[equals, from=1-1, to=2-1]
	\arrow[from=1-2, to=1-3]
	\arrow[equals, from=1-2, to=2-2]
	\arrow[from=1-3, to=2-3]
	\arrow[from=2-1, to=2-2]
	\arrow[from=2-2, to=2-3]
\end{tikzcd}\]
	However, the top horizontal composite need not be null-homotopic, since in general a colimit of null-homotopic maps need not be null-homotopic. This phenomenon does not occur in the special cases of suspensions and sums considered above.
\end{rem}

\begin{lem}
	The class of spectra admitting a metastable structure is closed under retracts.
\end{lem}

\begin{proof}
	Let $X$ be a spectrum with a metastable structure. Then the composite\[
X \longrightarrow (X^{\otimes 2})^{tC_2}
\longrightarrow \Sigma (X^{\otimes 2})_{hC_2}
\]
	is null-homotopic. Suppose that $Y$ is a retract of $X$. Thus there exist an equivalence $Y \to X \to Y$ and, by functoriality of the constructions, we obtain the following commutative diagram

\[\begin{tikzcd}
	Y & {(Y^{\otimes 2})^{tC_2}} & { \Sigma(Y^{\otimes 2})_{hC_2} } \\
	X & {(X^{\otimes 2})^{tC_2}} & {\Sigma(X^{\otimes 2})_{hC_2}} \\
	Y & {(Y^{\otimes 2})^{tC_2}} & { \Sigma(Y^{\otimes 2})_{hC_2} }
	\arrow[from=1-1, to=1-2]
	\arrow[from=1-1, to=2-1]
	\arrow["\simeq"{description}, curve={height=18pt}, from=1-1, to=3-1]
	\arrow[from=1-2, to=1-3]
	\arrow[from=1-2, to=2-2]
	\arrow[from=1-3, to=2-3]
	\arrow[from=2-1, to=2-2]
	\arrow[from=2-1, to=3-1]
	\arrow[from=2-2, to=2-3]
	\arrow[from=2-2, to=3-2]
	\arrow[from=2-3, to=3-3]
	\arrow[from=3-1, to=3-2]
	\arrow[from=3-2, to=3-3]
\end{tikzcd}\]
	Since the middle composite is null-homotopic, so is the bottom one. Hence $Y$ admits a metastable structure.
\end{proof}

\subsection{Obstructions}

In this subsection we develop obstructions to the existence of metastable structures.

We first show that if a bounded-below spectrum admits a metastable structure, then its $2$-completion must be connective. This already excludes many examples, such as negative suspensions of the sphere spectrum.

We then refine this obstruction by showing that spectra admitting a metastable structure have unstable $\mathbb{F}_2$-cohomology. This excludes further examples, including the mod $2$ Moore spectrum, Eilenberg--MacLane spectra, and several classical connective spectra arising in chromatic homotopy theory.

\begin{lem}
	\label{connectobs}
If $X$ is a bounded-below spectrum with a metastable structure, then $X^\wedge_2$
is connective.
\end{lem}

\begin{proof}
Since the class of spectra admitting a metastable structure is closed under
suspensions, it suffices to consider the case where $X$ is $(-1)$-connective.
We will show that, in this case, if the composite
\[
X \longrightarrow (X^{\otimes 2})^{tC_2}
\longrightarrow \Sigma (X^{\otimes 2})_{hC_2}
\]
is null-homotopic, then $\pi_{-1}X$ is $2$-divisible, and hence $X^\wedge_2$ is
connective.

Set $G := \pi_{-1}X$. By the homotopy orbit spectral
sequence, we obtain
\[
\pi_{-1}\Sigma (X^{\otimes 2})_{hC_2} \cong (G \otimes G)_{C_2}
\]
where $C_2=\langle\sigma\rangle$ acts on $G\otimes G$ by $\sigma\cdot(x\otimes y) = -\, y\otimes x$. The minus sign appears because $G$ is in odd degree.
It follows that
\[
\pi_{-1}\Sigma (X^{\otimes 2})_{hC_2} \cong G \wedge G
\]
where $G \wedge G$ is the alternating square of $G$. Under this identification, the induced map on $\pi_{-1}$ is given by the linear map
\[
G \longrightarrow G \wedge G \qquad g \longmapsto g \wedge g
\]
This map is null if and only if $G/2=0$, as shown by the following algebraic lemma.
\end{proof}

\begin{lem}
Let $G$ be an abelian group. Then the following are equivalent:
\begin{enumerate}
\item $G/2=0$.
\item $G \to G\wedge G$ is zero.
\item $\mathrm{Hom}(G,C_2)=0$.
\end{enumerate}
\end{lem}

\begin{proof} We prove the equivalence by showing the implications
\begin{enumerate}[leftmargin=2.4cm]
\item[$(1)\Rightarrow(2)$]
Since $2(g\wedge g)=0$ for every $g\in G$, the map $G\to G\wedge G$ factors through
$G/2$, and hence is zero if $G/2=0$. 

\item[$(2)\Rightarrow(3)$]
Given a homomorphism $G\to C_2$ we obtain a commutative diagram
\[\begin{tikzcd}
	G & {G \wedge G} \\
	{C_2} & {C_2 \wedge C_2}
	\arrow[from=1-1, to=1-2]
	\arrow[from=1-1, to=2-1]
	\arrow[from=1-2, to=2-2]
	\arrow["\cong", from=2-1, to=2-2]
\end{tikzcd}\]
and the bottom horizontal map is an isomorphism. Thus any map $G\to C_2$ must be
zero.

\item[$(3)\Rightarrow(1)$]
The group $G/2$ is an $\mathbb{F}_2$-vector space with trivial linear dual, since
$$\mathrm{Hom}(G/2,\mathbb{F}_2)=\mathrm{Hom}(G,\mathbb{F}_2)=0$$
This implies $G/2=0$ (using the axiom of choice). \qedhere
\end{enumerate}
\end{proof}

Among the spectra excluded by this lemma, the following class deserves special
attention.

\begin{ex}
The desuspensions of the sphere spectrum $\mathbb{S}^n$ do not admit a metastable structure for any $n<0$.
\end{ex}

\begin{rem}
	This shows that the class of spectra admitting a metastable structure is not closed under desuspensions and, in particular, is not stable.
\end{rem}

We can use this observation to construct an infinite-dimensional spectrum which does not admit a metastable structure.

\begin{ex}
	The spectrum $\bigvee_{n \in \mathbb{Z}} \mathbb{S}^n$ does not admit a metastable structure. Indeed, it contains negative dimensional spheres as retracts, and the class of spectra admitting a metastable structure is closed under retracts.
\end{ex}

We can refine this criterion further. A metastable structure on a spectrum imposes strong restrictions on its $\mathbb{F}_2$-cohomology. To formulate these restrictions, recall that an $\mathcal{A}$-module is said to be \emph{unstable} if $Sq^i x = 0$ whenever $i > |x|$.

We briefly recall how Steenrod squares arise from the Tate diagonal.
Let $E$ be an $\mathbb{E}_\infty$-ring spectrum. The Tate diagonal can be used to construct the Tate valued Frobenius \cite[def. IV.1.1]{NikolausScholze}, defined as the composite
$$
E \xlongrightarrow{\Delta} (E^{\otimes 2})^{tC_2} \xlongrightarrow{\mu} E^{tC_2}
$$
where $E \otimes E \xrightarrow{\mu} E$ denotes the multiplication map. In the special case $E = \mathbb{F}_2$, there is an equivalence
$$
\mathbb{F}_2^{tC_2} \simeq \prod_{i \in \mathbb{Z}} \Sigma^i \mathbb{F}_2
$$
Under this identification, the Tate valued Frobenius is given by the product of all the maps $Sq^i$ for $i \geq 0$, and is zero on the remaining factors \cite[thm. IV.1.15]{NikolausScholze}.

\begin{lem}
	\label{unstableobs}
	If $X$ is a spectrum with a metastable structure, then $\mathbb{F}_2^*X$ is an unstable $\mathcal{A}$-module. 
\end{lem}

\begin{proof}
	Let $X \to \Sigma^n \mathbb{F}_2$ be a map. Then the following diagram commutes by naturality.
\[\begin{tikzcd}
	X & {\Sigma^n\mathbb{F}_2} \\
	{(X^{\otimes 2})^{tC_2}} & {((\Sigma^n\mathbb{F}_2)^{\otimes 2})^{tC_2}} & {(\Sigma^{2n}\mathbb{F}_2)^{tC_2}} & {\Sigma^n\prod_{i \in \mathbb{Z}} \Sigma^i\mathbb{F}_2} \\
	{\Sigma(X^{\otimes 2})_{hC_2}} & {\Sigma((\Sigma^n\mathbb{F}_2)^{\otimes 2})_{hC_2}} & {\Sigma(\Sigma^{2n}\mathbb{F}_2)_{hC_2}} & {\Sigma^n\prod_{i > n} \Sigma^i\mathbb{F}_2}
	\arrow[from=1-1, to=1-2]
	\arrow[from=1-1, to=2-1]
	\arrow[from=1-2, to=2-2]
	\arrow["{ (Sq^i)_{i \in \mathbb{N}}}", from=1-2, to=2-4]
	\arrow[from=2-1, to=2-2]
	\arrow[from=2-1, to=3-1]
	\arrow["\mu"{description}, from=2-2, to=2-3]
	\arrow[from=2-2, to=3-2]
	\arrow[equals, from=2-3, to=2-4]
	\arrow[from=2-3, to=3-3]
	\arrow[two heads, from=2-4, to=3-4]
	\arrow[from=3-1, to=3-2]
	\arrow["\mu"{description}, from=3-2, to=3-3]
	\arrow[equals, from=3-3, to=3-4]
\end{tikzcd}\]
	If $X$ has a metastable structure, the composite along the left column is null-homotopic. It follows that the composition $$X \longrightarrow \Sigma^n \mathbb{F}_2 \xlongrightarrow{Sq^i} \Sigma^{n+i} \mathbb{F}_2$$
	is null-homotopic whenever $i > n$.
\end{proof}

We are now in a position to exhibit examples of $2$-complete connective spectra that do not admit a metastable structure. Furthermore, these examples demonstrate that the class of spectra admitting a metastable structure is not closed under cofibers, as announced earlier.

\begin{ex}
	The cofibers of the Hopf maps $\mathbb{S}/2$, $\mathbb{S}/\eta$, $\mathbb{S}/\nu$, and $\mathbb{S}/\sigma$ do not admit a metastable structure. Indeed, their $\mathbb{F}_2$-cohomology is of the form
	\[
	\mathbb{F}_2\langle \iota, Sq^{2^n}\iota \rangle
	\]
	with $|\iota| = 0$ and, respectively, $n = 0, 1, 2, 3$. This is not an unstable $\mathcal{A}$-module.
\end{ex}

\begin{rem}
	The suspensions of the mod $2$ Moore spectrum $\mathbb{S}^n/2$ \textit{do} admit a metastable structure for every $n > 0$, since they are suspension spectra.
\end{rem}

Among the spectra excluded by the above lemma, the following class deserves special attention.

\begin{ex}
	Spectra whose $\mathbb{F}_2$-cohomology is of the form $\mathcal{A}/\!/B$ for some
proper augmented subalgebra $B$ do not admit a metastable
structure. In particular, this applies to the 2-local spectra $\mathbb{F}_2$, $\mathbb{Z}$, $ko$, $ku$, $tmf$, $k(n)$, $BP\langle n \rangle$, $BP$ as summarized in the following table.
\begin{center}
\begin{tabular}{|c|c|c|c|c|c|c|c|c|}
\hline
$X$
& $\mathbb{F}_2$ 
& $\mathbb{Z}$ 
& $ko$ 
& $ku$ 
& $tmf$ 
& $k(n)$ 
& $BP\langle n\rangle$  
& $BP$ \\
\hline
$\mathbb{F}_2^*X$ 
& $\mathcal{A}$ 
& $\mathcal{A}/\!/A(0)$ 
& $\mathcal{A}/\!/A(1)$ 
& $\mathcal{A}/\!/E(1)$ 
& $\mathcal{A}/\!/A(2)$ 
& $\mathcal{A}/\!/E(Q_n)$ 
& $\mathcal{A}/\!/E(n)$ 
& $\mathcal{A}/\!/E$ \\
\hline
\end{tabular}
\end{center}
	Indeed, since $B$ is a proper subalgebra, there exists a Steenrod operation
$Sq^i \notin B$ acting nontrivially on the unit class in degree zero.
This violates the instability condition, and hence $\mathcal{A}/\!/B$ is not
unstable.
\end{ex}

\begin{ex}
	The complex cobordism spectrum $MU$ does not admit a metastable structure. Indeed, after $2$-localization the spectrum $BP$ is a retract of $MU$, and metastable structures are closed under retracts.
\end{ex}

\begin{rem}
	The previous lemma recovers our first criterion. Indeed, let $X$ be a bounded below spectrum whose $\mathbb{F}_2$-cohomology is unstable. Then $\mathbb{F}_2^*X$ vanishes in every negative degree. Explicitly, if $x \in \mathbb{F}_2^* X$, instability implies $x = Sq^0 x = 0$ whenever $0 > |x|$. Since $X$ is bounded below, it follows that the $2$-completion of $X$ is connective.
\end{rem}

Not only does the previous lemma recover the first criterion, but it also strengthens it. In particular, the existence of a metastable structure not only forces the $2$-completion of $X$ to be connective, but also implies that \(\pi_0X\) does not contain \(\mathbb{Z}/2\) as a direct summand.

\begin{lem}
If \(X\) is a bounded-below spectrum with a metastable structure, then \(X^\wedge_2\)
is connective and \(\pi_0X\) does not contain \(\mathbb{Z}/2\) as a direct summand.
\end{lem}

\begin{proof}
	We already know from lemma \ref{connectobs} that \(X^\wedge_2\) is connective. We now use that \(\mathbb{F}_2^*X\) is an unstable \(\mathcal{A}\)-module to prove that \(\pi_0X\) does not contain \(\mathbb{Z}/2\) as a direct summand.

	Consider a map \(\pi_0X \to \mathbb{Z}/2\). We want to prove that it factors through \(\mathbb{Z}/4\), and therefore cannot admit a section. Indeed, consider the class \(\theta\) in degree \(0\) given by the composite
	$$
	X \longrightarrow \pi_0X \longrightarrow \mathbb{Z}/2
	$$
	Since \(\mathbb{F}_2^*X\) is an unstable \(\mathcal{A}\)-module, we have \(Sq^1\theta = 0\), so \(\theta\) factors as
\[\begin{tikzcd}
	&& {\mathbb{Z}/4} & \\
	X & {\pi_0X} & {\mathbb{Z}/2} & {\Sigma\mathbb{Z}/2}
	\arrow[from=1-3, to=2-3]
	\arrow[dashed, from=2-1, to=1-3]
	\arrow[from=2-1, to=2-2]
	\arrow["\theta"{description}, curve={height=12pt}, from=2-1, to=2-3]
	\arrow[from=2-2, to=2-3]
	\arrow["{Sq^1}"', from=2-3, to=2-4]
\end{tikzcd}\]
	Passing to $\pi_0$ then gives the desired factorization.
\end{proof}

\begin{rem}
	We do not know whether this can be strengthened to show that $\pi_0X$ cannot contain any $2$-torsion. In particular, it is not known whether the Moore spectra $\mathbb{S}/2^n$ admit a metastable structure for any $n > 1$. More generally, we do not know whether $\mathbb{S}/\alpha$ admits a metastable structure for any nonzero, noninvertible $2$-local map $\alpha$.
\end{rem}

\begin{ques}
	Do the Moore spectra $\mathbb{S}/2^n$ admit a metastable structure for $n > 1$?
\end{ques}

\subsection{Classification}

In this subsection we study the problem of classifying metastable structures on a
spectrum. We introduce the space of metastable structures and compute its
connected components in several examples. These computations show that, for
suspension spectra, not every metastable structure arises from the space-level
diagonal. This motivates the introduction of the notion of exotic metastable
structures, and we conclude by exhibiting exotic structures on the sphere
spectrum.

\begin{definition}
	The \emph{space of metastable structures} $\overline{\mathcal{M}}_X$ on a spectrum $X$ is defined as the pullback
\[\begin{tikzcd}
	{\overline{\mathcal{M}}_X} & {\mathrm{Map}(X, (X^{\otimes 2})^{hC_2})} \\
	{*} & {\mathrm{Map}(X, (X^{\otimes 2})^{tC_2})}
	\arrow[from=1-1, to=1-2]
	\arrow[from=1-1, to=2-1]
	\arrow["\lrcorner"{anchor=center, pos=0.125}, draw=none, from=1-1, to=2-2]
	\arrow["{\mathrm{can}_*}", from=1-2, to=2-2]
	\arrow["\Delta"', from=2-1, to=2-2]
\end{tikzcd}\]
	The \emph{set of metastable structures} $\mathcal{M}_X$ on $X$ is defined as its connected components $\mathcal{M}_X := \pi_0 \overline{\mathcal{M}}_X$.
\end{definition}

\begin{lem}
	Let $X$ be a spectrum admitting a metastable structure. Then there are equivalences $$\overline{\mathcal{M}}_X \simeq \mathrm{Map}(X, X^{\otimes 2}_{hC_2}) \qquad \mathcal{M}_X \cong [X, X^{\otimes 2}_{hC_2}]$$
\end{lem}

\begin{proof}
	The second statement follows immediately from the first. For the first, recall
	that $\overline{\mathcal{M}}_X$ is defined by a pullback diagram. Since $X$
	admits a metastable structure, this pullback may be taken in the category of
	pointed spaces. It follows that
	\begin{align*}
		\overline{\mathcal{M}}_X &\simeq \mathrm{fib}(\mathrm{Map}(X, (X^{\otimes 2})^{hC_2}) \xlongrightarrow{\mathrm{can}_*}\mathrm{Map}(X, (X^{\otimes 2})^{tC_2})) \\
		&\simeq \mathrm{Map}(X, \mathrm{fib}((X^{\otimes 2})^{hC_2} \xlongrightarrow{\mathrm{can}} (X^{\otimes 2})^{tC_2})) \\
		&\simeq \mathrm{Map}(X, X^{\otimes 2}_{hC_2})
	\end{align*} 
	which completes the proof.
\end{proof}

\begin{rem}
	The reader may wish to compare again this discussion with Klein's work \cite[thm.~B]{Kleinmoduli}.
\end{rem}

\begin{ex}\label{metstructsn}
The $n^\mathrm{th}$ suspension of the sphere spectrum $\mathbb{S}^n$ has set of metastable structures given by
\[
\mathcal{M}_{\mathbb{S}^n}
\cong
[\mathbb{S}^n,\Sigma^{n}\mathbb{RP}^\infty_n]
\cong
\pi_0\mathbb{RP}^\infty_n
\cong
\begin{cases}
\mathbb{Z} & n=0 \\
0 & n>0
\end{cases}
\]
Indeed, this follows from the discussion above together with the identification
\[
(\mathbb{S}^n)^{\otimes 2}_{hC_2}
\simeq
\Sigma^{n}\mathbb{RP}^\infty_n
\]
In particular, $\mathbb{S}^n$ admits a unique metastable structure for $n>0$, and infinitely many for $n=0$. As we will see below, only one of the metastable structures on $\mathbb{S}$ arises from a space-level diagonal.
\end{ex}

\begin{ex}
The direct sum of $k$ copies of the sphere spectrum $\mathbb{S}^{\oplus k}$ has set of metastable structures given by
\[
\mathcal{M}_{\mathbb{S}^{\oplus k}}
\cong
\big[\mathbb{S}^{\oplus k},
(\mathbb{RP}^\infty_+)^{\oplus k}
\oplus
\mathbb{S}^{\oplus \binom{k}{2}}
\big]
\cong
\mathbb{Z}^{\frac{k^2(k+1)}{2}}
\]
Indeed, using the decomposition
\[
(X\oplus Y)^{\otimes 2}_{hC_2}
\simeq
X^{\otimes 2}_{hC_2}
\oplus
Y^{\otimes 2}_{hC_2}
\oplus
X\otimes Y
\]
one obtains
\[
(\mathbb{S}^{\oplus k})^{\otimes 2}_{hC_2}
\simeq
(\mathbb{RP}^\infty_+)^{\oplus k}
\oplus
\mathbb{S}^{\oplus \binom{k}{2}}
\]
and the claim follows by taking maps out of $\mathbb{S}^{\oplus k}$.
In particular, taking direct sums produces a rapid increase in the number of metastable structures.
\end{ex}

\begin{ex}
The suspension spectrum of the product of two $n$-spheres $\Sigma^\infty(S^n\times S^n)$ with $n>0$ has set of metastable structures given by
\begin{align*}
\mathcal{M}_{\Sigma^\infty(S^n\times S^n)}
&\cong
\mathbb{Z}
\oplus
\pi_n\mathbb{RP}^\infty_n
\oplus
\pi_n\mathbb{RP}^\infty_n \\
&\cong
\begin{cases}
\mathbb{Z}^{\oplus 3} & n \text{ even} \\
\mathbb{Z} \oplus (\mathbb{Z}/2)^{\oplus 2} & n \text{ odd}
\end{cases}
\end{align*}
Indeed, using the stable splitting
\[
\Sigma^\infty(S^n\times S^n)
\simeq
\mathbb{S}^n \oplus \mathbb{S}^n \oplus \mathbb{S}^{2n}
\]
together with the above decomposition formula for homotopy orbits, one reduces to computing
\[
[\mathbb{S}^{2n},
\Sigma^{n}\mathbb{RP}^\infty_n
\oplus
\Sigma^{n}\mathbb{RP}^\infty_n
\oplus
\mathbb{S}^{2n}]
\]
which yields the stated description.
In particular, as observed by Klein \cite[sec.~8]{Kleinmoduli}, all these metastable structures arise from suspension spectra.
\end{ex}

\begin{ex}
The suspension of the Moore spectrum $\mathbb{S}^n/2$ with $n>0$ has set of metastable structures given by
\[
\mathcal{M}_{\mathbb{S}^n/2}
\cong
\pi_{n+1}\big((\mathbb{S}^n/2)^{\otimes 2}_{hC_2}\otimes \mathbb{S}/2\big)
\cong
\begin{cases}
\mathbb{Z}/2 & n=1 \\
0 & n>1
\end{cases}
\]
Indeed, this follows from
\[
\overline{\mathcal{M}}_{\mathbb{S}^n/2}
\simeq
\mathrm{Map}(\mathbb{S}^n/2,
(\mathbb{S}^n/2)^{\otimes 2}_{hC_2})
\simeq
\Omega^n
\mathrm{Map}(\mathbb{S}/2,
(\mathbb{S}^n/2)^{\otimes 2}_{hC_2})
\]
and the general identification
\[
\mathrm{Map}(\mathbb{S}/2,X)
\simeq
\Omega^\infty \mathrm{map}(\mathbb{S}/2,X)
\simeq
\Omega^{\infty+1}X/2
\]
In particular, $\mathbb{S}^n/2$ admits a unique metastable structure for $n>1$, and exactly two for $n=1$.
\end{ex}

\begin{ex}
The suspension $\mathbb{S}^n/\alpha$ of the cofiber of a map $\alpha\in\pi_k\mathbb{S}$ with $n>1$ - whenever it admits a metastable structure - has set of metastable structures given by
\[
\mathcal{M}_{\mathbb{S}^n/\alpha}
\cong
\pi_{k+1}\mathbb{RP}^\infty_n
\]
Indeed, $(\mathbb{S}^n/\alpha)^{\otimes 2}_{hC_2}$ is $(2n-1)$-connected. Since $\mathbb{S}^n/\alpha$ has cells in degrees $n$ and $n+k+1$, and since $n+1\le 2n-1$, the bottom cell contributes no maps, so
\[
\mathcal{M}_{\mathbb{S}^n/\alpha}
\cong
[\mathbb{S}^{k+n+1},
(\mathbb{S}^n/\alpha)^{\otimes 2}_{hC_2}]
\]
Moreover, the cofiber of
\[
(\mathbb{S}^n)^{\otimes 2}_{hC_2}
\longrightarrow
(\mathbb{S}^n/\alpha)^{\otimes 2}_{hC_2}
\]
is $(2n+k)$-connected, and since $k+n+2\le 2n+k$ the induced map gives an isomorphism on maps out of $\mathbb{S}^{k+n+1}$. Thus one reduces to
\[
[\mathbb{S}^{k+n+1},\Sigma^n\mathbb{RP}^\infty_n]
\]
In particular, if $n>k+1$, the metastable structure is unique. The first nontrivial case occurs when $n=k+1$, in which case
\[
\mathcal{M}_{\mathbb{S}^{k+1}/\alpha}
\cong
\begin{cases}
\mathbb{Z} & k \text{ odd} \\
\mathbb{Z}/2 & k \text{ even}
\end{cases}
\]
As observed by Klein \cite[sec.~8]{Kleinmoduli}, these metastable structures arise from suspension spectra.
\end{ex}

The previous example illustrates that suspending a finite spectrum sufficiently often forces the metastable structure to be unique. The following lemma shows that this is a general phenomenon, and in fact strengthens lemma \ref{existencemeta}. We omit the proof, since it is entirely analogous to that of lemma \ref{existencemeta}.

\begin{lem}
	If $X$ is a spectrum and $k > \dim X - 2\conn X$, then $\Sigma^k X$ admits a unique metastable structure. In particular, if $\dim X < 2\conn X$, then $X$ admits a unique metastable structure.
\end{lem}

Although all the spectra appearing in the examples above are suspension spectra, they admit metastable structures that do not arise from the underlying suspension spectrum structure. This motivates the following definition.

\begin{definition}
	\label{spacelike}
	A metastable structure on a suspension spectrum $\Sigma^\infty X$ is called
	\emph{spacelike} if it is induced by the space-level diagonal
	$X \to X \wedge X$. A metastable structure on a suspension spectrum which is not spacelike is called
	\emph{exotic}.
\end{definition}

We conclude this subsection by showing that the sphere spectrum $\mathbb{S}$ admits
infinitely many exotic metastable structures. The key input is the following lemma.

\begin{lem}
	\label{spacelikes0}
	For any spacelike metastable structure on $\mathbb{S}$ the composite
	\[
	\mathbb{S} \longrightarrow (\mathbb{S}^{\otimes 2})^{hC_2}
	\longrightarrow \mathbb{S}^{\otimes 2} \simeq \mathbb{S}
	\]
	is the identity.
\end{lem}

\begin{proof}
	Take a spacelike metastable structure on $\mathbb{S}$. By definition,
	there exists a space $X$ such that $\Sigma^\infty X \simeq \mathbb{S}$ and such
	that the composite
	\[
	\mathbb{S} \longrightarrow (\mathbb{S}^{\otimes 2})^{hC_2}
	\longrightarrow \mathbb{S}^{\otimes 2} \simeq \mathbb{S}
	\]
	coincides with the suspension of the space-level diagonal
	\[
	\Sigma^\infty X \longrightarrow \Sigma^\infty(X \wedge X)
	\simeq \Sigma^\infty X \otimes \Sigma^\infty X
	\]
	By homological considerations, $X$ has two connected components. Hence, there
	exists a map $S^0 \to X$ inducing the identity on $H_0$. In particular, after
	applying $\Sigma^\infty$, this map is the identity. By functoriality of the
	diagonal, we obtain the following commutative diagram:
\[\begin{tikzcd}
	{\Sigma^\infty S^0} & {\Sigma^\infty X} \\
	{\Sigma^\infty S^0 \otimes \Sigma^\infty S^0} & {\Sigma^\infty X\otimes \Sigma^\infty X}
	\arrow[equals, from=1-1, to=1-2]
	\arrow[equals, from=1-1, to=2-1]
	\arrow[from=1-2, to=2-2]
	\arrow[equals, from=2-1, to=2-2]
\end{tikzcd}\]
	This proves that the right vertical map has to be the identity.
\end{proof}

\begin{rem}
	It is worth recalling that there exist spaces $X \not\simeq S^0$ such that $\Sigma^\infty X \simeq \mathbb{S}$. For example, if $A$ is an acyclic space, then $X := A_+$ satisfies this property. The lemma shows that, regardless of this choice of $X$, any metastable structure arising from a space-level diagonal induces the same self-map of $\mathbb{S}$.
\end{rem}

\begin{lem}
	\label{exotics0}
	The sphere spectrum $\mathbb{S}$ admits infinitely many metastable structures. These are distinguished by the degree of the induced diagonal
	\[
	\mathbb{S} \longrightarrow (\mathbb{S}^{\otimes 2})^{hC_2} \longrightarrow \mathbb{S}^{\otimes 2} \simeq \mathbb{S}
	\]
	which can be any odd integer. In particular, the spacelike structure is the unique one with diagonal of degree $1$, and the exotic structures are exactly those with diagonal of odd degree different from $1$.
\end{lem}

\begin{proof}
	Recall from example \ref{metstructsn} that $\mathcal{M}_{\mathbb{S}} \cong \mathbb{Z}$, and hence that the sphere spectrum $\mathbb{S}$ admits infinitely many metastable structures. Let $L$ denote the metastable structure on $\mathbb{S}$ arising from the space-level diagonal on $S^0$. Then any metastable structure on $\mathbb{S}$ is of the form
	\[
	L' = L + k\,\mathrm{nm} \circ \pi
	\]
	where $k \in \mathbb{Z}$ and $\pi$ is the projection map
	\[
	\mathbb{S} \longrightarrow \mathbb{S}_{hC_2} \simeq \mathbb{RP}^\infty_+
	\]
	We claim that the composite
	\[
	\mathbb{S} \xlongrightarrow{L'}
	(\mathbb{S}^{\otimes 2})^{hC_2}
	\longrightarrow
	\mathbb{S}^{\otimes 2} \simeq \mathbb{S}
	\]
	has degree $2k + 1$. In particular, by the previous lemma, $L'$ is exotic for every $k \neq 0$.

	To prove the claim, recall that for any spectrum $X$, the composite
	\[
	X^{\otimes 2}
	\xlongrightarrow{\pi}
	(X^{\otimes 2})_{hC_2}
	\xlongrightarrow{\mathrm{nm}}
	(X^{\otimes 2})^{hC_2}
	\xlongrightarrow{i}
	X^{\otimes 2}
	\]
	coincides with the map
	\[
	X^{\otimes 2} \xlongrightarrow{1 + \tau} X^{\otimes 2}
	\]
	where $\tau$ denotes the swap map.

	In the case $X = \mathbb{S}$, the swap map is the identity, so the above composite is multiplication by $2$. This is illustrated by the diagram
\[
\begin{tikzcd}
	{\mathbb{S}} & {\mathbb{S}_{hC_2}} & {\mathbb{S}^{hC_2}} & {\mathbb{S}}
	\arrow["\pi", from=1-1, to=1-2]
	\arrow["2"', curve={height=18pt}, from=1-1, to=1-4]
	\arrow["{\mathrm{nm}}", from=1-2, to=1-3]
	\arrow["i", from=1-3, to=1-4]
\end{tikzcd}
\]
	It follows that the contribution of $k\,\mathrm{nm}\circ\pi$ to the above composite has degree $2k$, while the contribution of $L$ has degree $1$. Hence the total degree is $2k + 1$, as claimed.
\end{proof}

\begin{rem}
	Every metastable structure on $\mathbb{S}$ has induced diagonal of odd degree. In particular, this diagonal is never null. As we will see later, this implies that no metastable structure on $\mathbb{S}$ desuspends.
\end{rem}

\subsection{Equivariant reformulation} 

The goal of this subsection is to reinterpret metastable structures in equivariant
terms. We will show that the existence of a metastable structure on a spectrum
$X$ is equivalent to a splitting property of the genuine $C_2$-spectrum
$X^{\otimes 2}$ endowed with the swap action.

To make this reformulation precise, we briefly recall the relevant fixed point
constructions for genuine $C_2$-spectra.

Let $Y$ be a genuine $C_2$-spectrum. Besides the homotopy fixed points $Y^{hC_2}$, it also comes equipped with the \textit{genuine fixed points} $Y^{C_2}$ and the \textit{geometric fixed points} $Y^{\Phi C_2}$. Moreover, there is a comparison map
\[
Y^{C_2} \longrightarrow Y^{\Phi C_2}
\]
With these constructions in place, we are now ready to give the following definition

\begin{definition}
	A genuine $C_2$-spectrum $Y$ is said to admit a \textit{tom Dieck splitting} if the comparison map $Y^{C_2} \longrightarrow Y^{\Phi C_2}$ admits a section.
\end{definition}

To relate this to the classical tom Dieck splitting, recall that the above map
fits into the \textit{Tate square} \cite{Tate}

\[
\begin{tikzcd}
	{Y^{C_2}} & {Y^{hC_2}} \\
	{Y^{\Phi C_2}} & {Y^{tC_2}}
	\arrow[from=1-1, to=1-2]
	\arrow[from=1-1, to=2-1]
	\arrow["\lrcorner"{anchor=center, pos=0.125}, draw=none, from=1-1, to=2-2]
	\arrow[from=1-2, to=2-2]
	\arrow[from=2-1, to=2-2]
\end{tikzcd}
\]
In particular, taking vertical fibers yields a fiber sequence
\[
Y_{hC_2} \longrightarrow Y^{C_2} \longrightarrow Y^{\Phi C_2}
\]
From this it follows immediately that

\begin{lem}
	A $C_2$-spectrum $Y$ admits a tom Dieck splitting if and only if
	\[
	Y^{C_2} \simeq Y^{\Phi C_2} \oplus Y_{hC_2}
	\]
\end{lem}

This splitting coincides with the classical tom Dieck splitting whenever $Y$ is a suspension spectrum. The main goal of this subsection is to relate metastable structures to tom Dieck splittings.

\begin{lem}
	\label{metsect}
	Let $X$ be a spectrum. Then the following are equivalent
	\begin{enumerate}
		\item $X$ admits a metastable structure.
		\item $X^{\otimes 2}$ admits a tom Dieck splitting.
	\end{enumerate}
\end{lem}

To prove this, recall that for $Y = X^{\otimes 2}$ with the swap action, the
geometric fixed points satisfy
\[
(X^{\otimes 2})^{\Phi C_2} \simeq X
\]
and under this identification the canonical map
\[
(X^{\otimes 2})^{\Phi C_2} \longrightarrow (X^{\otimes 2})^{tC_2}
\]
coincides with the Tate diagonal
\[
 X \xlongrightarrow{\Delta} (X^{\otimes 2})^{tC_2}
\]
This description agrees with the original construction of the Tate diagonal as
the topological Singer $\epsilon$-map \cite[def.~5.10]{Singercons}. Hence the
Tate square specializes to

\[
\begin{tikzcd}
	{(X^{\otimes 2})^{C_2}} & {(X^{\otimes 2})^{hC_2}} \\
	X & {(X^{\otimes 2})^{tC_2}}
	\arrow[from=1-1, to=1-2]
	\arrow[from=1-1, to=2-1]
	\arrow["\lrcorner"{anchor=center, pos=0.125}, draw=none, from=1-1, to=2-2]
	\arrow["{\mathrm{can}}", from=1-2, to=2-2]
	\arrow["\Delta"', from=2-1, to=2-2]
\end{tikzcd}
\]
\begin{proof}
	The equivalence of the two statements is illustrated by the following diagram

	\[
	\begin{tikzcd}
		X \\
		& {(X^{\otimes 2})^{C_2}} & {(X^{\otimes 2})^{hC_2}} \\
		& X & {(X^{\otimes 2})^{tC_2}}
		\arrow[dashed, from=1-1, to=2-2]
		\arrow[dashed, from=1-1, to=2-3]
		\arrow[equals, from=1-1, to=3-2]
		\arrow[from=2-2, to=2-3]
		\arrow[from=2-2, to=3-2]
		\arrow["\lrcorner"{anchor=center, pos=0.125}, draw=none, from=2-2, to=3-3]
		\arrow["{\mathrm{can}}", from=2-3, to=3-3]
		\arrow["\Delta"', from=3-2, to=3-3]
	\end{tikzcd}
	\]
	By the universal property of the pullback, the choice of one of the dotted arrows determines the other.
\end{proof}

As we have seen throughout this section, spectra with metastable structure already behave in many ways like suspension spectra. We conclude with a theorem of Klein, rephrased in the language used here, which shows that in a relatively low-dimensional range this is not just an analogy.

\begin{thm}[Klein, see {\cite[thm.~A]{Kleinmoduli}}]
If $X$ is a finite spectrum with a metastable structure and satisfies $\dim X \leq 3\conn X - 1$, then $X$ is a suspension spectrum.
\end{thm}

\section{The metastable category}

In the previous section we discussed metastable structures on spectra. As good homotopy theorists, we now want to organize this structure into a category.

The goal of this section is to construct the \emph{metastable category} $P_2\mathcal{S}_*$, equip it with a monoidal structure, and construct symmetric monoidal colimit-preserving functors fitting into the diagram

\[\begin{tikzcd}
	{(\mathcal{S}_*, \wedge)} & {(P_2\mathcal{S}_*, \otimes)} & {(\mathcal{Sp}, \otimes)}
	\arrow["{\Sigma^\infty_2}", from=1-1, to=1-2]
	\arrow["{\Sigma^\infty}"', curve={height=18pt}, from=1-1, to=1-3]
	\arrow["{\Sigma^\infty_{}}", from=1-2, to=1-3]
\end{tikzcd}\]
The metastable category should be viewed as the next approximation to the category of pointed spaces after spectra.

In the first subsection, we construct the metastable category together with the two colimit-preserving functors from spaces and to spectra. We show that there are infinitely many such functors from spaces, but only one that sends a space to its suspension spectrum equipped with the spacelike metastable structure. At the end of the subsection, we also give alternative descriptions of the category in order to match the discussion of the first section.

In the second subsection, we endow the metastable category with a symmetric monoidal structure and upgrade the two functors to symmetric monoidal functors. Monoidality forces uniqueness of the functor from spaces, without requiring the spacelike condition. We conclude the subsection by defining the spheres in the metastable category, including the exotic $0$-spheres, computing their tensor products, and proving that the Picard group of $P_2\mathcal{S}_*$ is trivial.

In the third subsection, we recall Heuts' theory of $n$-excisive categories and prove that the metastable category is $2$-excisive. We show that the forgetful functor to spectra exhibits spectra as its stabilization. We also describe its costabilization, which coincides with the category of spectra equipped with a trivialization of the Tate diagonal.

In the fourth subsection, we use the category we have constructed to study some of its internal colimits. In particular, we give criteria for an object to split as a coproduct and to desuspend. We conclude by showing that every object with a counital comultiplication is already a suspension in the metastable category, thereby recovering a classical result of Berstein--Hilton.

\subsection{Construction}

In this subsection we construct the metastable category. The strategy is purely formal: we first review the notion of coalgebras for an endofunctor and establish their basic properties. We then use this formalism to define categories of lifts associated to a span of endofunctors. Finally, we specialize this construction to the span
\[
1 \xlongrightarrow{\Delta} (-^{\otimes 2})^{tC_2}
\xlongleftarrow{\mathrm{can}} (-^{\otimes 2})^{hC_2}
\]
and thereby obtain the metastable category $P_2\mathcal{S}_*$ together with colimit-preserving functors $$\mathcal{S}_* \longrightarrow P_2\mathcal{S}_* \longrightarrow \mathcal{Sp}$$
In order to define the metastable category, we begin by recalling the notion of a \emph{coalgebra for an endofunctor} \cite[sec. II.5]{NikolausScholze}. In general, this construction is formulated using \emph{lax equalizers} \cite[def. II.1.4]{NikolausScholze}. However, since we will not require the full formalism of lax equalizers, we give a direct definition of coalgebras. For simplicity, we work throughout in the category of spectra $\mathcal{Sp}$.

\begin{definition}
	Let $F$ be an endofunctor of $\mathcal{Sp}$. The category of \emph{spectral $F$-coalgebras} is defined as the following pullback
\[\begin{tikzcd}
	{\mathrm{coAlg}\, F} & {\mathrm{Fun}(\Delta^1, \mathcal{Sp})} \\
	{\mathcal{Sp}} & {\mathcal{Sp} \times \mathcal{Sp}}
	\arrow[from=1-1, to=1-2]
	\arrow[from=1-1, to=2-1]
	\arrow["\lrcorner"{anchor=center, pos=0.125}, draw=none, from=1-1, to=2-2]
	\arrow["{(ev_0, ev_1)}", from=1-2, to=2-2]
	\arrow["{(1, F)}"', from=2-1, to=2-2]
\end{tikzcd}\]
	Concretely, a spectral $F$-coalgebra consists of a spectrum $X$ together with a structure map $X \to FX$.
\end{definition}

The left vertical functor $\mathrm{coAlg}\, F \to \mathcal{Sp}$ sends a spectral $F$-coalgebra $X \to FX$ to its underlying spectrum $X$. For this reason, we refer to it as the \emph{forgetful functor}.

We next record the basic structural properties of this construction.

\begin{lem}
	\label{propcoalg}
Let $F$ be an accessible reduced endofunctor of $\mathcal{Sp}$. Then
\begin{enumerate}
\item $\mathrm{coAlg}\, F$ is presentable and pointed.
\item The forgetful functor $\mathrm{coAlg}\, F \to \mathcal{Sp}$ is conservative and creates colimits.
\item A natural transformation $F \to G$ induces a colimit-preserving functor $\mathrm{coAlg}\, F \to \mathrm{coAlg}\, G$.
\end{enumerate}
\end{lem}

\begin{proof}
	The first two statements are proved in \cite[prop. II.1.5(ii)(iv)]{NikolausScholze}. For the third statement, let $\eta : F \to G$ be a natural transformation. Then we obtain the following diagram:
	
\[\begin{tikzcd}
	{\mathrm{coAlg}\,F} & {\mathrm{Fun}(\Delta^1, \mathcal{Sp}) } \\
	{\mathcal{Sp}} & {\mathcal{Sp} \times \mathcal{Sp}} \\
	{\mathrm{Fun}(\Delta^1, \mathcal{Sp}) } & {\mathcal{Sp}}
	\arrow[from=1-1, to=1-2]
	\arrow[from=1-1, to=2-1]
	\arrow[from=1-2, to=2-2]
	\arrow["{ev_1}", curve={height=-30pt}, from=1-2, to=3-2]
	\arrow["{(1, F)}"', from=2-1, to=2-2]
	\arrow["\eta"', from=2-1, to=3-1]
	\arrow["{\pi_1}", from=2-2, to=3-2]
	\arrow["{ev_0}"', from=3-1, to=3-2]
\end{tikzcd}\]
	In particular, this determines the dashed arrow and hence the following diagram.

\[\begin{tikzcd}
	{\textrm{coAlg }F} & {\mathrm{Fun}(\Delta^1, \mathcal{Sp}) \times_{\mathcal{Sp}} \mathrm{Fun}(\Delta^1, \mathcal{Sp}) } & {\mathrm{Fun}(\Delta^1, \mathcal{Sp}) } \\
	{\mathcal{Sp}} && {\mathcal{Sp} \times \mathcal{Sp}}
	\arrow[dashed, from=1-1, to=1-2]
	\arrow[from=1-1, to=2-1]
	\arrow["\circ", from=1-2, to=1-3]
	\arrow["{(ev_0, ev_1)}", from=1-3, to=2-3]
	\arrow["{(1, G)}"', from=2-1, to=2-3]
\end{tikzcd}\]
	This yields a functor $\mathrm{coAlg}\, F \to \mathrm{coAlg}\, G$. To see that it preserves colimits, note that a functor into $\mathrm{coAlg}\, G$ preserves colimits if and only if its composite with the forgetful functor preserves colimits, since the forgetful functor creates colimits. 
	From the diagram above, this composite coincides with the forgetful functor $\mathrm{coAlg}\, F \to \mathcal{Sp}$, which preserves colimits.
\end{proof}

\begin{rem}
	The lemma above is conceptually straightforward. At the level of objects, colimits in $\mathrm{coAlg}\,F$ are computed by forming colimits in spectra and using the canonical colimit-comparison map. Explicitly, given a diagram of spectral $F$-coalgebras $(X_i \to F X_i)$, their colimit is
	\[
	\mathrm{colim}_i (X_i \to F X_i)
	\;\simeq\;
	(\mathrm{colim}_i X_i
	\longrightarrow
	\mathrm{colim}_i F X_i
	\longrightarrow
	F(\mathrm{colim}_i X_i))
	\]
	where the second map is the canonical colimit-comparison map.
	
	Similarly, a natural transformation $F \to G$ induces a functor on coalgebras which, at the level of objects, sends
	\[
	(X \to F X)
	\longmapsto
	(X \to F X \to G X)
	\]
	The compatibility with colimits follows from the commutativity of the evident comparison diagram.
\[\begin{tikzcd}
	{\mathrm{colim}_iX_i} & {\mathrm{colim}_iF(X_i)} & {F(\mathrm{colim}_iX_i)} \\
	& {\mathrm{colim}_iG(X_i)} & {G(\mathrm{colim}_iX_i)}
	\arrow[from=1-1, to=1-2]
	\arrow[from=1-2, to=1-3]
	\arrow[from=1-2, to=2-2]
	\arrow[from=1-3, to=2-3]
	\arrow[from=2-2, to=2-3]
\end{tikzcd}\]
	This matches the formal argument given above.
\end{rem}

The previous lemma shows that the construction of coalgebras extends to a functor $$\mathrm{coAlg}: \mathrm{End}^{acc}_*(\mathcal{Sp}) \longrightarrow \mathrm{Pr}^L_{* / \mathcal{Sp}}$$
We can also construct a functor in the opposite direction. An object of $\mathrm{Pr}^L_{* /\mathcal{Sp}}$ consists of a presentable category $C$ together with a left adjoint $L : C \to \mathcal{Sp}$. Let $R$ denote its right adjoint. The composite $L R$ is then an accessible reduced endofunctor of $\mathcal{Sp}$. In this way we obtain a functor $$K: \mathrm{Pr}^L_{*/ \mathcal{Sp}} \longrightarrow \mathrm{End}^{acc}_*(\mathcal{Sp})$$ 
This construction provides a left adjoint to the functor $\textrm{coAlg}$ defined above.

\begin{lem}
	There is an adjunction
\[\begin{tikzcd}
	{K: \mathrm{Pr}^L_{*/ \mathcal{Sp}}} & {\mathrm{End}^{acc}_*(\mathcal{Sp}): \mathrm{coAlg}}
	\arrow[""{name=0, anchor=center, inner sep=0}, shift left, from=1-1, to=1-2]
	\arrow[""{name=1, anchor=center, inner sep=0}, shift left, from=1-2, to=1-1]
	\arrow["\dashv"{anchor=center, rotate=-90}, draw=none, from=0, to=1]
\end{tikzcd}\]
\end{lem}

\begin{proof}
	Let $L : C \to \mathcal{Sp}$ be an object of $\mathrm{Pr}^L_{* /\mathcal{Sp}}$. A map $C \to \mathrm{coAlg}\,F$ compatible with the forgetful functor is equivalent, by the universal property of the pullback, to specifying the dashed arrow in the following diagram
\[\begin{tikzcd}
	C & {\mathrm{Fun}(\Delta^1, \mathcal{Sp}) } \\
	{\mathcal{Sp}} & {\mathcal{Sp} \times \mathcal{Sp}}
	\arrow["\eta", dashed, from=1-1, to=1-2]
	\arrow["L"', from=1-1, to=2-1]
	\arrow["{(ev_0, ev_1)}", from=1-2, to=2-2]
	\arrow["{(1, F)}"', from=2-1, to=2-2]
\end{tikzcd}\]
	That is, this amounts to giving a natural transformation $\eta : L \Rightarrow F L$. Since $L$ admits a right adjoint $R$, precomposition with $L$ and $R$ induces an adjunction
\[\begin{tikzcd}
	{R^*: \mathrm{Fun}(C, \mathcal{Sp})} & {\mathrm{End}(\mathcal{Sp}) : L^*}
	\arrow[""{name=0, anchor=center, inner sep=0}, shift left, from=1-1, to=1-2]
	\arrow[""{name=1, anchor=center, inner sep=0}, shift left, from=1-2, to=1-1]
	\arrow["\dashv"{anchor=center, rotate=-90}, draw=none, from=0, to=1]
\end{tikzcd}\]
	In particular, under this adjunction, the natural transformation $\eta$ corresponds to a natural transformation $\eta': K(C) \simeq LR \Rightarrow F$. This identifies $K$ as a left adjoint to $\textrm{coAlg}$.
\end{proof}

In particular, the previous lemma allows us to construct a colimit-preserving functor
\[
\mathcal{Sp} \longrightarrow \mathrm{coAlg}\,1
\]
given by the unit of the adjunction $1 \Rightarrow \textrm{coAlg} \circ K$ applied to the object $\mathcal{Sp} = \mathcal{Sp}$. Together with the functoriality of the construction $F \mapsto \mathrm{coAlg}\,F$, this will be used to define the category of lifts.

\begin{definition}
	Let $1 \to G \leftarrow F$ be a span of natural transformations between endofunctors of $\mathcal{Sp}$. The \emph{category of lifts} is defined as the following pullback

\[\begin{tikzcd}
	{\mathrm{Lift}(1 \to G \leftarrow F)} && {\mathrm{coAlg}\,F} \\
	{\mathcal{Sp}} & {\mathrm{coAlg}\, 1} & {\mathrm{coAlg}\,G}
	\arrow[from=1-1, to=1-3]
	\arrow[from=1-1, to=2-1]
	\arrow["\lrcorner"{anchor=center, pos=0.125}, draw=none, from=1-1, to=2-2]
	\arrow[from=1-3, to=2-3]
	\arrow[from=2-1, to=2-2]
	\arrow[from=2-2, to=2-3]
\end{tikzcd}\]
\end{definition}

As before, the category of lifts comes equipped with a forgetful functor to spectra. We have the following.

\begin{lem}
	\label{proplift}
	Let $1 \to G \leftarrow F$ be a span of natural transformations between endofunctors of $\mathcal{Sp}$. 
	If $F$ and $G$ are accessible and reduced, then $\mathrm{Lift}(1 \to G \leftarrow F)$ is presentable, pointed, and the forgetful functor is conservative and creates colimits.
\end{lem}

\begin{proof}
	The diagram defining the category of lifts refines to a diagram in $\mathrm{Pr}^L_{*, /\mathcal{Sp}}$ by the above discussion. The claim then follows from the fact that this category is complete and that the forgetful functor $\mathrm{Pr}^L_{*, /\mathcal{Sp}} \to \mathrm{Pr}^L_*$ preserves contractible limits and the embedding $\mathrm{Pr}^L_* \to \mathrm{Cat}_\infty$ preserves all limits.
\end{proof}

With these preparations in place, we can now define the metastable category.

\begin{definition}
	The \emph{metastable category} is the category $P_2\mathcal{S}_*$ defined by
	\[
	P_2\mathcal{S}_* 
	:= 
	\mathrm{Lift}\bigl(1 \xlongrightarrow{\Delta} (-^{\otimes 2})^{tC_2} 
	\xlongleftarrow{\mathrm{can}} (-^{\otimes 2})^{hC_2}\bigr)
	\]
\end{definition}

From the above discussion, we immediately obtain the following.

\begin{lem}
	The category $P_2\mathcal{S}_*$ is presentable, pointed and comes equipped with a functor $P_2\mathcal{S}_* \to \mathcal{Sp}$ which is conservative and creates colimits.
\end{lem}

\begin{proof}
	It suffices to prove that the span
\[
1 \xlongrightarrow{\Delta} (-^{\otimes 2})^{tC_2}
\xlongleftarrow{\mathrm{can}} (-^{\otimes 2})^{hC_2}
\]
consists of accessible reduced functors. 

The tensor square functor is accessible because the tensor product preserves colimits separately in each variable, hence the composite with the diagonal is accessible. Homotopy orbits and homotopy fixed points are accessible because they are adjoints between presentable categories, and adjoint functors between presentable categories are accessible \cite[prop. 5.4.7.7]{HTT}. The Tate construction is accessible because it is defined as the cofiber of a natural transformation between accessible functors, and cofibers of accessible functors are accessible.
\end{proof}

We will see in the following subsections that the forgetful functor exhibits the category of spectra as the stabilization of the metastable category. We now relate the metastable category to the category of pointed spaces.

\begin{lem}
	\label{functorfromspaces}
	There is a unique colimit-preserving functor $\Sigma^\infty_2: \mathcal{S}_* \to P_2\mathcal{S}_*$ such that the diagram
\[\begin{tikzcd}
	{\mathcal{S}_*} & {P_2\mathcal{S}_*} & {\mathcal{Sp}}
	\arrow["{\Sigma^\infty_2}", from=1-1, to=1-2]
	\arrow["{\Sigma^\infty}"', curve={height=18pt}, from=1-1, to=1-3]
	\arrow[from=1-2, to=1-3]
\end{tikzcd}\]
	commutes, and whose essential image consists of suspension spectra equipped with a spacelike metastable structure.
\end{lem}

\begin{proof}
	We must produce the dashed arrow in the diagram
\[
\begin{tikzcd}
	{\mathcal{S}_*} & {\mathrm{coAlg}\,(-^{\otimes 2})^{hC_2}} \\
	{\mathcal{Sp}} & {\mathrm{coAlg}\,(-^{\otimes 2})^{tC_2}}
	\arrow[dashed, from=1-1, to=1-2]
	\arrow["{\Sigma^\infty}"', from=1-1, to=2-1]
	\arrow["{\mathrm{can}_*}", from=1-2, to=2-2]
	\arrow["\Delta"', from=2-1, to=2-2]
\end{tikzcd}
\]
	By the adjunction $K \dashv \mathrm{coAlg}$, this is equivalent to specifying a natural transformation
	\[
	\Sigma^\infty\Omega^\infty \Rightarrow (-^{\otimes 2})^{hC_2}
	\]
	compatible with the Tate diagonal.

	The adjunction $\Sigma^\infty \dashv \Omega^\infty$ induces an adjunction
	$(\Omega^\infty)^* \dashv (\Sigma^\infty)^*$ on functor categories by precomposition. 
	Hence the natural transformation above corresponds to a natural transformation
	\[
	\Sigma^\infty \Rightarrow (\Sigma^\infty(-)^{\otimes 2})^{hC_2}
	\]
	This, in turn, corresponds to a $C_2$-equivariant natural transformation
	\[
	\Sigma^\infty \Rightarrow \Sigma^\infty(-)^{\otimes 2}
	\]
	which we take to be the one induced by the space-level diagonal
	\[
	1 \Rightarrow (-)^{\wedge 2}
	\]
	This natural transformation is compatible with the Tate diagonal by \cite[lem.~IV.1.3]{NikolausScholze}.
\end{proof}

\begin{rem}
	\label{remfunctorfromspaces}
	Alternatively, a colimit-preserving functor $\mathcal{S}_* \to P_2\mathcal{S}_*$ is determined by the image of $S^0$. Since we require its composite with the forgetful functor to coincide with the stabilization functor, $S^0$ must be sent to a metastable structure on the sphere spectrum $\mathbb{S}$. We have seen that $\mathbb{S}$ admits infinitely many metastable structures (see example \ref{metstructsn}), and hence there are infinitely many such colimit-preserving functors. However, only one of these structures is spacelike (see lemmas \ref{spacelikes0} and \ref{exotics0}), which yields the uniqueness stated above. We return to this point in the next subsection.
\end{rem}

We conclude this subsection by recasting the earlier characterizations of metastable structures in categorical terms. More precisely, we reinterpret the descriptions in terms of nullhomotopies (see lemma~\ref{metnull}) and sections (see lemma~\ref{metsect}) as instances of the construction of lifts.

\begin{definition}
	The following constructions arise as special instances of the category of lifts. 
	\begin{enumerate}
		\item Let $P \to 1$ be a natural transformation from an endofunctor $P$ to the identity. The category of \emph{sections} of $P$ is defined by
		\[
		\mathrm{Sect}(P \to 1) := \mathrm{Lift}(1 = 1 \leftarrow P)
		\]
		\item Let $1 \to Q$ be a natural transformation from the identity to an endofunctor $Q$. The category of \emph{nullhomotopies} of $Q$ is defined by
		\[
		\mathrm{Null}(1 \to Q) := \mathrm{Lift}(1 \to Q \leftarrow 0)
		\]
	\end{enumerate}
\end{definition}

\begin{lem}
	Let $1 \to G \leftarrow F$ be a span of natural transformations between endofunctors of $\mathcal{Sp}$. Let $P$ denote the pullback $1 \times_G F$ and $Q$ the cofiber of the natural transformation $F \to G$. Then there are equivalences of categories
	\begin{enumerate}
		\item $\mathrm{Sect}(P \to 1)$
		\item $\mathrm{Lift}(1 \to G \leftarrow F)$
		\item $\mathrm{Null}(1 \to Q)$
	\end{enumerate}
\end{lem}

\begin{proof}
	Consider the diagram
\[\begin{tikzcd}
	& {\mathrm{coAlg}\,P} & {\mathrm{coAlg}\, F} & {\mathrm{coAlg}\, 0} \\
	{\mathcal{Sp}} & {\mathrm{coAlg}\, 1} & {\mathrm{coAlg}\, G} & {\mathrm{coAlg}\,Q}
	\arrow[from=1-2, to=1-3]
	\arrow[from=1-2, to=2-2]
	\arrow["\lrcorner"{anchor=center, pos=0.125}, draw=none, from=1-2, to=2-3]
	\arrow[from=1-3, to=1-4]
	\arrow[from=1-3, to=2-3]
	\arrow["\lrcorner"{anchor=center, pos=0.125}, draw=none, from=1-3, to=2-4]
	\arrow[from=1-4, to=2-4]
	\arrow[from=2-1, to=2-2]
	\arrow[from=2-2, to=2-3]
	\arrow[from=2-3, to=2-4]
\end{tikzcd}\]
	The squares are pullbacks because $\mathrm{coAlg}$ preserves limits of endofunctors, being a right adjoint. The pullback of the leftmost span is, by definition, $\mathrm{Sect}(P \to 1)$, while the pullback of the left and center spans is $\mathrm{Lift}(1 \to G \leftarrow F)$, and the pullback of the entire diagram is $\mathrm{Null}(1 \to Q)$. Since these pullbacks agree, the three categories are equivalent. 
\end{proof}

Specializing the previous lemma to the span
\[
1 \xlongrightarrow{\Delta} (-^{\otimes 2})^{tC_2}
\xlongleftarrow{\mathrm{can}} (-^{\otimes 2})^{hC_2}
\]
and recalling that the corresponding pullback is $(-^{\otimes 2})^{C_2}$ and the cofiber is $\Sigma (-^{\otimes 2})_{hC_2}$, we obtain the following equivalent descriptions of the metastable category.

\begin{lem}
	\label{p2equiv}
	The metastable category can equivalently be described as
	\begin{enumerate}
		\item $\mathrm{Sect}((-^{\otimes 2})^{C_2} \longrightarrow 1)$
		\item $\mathrm{Lift}(1 \xlongrightarrow{\Delta} (-^{\otimes 2})^{tC_2} \xlongleftarrow{\mathrm{can}} (-^{\otimes 2})^{hC_2})$
		\item $\mathrm{Null}(1 \longrightarrow \Sigma (-^{\otimes 2})_{hC_2})$
	\end{enumerate}
\end{lem}

\subsection{Monoidality}

In this subsection we endow the metastable category with a symmetric monoidal structure and upgrade the functors
\[
(\mathcal{S}_*, \wedge) \longrightarrow (P_2\mathcal{S}_*, \otimes) \longrightarrow (\mathcal{Sp}, \otimes)
\]
to symmetric monoidal functors.

To achieve this, we show that whenever an endofunctor - respectively a span of endofunctors - is lax symmetric monoidal, the associated category of coalgebras - respectively the category of lifts - inherits a natural symmetric monoidal structure. We then apply this general construction to the span given by the Tate diagonal and the canonical comparison map, which we verify to be lax symmetric monoidal.

We then introduce the spheres in $P_2\mathcal{S}_*$, including the exotic $0$-spheres, compute their tensor products, and deduce that the Picard group of the metastable category is trivial.

In order to do this, recall that a symmetric monoidal category is given by a cocartesian fibration of operads
\[
C^\otimes \longrightarrow \mathrm{N}(\mathcal{Fin}_*)
\]
Whenever $C$ has a symmetric monoidal structure and $S$ is any simplicial set, the functor category $\mathrm{Fun}(S, C)$ inherits a symmetric monoidal structure defined by
\[
\mathrm{Fun}(S, C)^\otimes
:=
\mathrm{Fun}(S, C^\otimes)
\times_{\mathrm{Fun}(S, \mathrm{N}(\mathcal{Fin}_*))}
\mathrm{N}(\mathcal{Fin}_*)
\]
Moreover, denote by $\mathrm{Fun}(\Delta^1, C^\otimes)_{\mathrm{id}}$ the full subcategory of $\mathrm{Fun}(\Delta^1, C^\otimes)$ consisting of those morphisms which lie over identities in $\mathrm{Fun}(\Delta^1, \mathrm{N}(\mathcal{Fin}_*))$. With this notation in place, we recall the following construction from \cite[cons. IV.2.1(i)(ii)]{NikolausScholze}.

\begin{lem}
	Let $F$ be a lax symmetric monoidal endofunctor of $\mathcal{Sp}$. Then $\mathrm{coAlg}\ F$ admits a natural symmetric monoidal structure defined by the pullback
\[\begin{tikzcd}
	{\mathrm{coAlg}^\otimes F} & {\mathrm{Fun}(\Delta^1, \mathcal{Sp}^{\otimes})_{\mathrm{id}}} \\
	{\mathcal{Sp}^{\otimes}} & {\mathcal{Sp}^{\otimes} \times \mathcal{Sp}^{\otimes}}
	\arrow[from=1-1, to=1-2]
	\arrow[from=1-1, to=2-1]
	\arrow["\lrcorner"{anchor=center, pos=0.125}, draw=none, from=1-1, to=2-2]
	\arrow["{(ev_0, ev_1)}", from=1-2, to=2-2]
	\arrow["{(1, F)}"', from=2-1, to=2-2]
\end{tikzcd}\]
	and the forgetful functor is symmetric monoidal. Moreover, any lax symmetric monoidal natural transformation $F \to G$ induces a symmetric monoidal functor $\mathrm{coAlg}^\otimes F \to \mathrm{coAlg}^\otimes G$.
\end{lem}

We will henceforth drop the notation $\mathrm{coAlg}^\otimes F$ and simply write $\mathrm{coAlg}\,F$ for the underlying symmetric monoidal category.

\begin{rem}
	\label{objecttensor}
At the level of objects, the symmetric monoidal structure on $\mathrm{coAlg}\,F$ is given by
\[
(X_1 \to F X_1) \otimes (X_2 \to F X_2)
\simeq
\bigl(
X_1 \otimes X_2
\longrightarrow
F X_1 \otimes F X_2
\longrightarrow
F(X_1 \otimes X_2)
\bigr)
\]
where the second map is induced by the lax symmetric monoidal structure on $F$.
\end{rem}

The previous lemma shows that the construction of coalgebras extends to a functor $$\mathrm{coAlg}: \mathrm{End}^{acc}_{lax, *}(\mathcal{Sp}) \to \mathrm{CAlg}(\mathrm{Pr}^L_*)_{/\mathcal{Sp}}$$

We now construct a functor in the opposite direction. Let $L \colon C \to \mathcal{Sp}$ be a symmetric monoidal left adjoint. Then its right adjoint $R$ is lax symmetric monoidal, and the composite $LR$ is an accessible lax symmetric monoidal endofunctor of $\mathcal{Sp}$. In this way we obtain a functor $$K: \mathrm{CAlg}(\mathrm{Pr}^L_*)_{/\mathcal{Sp}} \to \mathrm{End}^{acc}_{lax, *}(\mathcal{Sp})$$

As in the previous subsection, this construction provides a left adjoint to the functor $\mathrm{coAlg}$ defined above.

\begin{lem}
	There is an adjunction
\[\begin{tikzcd}
	{K: \mathrm{CAlg}(\mathrm{Pr}^L_*)_{/ \mathcal{Sp}}} & {\mathrm{End}^{acc}_{lax, *}(\mathcal{Sp}): \mathrm{coAlg}}
	\arrow[""{name=0, anchor=center, inner sep=0}, shift left, from=1-1, to=1-2]
	\arrow[""{name=1, anchor=center, inner sep=0}, shift left, from=1-2, to=1-1]
	\arrow["\dashv"{anchor=center, rotate=-90}, draw=none, from=0, to=1]
\end{tikzcd}\]
\end{lem}

This construction also gives a symmetric monoidal structure on categories of lifts.

\begin{lem}
Let $1 \to G \leftarrow F$ be a span of natural transformations between accessible, lax symmetric monoidal, reduced endofunctors of $\mathcal{Sp}$. Then $\mathrm{Lift}(1 \to G \leftarrow F)$ admits a natural symmetric monoidal structure, and the forgetful functor is symmetric monoidal.
\end{lem}

\begin{proof}
The diagram defining the category of lifts refines to a diagram in $\mathrm{CAlg}(\mathrm{Pr}^L_*)_{/\mathcal{Sp}}$ by the above discussion. The claim then follows from the fact that this category is complete and that the forgetful functor $\mathrm{CAlg}(\mathrm{Pr}^L_*)_{/\mathcal{Sp}} \to \mathrm{CAlg}(\mathrm{Pr}^L_*)$ preserves contractible limits and the forgetful functor $\mathrm{CAlg}(\mathrm{Pr}^L_*) \to \mathrm{Pr}^L_*$ preserves all limits.
\end{proof}

From the above discussion, we immediately obtain the following.

\begin{lem}
	\label{forgetmonoidal}
	The category $P_2\mathcal{S}_*$ admits a symmetric monoidal structure, and the forgetful functor $P_2\mathcal{S}_* \to \mathcal{Sp}$ is symmetric monoidal.
\end{lem}

\begin{proof}
	It suffices to prove that the span
\[
1 \xlongrightarrow{\Delta} (-^{\otimes 2})^{tC_2}
\xlongleftarrow{\mathrm{can}} (-^{\otimes 2})^{hC_2}
\]
consists of lax symmetric monoidal natural transformations.

The homotopy fixed points functor $-^{hC_2}$ and the Tate construction $-^{tC_2}$ are lax symmetric monoidal, and the canonical natural transformation $\mathrm{can}$ is lax symmetric monoidal by \cite[thm. I.3.1]{NikolausScholze}. This remains true after composition with the symmetric monoidal functor $-^{\otimes 2}$. The Tate diagonal $\Delta$ is lax symmetric monoidal by \cite[prop. III.3.1]{NikolausScholze}.
\end{proof}

We now also make the functor from pointed spaces symmetric monoidal.

\begin{lem}
	\label{monoidalfromspaces}
There exists a unique symmetric monoidal colimit-preserving functor $\mathcal{S}_* \to P_2\mathcal{S}_*$. Its underlying functor coincides with the functor $\Sigma^\infty_2$ of lemma~\ref{functorfromspaces}.
\end{lem}

\begin{proof}
Existence and uniqueness follow from the fact that the category of pointed spaces $\mathcal{S}_*$ is the unit object of $\mathrm{Pr}^L_*$ with respect to the Lurie tensor product. Consequently, for any pointed presentable symmetric monoidal category, there exists a unique symmetric monoidal colimit-preserving functor out of $\mathcal{S}_*$.

It remains to identify the underlying functor. By the adjunction $K \dashv \mathrm{coAlg}$, this amounts to verifying that the natural transformation
\[
\Sigma^\infty \Rightarrow (\Sigma^\infty(-)^{\otimes 2})^{hC_2}
\]
chosen in lemma~\ref{functorfromspaces} is lax symmetric monoidal. This is established in \cite[lem.~IV.1.3]{NikolausScholze}, and hence the resulting symmetric monoidal functor agrees with the one constructed in lemma~\ref{functorfromspaces}.
\end{proof}

\begin{rem}
	\label{remsymfunctorfromspaces}
As observed in remark~\ref{remfunctorfromspaces}, a colimit-preserving functor $\mathcal{S}_* \to P_2\mathcal{S}_*$ factoring the stabilization functor is determined by the choice of a metastable structure on the sphere spectrum $\mathbb{S}$. In order to obtain uniqueness there, one had to restrict to those functors whose essential image consists of suspension spectra equipped with the spacelike metastable structure, thereby singling out the unique spacelike structure on $\mathbb{S}$.

In the symmetric monoidal setting, this additional restriction becomes unnecessary. A symmetric monoidal functor must preserve the monoidal unit, and hence send $S^0$ to the unit object of $P_2\mathcal{S}_*$. The latter is precisely $\mathbb{S}$ endowed with its unique spacelike metastable structure, as will be established below. Thus symmetric monoidality rigidifies the choice automatically, and uniqueness follows formally from preservation of the unit.
\end{rem}

We conclude this subsection by introducing the spheres in the metastable category, computing their tensor products, and using this to determine the Picard group of $P_2\mathcal{S}_*$.

We begin with the ordinary spheres.

\begin{definition}
	\label{spheresp2}
Let $n \ge 0$. The \emph{$n$-sphere} $S^n$ in $P_2\mathcal{S}_*$ is the image of the classical $n$-sphere $S^n$ in $\mathcal{S}_*$ under the canonical symmetric monoidal functor $\Sigma^\infty_2\colon \mathcal{S}_* \to P_2\mathcal{S}_*$. Equivalently, it is the unique object of $P_2\mathcal{S}_*$ whose underlying spectrum is $\mathbb{S}^n$ and whose metastable structure is the spacelike one.
\end{definition}

We next introduce the exotic $0$-spheres. Recall that the sphere spectrum $\mathbb{S}$ admits infinitely many metastable structures, classified by the necessarily odd degree $k$ of the induced diagonal map, by lemma \ref{exotics0}. We will only consider $k>0$. Indeed, if $L$ is a metastable structure on a spectrum $X$, then the automorphism $-1$ of $X$ induces a new metastable structure $L'$ given by the composite
$$
X \xlongrightarrow{-1} X \xlongrightarrow{L} (X^{\otimes 2})^{hC_2} \xlongrightarrow{((-1)\otimes(-1))^{hC_2}} (X^{\otimes 2})^{hC_2}
$$
obtained by transporting $L$ along $-1$. It is clear that $(X,L)$ and $(X,L')$ define equivalent objects of $P_2\mathcal{S}_*$. Moreover, since $(-1)^{\otimes 2}=1$, the last map is the identity, so $L'=-L$. In particular, the metastable structures on $\mathbb{S}$ with diagonal $k$ and $-k$ define equivalent objects. Thus it suffices to consider $k>0$.

\begin{definition}
	\label{exoticspheresp2}
Let $k$ be an odd natural number. The \emph{$k$-exotic $0$-sphere} $S^{0, k}$ is the unique object of $P_2\mathcal{S}_*$ whose underlying spectrum is $\mathbb{S}$ and whose induced diagonal $\mathbb{S} \to (\mathbb{S}^{\otimes 2})^{hC_2}
\to \mathbb{S}^{\otimes 2} \simeq \mathbb{S}$ has degree $k$.
\end{definition}

\begin{rem}
For the moment, it is not obvious that for $j \neq k$ the exotic spheres $S^{0,j}$ and $S^{0,k}$ are distinct objects of the metastable category. We will use this fact in what follows, and prove it in the next section.
\end{rem}

The spacelike structure on $\mathbb{S}$ corresponds to $k=1$ (see lemma~\ref{spacelikes0}), so $S^0 \simeq S^{0,1}$. For uniformity of notation, we will sometimes write $S^{n,0} := S^n$ for $n>0$, reflecting the fact that for $n>0$ the induced diagonal $\mathbb{S}^n \to \mathbb{S}^{2n}$ is null.

With this convention in place, the tensor product of spheres is given by the following formula.

\begin{lem}
For all $m,n \ge 0$ and all odd or zero $j,k \geq 0$, there is an equivalence
\[
S^{m,j} \otimes S^{n,k} \simeq S^{m+n, jk}
\]
\end{lem}

\begin{proof}
	Consider the following diagram:
\[\begin{tikzcd}
	{\mathbb{S}^{m+n} \simeq \mathbb{S}^m \otimes \mathbb{S}^n} & {((\mathbb{S}^m)^{\otimes 2})^{hC_2} \otimes ((\mathbb{S}^n)^{\otimes 2})^{hC_2} } & {((\mathbb{S}^m \otimes \mathbb{S}^n)^{\otimes 2})^{hC_2} } \\
	& {(\mathbb{S}^m)^{\otimes 2} \otimes (\mathbb{S}^n)^{\otimes 2}} & {(\mathbb{S}^m \otimes \mathbb{S}^n)^{\otimes 2}}
	\arrow[from=1-1, to=1-2]
	\arrow["jk"', from=1-1, to=2-2]
	\arrow[from=1-2, to=1-3]
	\arrow[from=1-2, to=2-2]
	\arrow[from=1-3, to=2-3]
	\arrow["{1 \otimes \tau \otimes 1}"', from=2-2, to=2-3]
	\arrow["\simeq", draw=none, from=2-2, to=2-3]
\end{tikzcd}\]
	The top row is the tensor product in $\mathrm{coAlg}$, computed as described in remark~\ref{objecttensor} The vertical maps forget the $C_2$-action. The bottom horizontal map is the canonical identification induced by the swap map $\tau$.
	
	Tracing $\mathbb{S}^{m+n} \simeq \mathbb{S}^m \otimes \mathbb{S}^n$ around the left-hand square shows that the induced diagonal
	\[
	\mathbb{S}^{m+n} \to (\mathbb{S}^{m+n})^{\otimes 2}
	\]
	is the tensor product of the diagonals of $\mathbb{S}^m$ and $\mathbb{S}^n$, up to the canonical permutation of factors. Since the degrees of the diagonals of $S^{m,j}$ and $S^{n,k}$ are $j$ and $k$, respectively, the resulting degree is $jk$ up to the degree $(-1)^{mn}$ of the permutation. This sign is irrelevant for $m+n>0$ since $jk = 0$, and equals $1$ when $m=n=0$, so the degree is $jk$.
	
	Thus $S^{m,j} \otimes S^{n,k}$ has underlying spectrum $\mathbb{S}^{m+n}$ and induced diagonal of degree $jk$, and hence identifies with the sphere $S^{m+n, jk}$ by definition.
\end{proof}

\begin{rem}
If $m+n>0$, the result is automatic: the underlying spectrum of $S^{m,j} \otimes S^{n,k}$ is $\mathbb{S}^{m+n}$, which admits a unique metastable structure. Thus the only nontrivial case in the lemma is $m=n=0$.
\end{rem}

We now make precise remark~\ref{remsymfunctorfromspaces} by identifying $S^0$ (equivalently $S^{0,1}$) as the monoidal unit of $P_2\mathcal{S}_*$.

\begin{lem}
$S^0$ is the monoidal unit of $P_2\mathcal{S}_*$.
\end{lem}

\begin{proof}
Let $\mathbb{1}$ denote the monoidal unit of $P_2\mathcal{S}_*$. Since the forgetful functor $P_2\mathcal{S}_* \to \mathcal{Sp}$ is symmetric monoidal, the underlying spectrum of $\mathbb{1}$ is $\mathbb{S}$. Hence $\mathbb{1} \simeq S^{0,k}$ for some odd natural number $k$. By unitality and the previous lemma
\[
S^{0,j} \simeq S^{0,j} \otimes S^{0,k} \simeq S^{0, jk}
\]
Thus $jk=j$ for all odd $j$, and therefore $k=1$. Hence $\mathbb{1} \simeq S^{0,1} = S^0$.
\end{proof}

We conclude by computing the Picard group of $P_2\mathcal{S}_*$, i.e. the group of equivalence classes of invertible objects under $\otimes$.

\begin{lem}
	\label{pic}
The Picard group of $P_2\mathcal{S}_*$ is trivial. 
\end{lem}

\begin{proof}
Let $X$ be an invertible object of $P_2\mathcal{S}_*$. Since the forgetful functor
$P_2\mathcal{S}_* \to \mathcal{Sp}$ is symmetric monoidal, the underlying spectrum of $X$ is invertible in $\mathcal{Sp}$, hence equivalent to $\mathbb{S}^n$ for some $n \in \mathbb{Z}$. As $\mathbb{S}^n$ admits a metastable structure only for $n \ge 0$, we cannot have $n>0$, since then $X$ would admit no inverse. Thus $n=0$ and $X \simeq S^{0,k}$ for some odd natural number $k$. If $S^{0,j}$ is an inverse for $S^{0, k}$, then by the tensor product formula,
\[
S^{0,1} \simeq \mathbb{1} \simeq S^{0,j} \otimes S^{0,k} \simeq S^{0,jk}
\]
Thus $jk=1$, and hence $j=k=1$. Therefore the only invertible object is $S^{0,1}$.
\end{proof}

\subsection{Polynomial approximations}

In this subsection we recall the notion of an $m$-excisive category and the construction of the universal $m$-excisive approximation $P_m$. We then prove that the metastable category is $2$-excisive and that its stabilization is equivalent to spectra. In particular, we identify the forgetful functor to spectra with the stabilization functor. Finally, we describe its costabilization as the category of spectra equipped with a trivialization of the Tate diagonal.

We start with the basic definitions.

An $(m+1)$-cube in a category $C$ is a functor $P(m+1) \to C$ where $P(m+1)$ denotes the poset of subsets of $\{1,\dots,m+1\}$. A cube is \emph{cartesian} if it is a limit diagram, and \emph{cocartesian} if it is a colimit diagram. It is \emph{strongly cocartesian} if every restriction to a $2$-dimensional face is cocartesian. It is \emph{reduced} if it is the terminal object on every singleton. A \emph{punctured} $(m+1)$-cube is a functor $P_0(m+1) \to C$ where $P_0(m+1)$ denotes the poset of \textit{non-empty} subsets of $\{1,\dots,m+1\}$. Finally, a functor is called \emph{$m$-excisive} if it sends strongly cocartesian $(m+1)$-cubes to cartesian cubes.

With this terminology in place, we adopt the following definition \cite[cor.~2.18]{Gijsapprox}.

\begin{definition}
A pointed compactly generated category $C$ is said to be  
\emph{$m$-excisive} if it satisfies the following two properties:
\begin{enumerate}
  \item The identity functor is $m$-excisive.
  \item Every reduced strongly cocartesian punctured $(m+1)$-cube in $C$
  extends to an $(m+1)$-cube which is both cartesian and strongly cocartesian.
\end{enumerate}
We write $\mathrm{Cat}^\omega_{*, \le m} \subseteq \mathrm{Cat}_*^{\omega}$
for the full subcategory spanned by the $m$-excisive categories.
\end{definition}

We next recall the notion of $m$-excisive equivalence
\cite[def.~1.2, 1.3]{Gijsapprox}.

\begin{definition}
A left adjoint $F: C \to D$ is called an \emph{$m$-excisive equivalence}
if the unit and counit of the adjunction are $P_m$-equivalences.
An $m$-excisive equivalence $F: C \to D$ is called an
\emph{$m$-excisive approximation} if $D$ is $m$-excisive.
\end{definition}

The following theorem of Heuts
\cite[thm.~1.7]{Gijsapprox}
establishes the existence and basic properties of such approximations.

\begin{thm}
The inclusion $\mathrm{Cat}^\omega_{*, \le m} \hookrightarrow \mathrm{Cat}^\omega_*$
admits a left adjoint $P_m$ with the following properties:
\begin{enumerate}
\item For every pointed compactly generated category $C$, the unit
$C \to P_m C$ is an $m$-excisive approximation of $C$.
\item The functor $P_m$ preserves finite limits.
\end{enumerate}
\end{thm}

As the notation suggests, the metastable category $P_2\mathcal{S}_*$ coincides with $P_2$ applied to the category of pointed spaces $\mathcal{S}_*$ \cite[ex.~3.1]{Gijsapprox}. More generally, Heuts identifies $P_m\mathcal{S}_*$ with the category of $m$-truncated spectral Tate coalgebras. We will not establish this identification here. Instead, we give a direct argument that the metastable category is $2$-excisive and that its stabilization is equivalent to spectra.

To this end, we briefly recall the construction of $P_m$. Let $C$ be a pointed compactly generated category. Denote by $T_m C$ the full subcategory of $\mathrm{Fun}(P_0(m+1), C)$ spanned by the special punctured strongly cocartesian $(m+1)$-cubes. We will discuss the structure of $T_m C$ in more detail later.

There is a functor $L_m : C \to T_m C$ which sends an object $X$ and a subset $S \subseteq \{1,\dots,m+1\}$ to the object $X * S$, where $*$ denotes the join. Equivalently, this may be identified with $\bigvee_{|S|-1}\Sigma X$. The functor $L_m$ admits a right adjoint given by taking the limit of the cube. If $C$ is stable, this adjunction is an equivalence, since strongly cocartesian cubes are then cartesian (see \cite[lem. 2.12 or prop. 2.16 $n = 1$]{Gijsapprox}).

The $m$-excisive approximation of $C$ is defined as the sequential colimit
\[
P_m C \simeq \mathrm{colim}\bigl( C \xlongrightarrow{L_m} T_m C \xlongrightarrow{L_m} T_m T_m C \longrightarrow \cdots \bigr)
\]
This construction is completely analogous to the classical Goodwillie calculus of functors. In fact, it can be used to define the Goodwillie approximations. Let $F : C \to D$ be a reduced functor between pointed compactly generated categories. We define $T_mF$ as the composite
\[
C \xlongrightarrow{L_m} T_mC \xlongrightarrow{F_*} T_mD \xlongrightarrow{\mathrm{lim}} D
\]
By the universal property of limits, there is a natural transformation $F \to T_mF$, and we define the $m^{\mathrm{th}}$ Goodwillie approximation by
\[
P_mF := \mathrm{colim}(F \longrightarrow T_mF \longrightarrow T_mT_mF \longrightarrow \cdots)
\]
With this in place, we can prove the following. 

\begin{lem}
Let $F$ be a reduced endofunctor of $\mathcal{Sp}$. Then
\[
P_m \mathrm{coAlg}\, F \simeq \mathrm{coAlg}\, P_m F
\]
\end{lem}

\begin{proof}
It suffices to prove that
\[
T_m \mathrm{coAlg}\, F \simeq \mathrm{coAlg}\, T_m F
\]
since $P_m$ is defined as the filtered colimit of the iterates of $T_m$, and $\mathrm{coAlg}$ preserves filtered colimits of endofunctors, being defined by a pullback construction.

We now prove the statement for $T_m$. A cube in $\mathrm{coAlg}\, F$ is equivalently a coalgebra for $F$ in the category of cubes. More precisely, for any category $C$ and endofunctor $F$ of $C$, one may define the category $\mathrm{coAlg}_C F$ of $F$-coalgebras in $C$ exactly as in the spectral case. With this notation, we obtain a natural identification
\[
T_m \mathrm{coAlg}\, F \simeq \mathrm{coAlg}_{T_m\mathcal{Sp}} F_*
\]
where $F_* : T_m\mathcal{Sp} \to T_m\mathcal{Sp}$ is given by postcomposition with $F$.

Consider the adjunction
\[\begin{tikzcd} {L_m:\mathcal{Sp}} & {T_m\mathcal{Sp}: \mathrm{lim}} \arrow[""{name=0, anchor=center, inner sep=0}, shift left, from=1-1, to=1-2] \arrow[""{name=1, anchor=center, inner sep=0}, shift left, from=1-2, to=1-1] \arrow["\dashv"{anchor=center, rotate=-90}, draw=none, from=0, to=1] \end{tikzcd}\]
Since $\mathcal{Sp}$ is stable, strongly cocartesian cubes are cartesian, and hence this adjunction is an equivalence. Using this equivalence, the endofunctor $F_* : T_m\mathcal{Sp} \to T_m\mathcal{Sp}$ corresponds to the composite
\[
\mathcal{Sp}
\xlongrightarrow{L_m}
T_m\mathcal{Sp}
\xlongrightarrow{F_*}
T_m\mathcal{Sp}
\xlongrightarrow{\mathrm{lim}}
\mathcal{Sp}
\]
which is precisely $T_mF$ by definition. Therefore
\[
\mathrm{coAlg}_{T_m\mathcal{Sp}} F_* \simeq \mathrm{coAlg}\, T_mF
\]
and the claim follows.
\end{proof}

\begin{lem}
Let $1 \to G \leftarrow F$ be a span of natural transformations between reduced endofunctors of $\mathcal{Sp}$. Then
\[
P_m\mathrm{Lift}(1 \to G \leftarrow F) \simeq \mathrm{Lift}(1 \to P_mG \leftarrow P_mF)
\]
In particular, if $F$ and $G$ are $m$-excisive functors, then the category of lifts is $m$-excisive.
\end{lem}

\begin{proof}
The category of lifts is defined by a pullback diagram in $\mathrm{Cat}_*^\omega$. Since $P_m$ preserves finite limits, applying it to this diagram yields
\[
P_m\mathrm{Lift}(1 \to G \leftarrow F)
\simeq
\mathrm{Lift}(1 \to P_mG \leftarrow P_mF)
\]
where we use the previous lemma to identify the coalgebra terms and $P_m1 \simeq 1$.
\end{proof}

We now apply this to the metastable category. Recall the fiber sequence of endofunctors
\[
(-^{\otimes 2})_{hC_2}
\longrightarrow
(-^{\otimes 2})^{hC_2}
\longrightarrow
(-^{\otimes 2})^{tC_2}
\]
This coincides with the Goodwillie filtration of $(-^{\otimes 2})^{hC_2}$. Indeed, $(-^{\otimes 2})^{tC_2}$ is linear \cite[prop.~III.1.1]{NikolausScholze}, while $(-^{\otimes 2})_{hC_2}$ is $2$-homogeneous. From this it follows immediately that

\begin{lem}
	\label{forgetstabilization}
The category $P_2\mathcal{S}_*$ is $2$-excisive, and the forgetful functor exhibits $\mathcal{Sp}$ as its stabilization.
\end{lem}

\begin{proof}
The first statement follows from the previous lemma, since the functors defining the metastable category are $2$-excisive.

For the second statement, observe that
\[
P_1 P_2\mathcal{S}_*
\simeq
\mathrm{Lift}(1 \to P_1(-^{\otimes 2})^{tC_2} \leftarrow P_1(-^{\otimes 2})^{hC_2})
\simeq
\mathrm{Lift}(1 \to (-^{\otimes 2})^{tC_2} \xlongleftarrow{\sim} (-^{\otimes 2})^{tC_2})
\]
In general, for any endofunctor $F$, the forgetful functor
\[
\mathrm{Lift}(1 \longrightarrow F \xlongleftarrow{\sim} F)
\longrightarrow
\mathcal{Sp}
\]
is an equivalence, as it is obtained by pullback along an equivalence. This identifies $P_1 P_2\mathcal{S}_*$ with $\mathcal{Sp}$, and the claim follows.
\end{proof}

We conclude this subsection by computing the costabilization of the metastable category. Recall that for any category $C$, the costabilization \cite[def. 3.1.2]{Yuqing} is defined by
\[
\mathrm{coSp}(C) := \mathrm{Sp}(C^{\mathrm{op}})^{\mathrm{op}}
\]
and can be computed explicitly \cite[prop. 3.1.9]{Yuqing} as
\[
\mathrm{coSp}(C) \simeq \mathrm{lim}(\cdots \to C \xlongrightarrow{\Sigma} C \xlongrightarrow{\Sigma} C)
\]
We now need a dual version of Goodwillie calculus. This theory was developed by McCarthy \cite{dualcalc}, but we do not require its full strength here. We define the first dual approximation of a reduced functor $F : C \to D$ to be
\[
P^1F(X) := \mathrm{lim}_n \Sigma^n F(\Omega^n X)
\]
With this notation in place, we are now ready to prove the following.

\begin{lem}
Let $F$ be a reduced endofunctor of $\mathcal{Sp}$. Then
\[
\mathrm{coSp}(\mathrm{coAlg}\, F) \simeq \mathrm{coAlg}\, P^1F
\]
\end{lem}

\begin{proof}
The argument is analogous to the proof for $P_m$. Define $T^1F := \Sigma F \Sigma^{-1}$. There is an equivalence
\[
\mathrm{coAlg}\, T^1F \xlongrightarrow{\sim} \mathrm{coAlg}\, F
\]
given by adjoining $\Sigma^{-1}$, that is,
\[
(X \to \Sigma F \Sigma^{-1}X)
\longmapsto
(\Sigma^{-1}X \to F(\Sigma^{-1}X))
\]
This identifies $\mathrm{coAlg}\, T^1F$ with $\mathrm{coAlg}\, F$. Under this identification, the composite
\[
\mathrm{coAlg}\, T^1F
\longrightarrow
\mathrm{coAlg}\, F
\xlongrightarrow{\Sigma}
\mathrm{coAlg}\, F
\]
corresponds to the canonical map induced by the natural transformation $T^1F \to F$. Therefore the defining diagram for the costabilization becomes
\[
\mathrm{coSp}(\mathrm{coAlg}\, F)
\simeq
\mathrm{lim}_n \mathrm{coAlg}\,(T^1)^{(n)} F \simeq \mathrm{coAlg}\, P^1F
\]
where in the last step we use that $\mathrm{coAlg}$ preserves limits, as it is a right adjoint.
\end{proof}

\begin{lem}
Let $1 \to G \leftarrow F$ be a span of natural transformations of reduced endofunctors of $\mathcal{Sp}$. Then
\[
\mathrm{coSp}(\mathrm{Lift}(1 \to G \leftarrow F))
\simeq
\mathrm{Lift}(1 \to P^1G \leftarrow P^1F)
\]
\end{lem}

\begin{proof}
The category of lifts is defined by a pullback diagram in $\mathrm{Cat}_*^\omega$. Since costabilization is defined by a limit construction and therefore preserves limits, applying $\mathrm{coSp}$ to the defining diagram yields
\[
\mathrm{coSp}(\mathrm{Lift}(1 \to G \leftarrow F))
\simeq
\mathrm{Lift}(1 \to P^1G \leftarrow P^1F)
\]
where we use the previous lemma to identify the coalgebra terms.
\end{proof}

We are now ready to show that the costabilization of the metastable category is the category of spectra equipped with a nullhomotopy of the Tate diagonal.

\begin{lem}
	\label{cosp}
The costabilization of the metastable category is given by
\[
\mathrm{coSp}(P_2\mathcal{S}_*) \simeq \mathcal{Sp}_{\Delta=0} := \mathrm{Null}(1 \xlongrightarrow{\Delta} (-^{\otimes 2})^{tC_2})
\]
\end{lem}

\begin{proof}
By the previous lemma we have
\[
\mathrm{coSp}(P_2\mathcal{S}_*)
\simeq
\mathrm{Lift}(1 \to P^1(-^{\otimes 2})^{tC_2} \leftarrow P^1(-^{\otimes 2})^{hC_2})
\]
We claim that $P^1(-^{\otimes 2})^{hC_2} \simeq *$ or, equivalently, in the language of dual calculus, that the functor $(-^{\otimes 2})^{hC_2}$ is $2$-cohomogeneous. Indeed,
\[
P^1(-^{\otimes 2})^{hC_2}
\simeq
\mathrm{lim}_n \Sigma^n(\mathbb{S}^{-2n} \otimes -^{\otimes 2})^{hC_2}
\simeq
(\mathrm{lim}_n \mathbb{S}^{-n} \otimes -^{\otimes 2})^{hC_2}
\simeq
*
\]
where we use that homotopy fixed points commute with limits, since they are limits themselves. On the other hand, the Tate power is exact, so $P^1(-^{\otimes 2})^{tC_2} \simeq (-^{\otimes 2})^{tC_2}$. It follows that the pullback defining the category of lifts reduces to
\[
\mathrm{coSp}(P_2\mathcal{S}_*)
\simeq
\mathrm{Lift}(1 \xrightarrow{\Delta} (-^{\otimes 2})^{tC_2} \leftarrow 0)
\]
and hence to the category of nullhomotopies of the Tate diagonal, as claimed.
\end{proof}

\begin{rem}
	Many of the examples of spectra with metastable structure discussed in the first section - such as $2$-acyclic spectra (ex. \ref{2inv}), the Brown--Comenetz dual (ex. \ref{0square}), and monochromatic spectra (ex. \ref{monochrom}) - arise naturally from the costabilization of the metastable category.
\end{rem}

It is worth contrasting this with the classical cases. The costabilization of the category of pointed spaces is trivial, while spectra are already stable and hence equivalent to their own costabilization. In contrast, the metastable category is intermediate: its costabilization is nontrivial and consists of those spectra for which the Tate diagonal is nullhomotopic, thereby eliminating a substantial portion of spectra. This situation is summarized by the following diagram
\[\begin{tikzcd}
	{*} & {\mathcal{Sp}_{\Delta =0}} & {\mathcal{Sp}} \\
	{\mathcal{S}_*} & {P_2\mathcal{S}_*} & {\mathcal{Sp}} \\
	{\mathcal{Sp}} & {\mathcal{Sp}} & {\mathcal{Sp}}
	\arrow[from=1-1, to=1-2]
	\arrow["{\Sigma_\infty}", from=1-1, to=2-1]
	\arrow[from=1-2, to=1-3]
	\arrow["{\Sigma_\infty}", from=1-2, to=2-2]
	\arrow["{\Sigma_\infty}", equals, from=1-3, to=2-3]
	\arrow[from=2-1, to=2-2]
	\arrow["{\Sigma^\infty}", from=2-1, to=3-1]
	\arrow[from=2-2, to=2-3]
	\arrow["{\Sigma^\infty}", from=2-2, to=3-2]
	\arrow["{\Sigma^\infty}", equals, from=2-3, to=3-3]
	\arrow[equals, from=3-1, to=3-2]
	\arrow[equals, from=3-2, to=3-3]
\end{tikzcd}\]
which is induced by applying costabilization (top row) and stabilization (bottom row) to the diagram of categories
\[
\mathcal{S}_* \longrightarrow P_2\mathcal{S}_* \longrightarrow \mathcal{Sp}
\]
\subsection{Splittings and desuspensions}

In this subsection we study certain colimits in the metastable category $P_2\mathcal{S}_*$. We first give a criterion for an object to split as a coproduct and then establish a criterion for an object to desuspend. Finally, we relate the desuspension criterion to the theory of co-H objects and show that in the metastable category suspensions coincide with co-H objects.

From the discussion in the previous subsections, the colimit of a diagram $(X_i)$ in $P_2\mathcal{S}_*$ is computed on underlying spectra, and the metastable structure is obtained by taking the colimit of the defining lifts. Concretely, the induced maps fit into the diagram
\[
\begin{tikzcd}
	& {\mathrm{colim}_i((X_i)^{\otimes 2})^{hC_2}} & {((\mathrm{colim}_i X_i)^{\otimes 2})^{hC_2}} \\
	{\mathrm{colim}_i X_i} & {\mathrm{colim}_i((X_i)^{\otimes 2})^{tC_2}} & {((\mathrm{colim}_i X_i)^{\otimes 2})^{tC_2}}
	\arrow[from=1-2, to=1-3]
	\arrow[from=1-2, to=2-2]
	\arrow[from=1-3, to=2-3]
	\arrow[from=2-1, to=1-2]
	\arrow[from=2-1, to=2-2]
	\arrow[from=2-2, to=2-3]
\end{tikzcd}
\]
Here the maps on the left are obtained by taking the colimit of the diagrams defining the metastable structures on the $X_i$, while the maps on the right are the canonical colimit comparison maps. Since $(-^{\otimes 2})^{tC_2}$ is exact, the bottom-right comparison map is an equivalence for finite colimits.

\begin{lem}
	\label{coprodp2}
Let $X \in P_2\mathcal{S}_*$. Then $X$ is a coproduct if and only if the underlying spectrum splits as
\[
X \simeq X_1 \oplus X_2
\]
and the composite
\[
X \simeq X_1 \oplus X_2 
\longrightarrow 
((X_1 \oplus X_2)^{\otimes 2})^{hC_2} 
\longrightarrow 
X_1 \otimes X_2
\]
is null.
\end{lem}

\begin{proof}
	Consider two objects $X_1$ and $X_2$ in $P_2\mathcal{S}_*$. Their coproduct is given by 
\[
\begin{tikzcd}[cramped]
	& {(X_1^{\otimes 2})^{hC_2} \oplus (X_2^{\otimes 2})^{hC_2}} & {((X_1 \oplus X_2)^{\otimes 2})^{hC_2}} & {(X_1^{\otimes 2})^{hC_2} \oplus (X_2^{\otimes 2})^{hC_2} \oplus (X_1 \otimes X_2)} \\
	{X_1 \oplus X_2} & {(X_1^{\otimes 2})^{tC_2} \oplus (X_2^{\otimes 2})^{tC_2}} & {((X_1 \oplus X_2)^{\otimes 2})^{tC_2}}
	\arrow[from=1-2, to=1-3]
	\arrow[from=1-2, to=2-2]
	\arrow[equals, from=1-3, to=1-4]
	\arrow[from=1-3, to=2-3]
	\arrow[from=2-1, to=1-2]
	\arrow[from=2-1, to=2-2]
	\arrow[equals, from=2-2, to=2-3]
\end{tikzcd}
\]
The identification on the bottom row follows from exactness of the Tate power construction, while the identification in the top row follows from the canonical $C_2$-equivariant decomposition
\[
(X_1 \oplus X_2)^{\otimes 2}
\simeq
X_1^{\otimes 2}
\oplus
X_2^{\otimes 2}
\oplus
C_{2+} \otimes (X_1 \otimes X_2)
\]
Now we want to reverse-engineer this. From the diagram we immediately see that a coproduct $X$ in $P_2\mathcal{S}_*$ must have a split underlying spectrum. This also follows from the fact that the forgetful functor preserves colimits. If $X \simeq X_1 \oplus X_2$ as spectra, then we obtain the diagram above except for the diagonal arrow. To produce the diagonal it suffices that the composite
\[
X_1 \oplus X_2 \longrightarrow ((X_1 \oplus X_2)^{\otimes 2})^{hC_2} \longrightarrow X_1 \otimes X_2
\]
is null, since the rightmost summand is the cofiber of the map
\[
(X_1^{\otimes 2})^{hC_2} \oplus (X_2^{\otimes 2})^{hC_2}
\longrightarrow
((X_1 \oplus X_2)^{\otimes 2})^{hC_2}
\]
This completes the proof.
\end{proof}

\begin{ex}
	Consider the images of the product of two spheres $S^n \times S^n$ and of the wedge $S^n \vee S^n \vee S^{2n}$ under the functor $\Sigma^\infty_2\colon \mathcal{S}_* \to P_2\mathcal{S}_*$. Both objects have underlying spectrum
	\[
	\mathbb{S}^n \oplus \mathbb{S}^n \oplus \mathbb{S}^{2n}
	\]
	so they are indistinguishable via the functor $\Sigma^\infty\colon \mathcal{S}_* \to \mathcal{Sp}$.

	They are different in the metastable category. The wedge of spheres is a coproduct in $P_2\mathcal{S}_*$ since the functor from spaces preserves colimits. The product of spheres is not. To see this, recall that
	\[
	\mathbb{Z}^*(S^n \times S^n) \cong \mathbb{Z}[\alpha, \beta]/(\alpha^2, \beta^2)
	\]
	with $|\alpha| = |\beta| = n$. The mixed summand in the Künneth decomposition detects the cup product $\alpha\beta$. Since $\alpha\beta \neq 0$, the mixed component of the diagonal is nontrivial, hence the product of spheres does not split in the metastable category.
\end{ex}

\begin{lem}
	\label{suspensionp2}
Let $X \in P_2\mathcal{S}_*$. Then $X$ is a suspension if and only if the diagonal map
\[
X \longrightarrow (X^{\otimes 2})^{hC_2} \longrightarrow X \otimes X
\]
is null.
\end{lem}

\begin{proof}
	Let $Y$ be a spectrum with a metastable structure. Its suspension in $P_2\mathcal{S}_*$ is given by 
\[\begin{tikzcd}
	& {\Sigma(Y^{\otimes 2})^{hC_2}} & {((\Sigma Y)^{\otimes 2})^{hC_2}} \\
	{\Sigma Y} & {\Sigma(Y^{\otimes 2})^{tC_2}} & {((\Sigma Y)^{\otimes 2})^{tC_2}}
	\arrow[from=1-2, to=1-3]
	\arrow[from=1-2, to=2-2]
	\arrow[from=1-3, to=2-3]
	\arrow[from=2-1, to=1-2]
	\arrow[from=2-1, to=2-2]
	\arrow[equals, from=2-2, to=2-3]
\end{tikzcd}\]
	We now reverse-engineer this description. Let $X \in P_2\mathcal{S}_*$ and let $Y$ be the spectrum $\Sigma^{-1}X$. Then $X$ desuspends in $P_2\mathcal{S}_*$ if and only if the metastable structure on $X$ is induced from one on $Y$. This holds precisely when there exists a dotted arrow
\[\begin{tikzcd}
	{(\Sigma^{-\rho}X^{\otimes 2})^{hC_2}} & {\Sigma((\Sigma^{-1}X)^{\otimes 2})^{hC_2}} & {(X^{\otimes 2})^{hC_2}} \\
	& X & {(X^{\otimes 2})^{tC_2}}
	\arrow[equals, from=1-1, to=1-2]
	\arrow[from=1-2, to=1-3]
	\arrow[from=1-3, to=2-3]
	\arrow[dashed, from=2-2, to=1-2]
	\arrow[from=2-2, to=1-3]
	\arrow[from=2-2, to=2-3]
\end{tikzcd}\]
	where $\rho$ denotes the sign representation. It therefore suffices to understand the cofiber of the top horizontal map. Recall the cofiber sequence
	$$C_{2+} \simeq S(\rho)_+ \longrightarrow S^0 \longrightarrow S^\rho$$
	where $S(V)$ denotes the sphere inside the representation $V$. Taking duals, smashing with $X^{\otimes 2}$, and passing to fixed points yields the sequence
	$$(\Sigma^{-\rho}X^{\otimes 2})^{hC_2} \longrightarrow (X^{\otimes 2})^{hC_2} \longrightarrow (C_{2+} \otimes X^{\otimes 2})^{hC_2} \simeq X^{\otimes 2}$$
	Hence the lift exists if and only if the composite
	$$X \longrightarrow (X^{\otimes 2})^{hC_2} \longrightarrow X \otimes X$$
	is null.
\end{proof}

\begin{ex}
	For every $n > 0$, the diagonal map on $\mathbb{S}^n$ is null for degree reasons. This agrees with the fact that the sphere $S^n$ in $P_2\mathcal{S}_*$ is a suspension for $n > 0$. For $n = 0$, the diagonal map on the exotic $0$-sphere $S^{0,k}$ has degree $k$, which is always an odd integer and in particular nonzero. Hence no exotic $0$-sphere desuspends. This can also be seen from the fact that the spectrum $\mathbb{S}^{-1}$ does not admit a metastable structure.

	From the proof of the lemma above, we also see that the metastable structure on $\mathbb{S}^n$ is obtained by suspending $n$ times the composite
$$
\mathbb{S} \longrightarrow \mathbb{S}^{hC_2} \longrightarrow (\mathbb{S}^{n\rho})^{hC_2}
$$
where the first map is the spacelike metastable structure on $\mathbb{S}$, namely the inclusion induced by equipping $\mathbb{S}$ with the trivial $C_2$-action.
\end{ex}

\begin{ex}
	The image of $S^1/2^n$ does not desuspend in $P_2\mathcal{S}_*$. Indeed, its $\mathbb{Z}/2^n$-cohomology has basis $\{e_1, e_2\}$ with $|e_i| = i$ and nontrivial cup product given by \cite[proof of thm.~7.1]{coeff}
	$$e_1 \cup e_1 = \binom{2^n + 1}{2} e_2 = 2^{n-1} e_2 \neq 0$$
	This result is less immediate than the previous one, since at this stage we do not know whether $\mathbb{S}/2^n$ admits a metastable structure for $n > 1$.
\end{ex}

We conclude this subsection by relating the previous lemma to the general behavior of $2$-excisive categories with respect to counital comonoids, or co-H objects.

Recall that in a pointed category $C$, a co-H object consists of an object $X$ in $C$ together with a factorization of the diagonal

\[\begin{tikzcd}
	& {X \vee X} \\
	X & {X \times X}
	\arrow[from=1-2, to=2-2]
	\arrow[dashed, from=2-1, to=1-2]
	\arrow[from=2-1, to=2-2]
\end{tikzcd}\]
This is equivalent to specifying a counital comultiplication $X \to X \vee X$. This notion coincides with that of a comonoid over the $\mathbb{A}_2$ operad, which in turn is equivalent to a coalgebra over $\mathbb{A}_2$ with respect to the cocartesian monoidal structure. We therefore write $\mathrm{coMon}_{\mathbb{A}_2} C$ for the collection of co-H objects in $C$. A colimit-preserving functor $C \to D$ induces a functor
 $$\mathrm{coMon}_{\mathbb{A}_2}C \to \mathrm{coMon}_{\mathbb{A}_2}D$$

Applying this to our context, we immediately obtain a functor
\[
\mathrm{coMon}_{\mathbb{A}_2}\mathcal{S}_*
\longrightarrow
\mathrm{coMon}_{\mathbb{A}_2}P_2\mathcal{S}_*
\]
Now, a co-H structure on $X$ in $\mathcal{S}_*$ - that is, a co-H space - determines a choice of nullhomotopy of the diagonal $X \to X \wedge X$. This induces a nullhomotopy of $
\Sigma^\infty X \to \Sigma^\infty X \otimes \Sigma^\infty X$, and hence, by the previous lemma, a choice of desuspension structure on the image of $X$ in the metastable category $P_2\mathcal{S}_*$. In other words, the functor above factors as

\[\begin{tikzcd}
	& {P_2\mathcal{S}_*} \\
	{\mathrm{coMon}_{\mathbb{A}_2}\mathcal{S}_*} & {\mathrm{coMon}_{\mathbb{A}_2}P_2\mathcal{S}_*}
	\arrow["\Sigma", from=1-2, to=2-2]
	\arrow[dashed, from=2-1, to=1-2]
	\arrow[from=2-1, to=2-2]
\end{tikzcd}\]
This result can also be obtained in another way, using the following.

\begin{lem}
	The suspension functor
	$$P_2\mathcal{S}_* \xlongrightarrow{\Sigma} \mathrm{coMon}_{\mathbb{A}_2}P_2\mathcal{S}_*$$
	induces an equivalence.
\end{lem}

Before discussing the proof, we record the immediate corollary of this lemma together with the previous one, which characterizes suspensions in the metastable category.

\begin{lem}
	\label{suspensionp2gen}
	Let $X \in P_2\mathcal{S}_*$. Then the following are equivalent:
	\begin{enumerate}
		\item $X$ is a suspension.
		\item $X$ is a co-H object.
		\item $X \to (X^{\otimes 2})^{hC_2} \to X \otimes X$ is null.
	\end{enumerate}
\end{lem}

\begin{rem}
	The reader should compare this result with the Berstein--Hilton theorem \cite[thm.~A]{coHsusp}, which states that for a space $X$ with cells only in the metastable range, the data of a co-H structure is equivalent to the data of a suspension structure.
\end{rem}

We now give the proof.

\begin{proof}
	The first thing to observe is that, since $P_2\mathcal{S}_*$ is $2$-excisive, the functor $L_2$ is an equivalence \cite[prop.~2.16]{Gijsapprox}. It therefore suffices to construct the dotted arrow in the diagram
\[\begin{tikzcd}
	{P_2\mathcal{S}_*} & {T_2P_2\mathcal{S}_*} \\
	& {\mathrm{coMon}_{\mathbb{A}_2}P_2\mathcal{S}_*}
	\arrow["{L_2}", from=1-1, to=1-2]
	\arrow["\Sigma"', from=1-1, to=2-2]
	\arrow["\simeq"', dashed, from=2-2, to=1-2]
\end{tikzcd}\]
	In order to do this, to each co-H object
	$$X \xlongrightarrow{\Delta} X \vee X$$
	we associate the diagram in $T_2P_2\mathcal{S}_*$ given by
	\[\begin{tikzcd}[cramped]
	&& {*} & \\
	& {*} && X \\
	{*} && X \\
	& X && {X \vee X}
	\arrow[from=1-3, to=2-4]
	\arrow[from=1-3, to=3-3]
	\arrow[from=2-2, to=2-4]
	\arrow[from=2-2, to=4-2]
	\arrow["\Delta"{description}, from=2-4, to=4-4]
	\arrow[from=3-1, to=3-3]
	\arrow[from=3-1, to=4-2]
	\arrow["{i_2}"{description}, hook, from=3-3, to=4-4]
	\arrow["{i_1}"{description}, hook, from=4-2, to=4-4]
\end{tikzcd}\]
	where $i_j$ denote the inclusion maps. This punctured cube is evidently reduced. It remains to verify that it is strongly cocartesian. This is clear for the bottom square, so we only need to check the remaining two. By symmetry, it suffices to verify this for one of them. Consider the following diagram, obtained by taking horizontal cofibers
\[\begin{tikzcd}
	{*} & X & X \\
	X & {X \vee X} & X
	\arrow[from=1-1, to=1-2]
	\arrow[from=1-1, to=2-1]
	\arrow[equals, from=1-2, to=1-3]
	\arrow["\Delta", from=1-2, to=2-2]
	\arrow[equals, from=1-2, to=2-3]
	\arrow[dashed, from=1-3, to=2-3]
	\arrow["{i_1}"', from=2-1, to=2-2]
	\arrow["{\pi_2}"', from=2-2, to=2-3]
\end{tikzcd}\]
	The diagonal arrow is the identity by the definition of a co-H object, hence the induced map on cofibers is an equivalence. It follows that the left square is a pushout, since the forgetful functor $P_2\mathcal{S}_* \to \mathcal{Sp}$ creates colimits and the category of spectra is stable.

	Hence we obtain a functor
	\[
	\mathrm{coMon}_{\mathbb{A}_2}P_2\mathcal{S}_*
	\longrightarrow
	T_2P_2\mathcal{S}_*
	\]
	Moreover, it is immediate that precomposing with the suspension functor recovers the functor $L_2$.

	It remains to prove that this functor is an equivalence. We sketch the essential surjectivity to provide some intuition, and then refer to the literature for the general statement.

	Consider a punctured $3$-cube
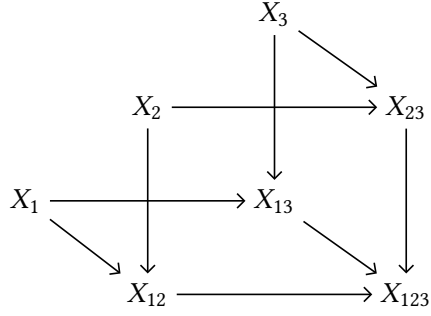
\[\begin{tikzcd}[cramped]
	&& {X_3} & \\
	& {X_2} && {X_{23}} \\
	{X_1} && {X_{13}} \\
	& {X_{12}} && {X_{123}}
	\arrow[from=1-3, to=2-4]
	\arrow[from=1-3, to=3-3]
	\arrow[from=2-2, to=2-4]
	\arrow[from=2-2, to=4-2]
	\arrow[from=2-4, to=4-4]
	\arrow[from=3-1, to=3-3]
	\arrow[from=3-1, to=4-2]
	\arrow[from=3-3, to=4-4]
	\arrow[from=4-2, to=4-4]
\end{tikzcd}\]
	The condition of being special implies $X_1 \simeq X_2 \simeq X_3 \simeq *$. Since the cube is strongly cocartesian, all faces are pushouts. In particular, we may identify the bottom square with
\[\begin{tikzcd}
	{*} & {X_{13}} \\
	{X_{12}} & {X_{12} \vee X_{13} \simeq X_{123}}
	\arrow[from=1-1, to=1-2]
	\arrow[from=1-1, to=2-1]
	\arrow["{i_2}", hook, from=1-2, to=2-2]
	\arrow["{i_1}"', hook, from=2-1, to=2-2]
	\arrow["\lrcorner"{anchor=center, pos=0.125, rotate=180}, draw=none, from=2-2, to=1-1]
\end{tikzcd}\]
	Thus $X_{123} \simeq X_{12} \vee X_{13}$ and the remaining edge gives a map
	$$X_{23} \xlongrightarrow{\Delta} X_{12} \vee X_{13}$$
	which will play the role of the comultiplication. Since all faces are pushouts, taking horizontal cofibers yields the diagram
\[\begin{tikzcd}
	{*} & {X_{23}} & {X_{23}} \\
	{X_{12}} & {X_{12} \vee X_{13}} & {X_{13}}
	\arrow[from=1-1, to=1-2]
	\arrow[from=1-1, to=2-1]
	\arrow[equals, from=1-2, to=1-3]
	\arrow["\Delta", from=1-2, to=2-2]
	\arrow["\simeq", dashed, from=1-3, to=2-3]
	\arrow["{i_1}"', hook, from=2-1, to=2-2]
	\arrow["\lrcorner"{anchor=center, pos=0.125, rotate=180}, draw=none, from=2-2, to=1-1]
	\arrow["{\pi_2}"', from=2-2, to=2-3]
\end{tikzcd}\]
	This identifies $X_{23}$ with $X_{13}$ and establishes one of the counitality conditions. By symmetry, applying the same argument to the other face identifies $X_{23}$ with $X_{12}$ and yields the second counitality condition. Combining these identifications we obtain $X_{12} \simeq X_{13}$, and hence all three doubleton vertices may be identified with a single object $X$. Under these identifications the map $\Delta$ becomes a map $$X \xlongrightarrow{\Delta} X \vee X$$ satisfying the counitality conditions. This shows that the cube lies in the image of the functor constructed above.
\end{proof}

The general result was proven by Fuentes-Keuthan in \cite{Fuentes}. To state it, we recall some notation. For a category $C$, the category of $\mathbb{A}_m$-comonoids is defined as
$$\mathrm{coMon}_{\mathbb{A}_m}C := \mathrm{Mon}_{\mathbb{A}_m}(C^{\mathrm{op}})^{\mathrm{op}}$$

Explicitly, $\mathbb{A}_m$-comonoids in $C$ are functors $\Delta_m \to C$ that satisfy the Segal condition and a unitality condition (see \cite[def.~1.3]{Fuentes} for a complete definition). We also define the category of $\mathbb{A}_m$-cogroups - or grouplike $\mathbb{A}_m$-comonoids - by
$$\mathrm{coMon}_{\mathbb{A}_m}^{gp}C := \mathrm{Mon}^{gp}_{\mathbb{A}_m}(C^{\mathrm{op}})^{\mathrm{op}}$$

Explicitly, these are $\mathbb{A}_m$-comonoids that satisfy the additional grouplike condition, which requires certain canonical squares associated to the comultiplication to be pushouts (see \cite[def.~1.5]{Fuentes}).

With this notation in place, we state the main result of Fuentes-Keuthan \cite[thm.~1.22 and thm.~2.5]{Fuentes}.

\begin{thm}
	Let $C$ be a pointed compactly generated category. Then there is an equivalence
	$$\mathrm{coMon}_{\mathbb{A}_m}^{gp}C \xlongrightarrow{\sim} T_mC$$
	such that the diagram
\[\begin{tikzcd}
	C & {T_mC} \\
	& {\mathrm{coMon}^{gp}_{\mathbb{A}_m}C}
	\arrow["{L_m}", from=1-1, to=1-2]
	\arrow["\Sigma"', from=1-1, to=2-2]
	\arrow["\simeq"', from=2-2, to=1-2]
\end{tikzcd}\]
	commutes. In particular, $C$ is $m$-excisive if and only if the suspension functor
	$$C \xlongrightarrow{\Sigma} \mathrm{coMon}_{\mathbb{A}_m}^{gp}C$$
	induces an equivalence.
\end{thm}

\begin{rem}
	This is slightly weaker than what we proved for $P_2\mathcal{S}_*$. Indeed, in this case we dropped the grouplike condition. We can do this because the inclusion
	$$\mathrm{coMon}^{gp}_{\mathbb{A}_2} P_2\mathcal{S}_* \hookrightarrow \mathrm{coMon}_{\mathbb{A}_2} P_2\mathcal{S}_*$$
	is an equivalence of categories. This was implicit in the proof above, but for clarity we repeat the argument here. An $\mathbb{A}_2$-comonoid $X$ is grouplike whenever the diagrams
\[\begin{tikzcd}
	{*} & X \\
	X & {X \vee X}
	\arrow[from=1-1, to=1-2]
	\arrow[from=1-1, to=2-1]
	\arrow["\Delta", from=1-2, to=2-2]
	\arrow["{i_j}"', from=2-1, to=2-2]
\end{tikzcd}\]
	for $j = 1, 2$ are pushouts. By counitality, the induced map on horizontal cofibers is an equivalence. Hence this diagram is cocartesian whenever $C$ admits a functor to a stable category that creates colimits, as is the case for $P_2\mathcal{S}_*$
\end{rem}

We conclude this subsection by observing that Fuentes-Keuthan's result provides yet another description of the metastable category as obtained by formally inverting suspension on grouplike co-H spaces
$$P_2\mathcal{S}_* \simeq \mathrm{colim}\big(\mathcal{S}_* \xlongrightarrow{\Sigma} \mathrm{coMon}_{\mathbb{A}_2}^{gp}\mathcal{S}_* \xlongrightarrow{\Sigma} \mathrm{coMon}_{\mathbb{A}_2}^{gp}\mathrm{coMon}_{\mathbb{A}_2}^{gp}\mathcal{S}_* \longrightarrow \cdots\big)$$

\section{EHP sequences}

In this section we develop computational tools for the metastable category. The main structural result is that every object $X$ of $P_2\mathcal{S}_*$ fits into a fiber sequence
$$
X \xlongrightarrow{E^n} \Omega^n\Sigma^n X \xlongrightarrow{H^n} \Omega^\infty\Sigma^\infty(S(n\rho)_+ \otimes X^{\otimes 2})_{hC_2}
$$
where $\rho$ denotes the sign representation.

The case $n=\infty$ gives the canonical resolution
$$
X \xlongrightarrow{E^\infty} \Omega^\infty\Sigma^\infty X \xlongrightarrow{H^\infty} \Omega^\infty\Sigma^\infty X^{\otimes 2}_{hC_2}
$$
which should be thought of as the analogue of a resolution of a $2$-step nilpotent group in terms of abelian groups. We refer to the map $H^\infty$ as the James--Hopf map, since it induces the classical James--Hopf map after passing to pointed spaces.

The case $n=2$ gives a description of the fiber of the double suspension in terms of the swap map $\tau$,
$$
X \xlongrightarrow{E^2} \Omega^2\Sigma^2 X \xlongrightarrow{H^2} \Omega^\infty\Sigma^\infty X^{\otimes 2}/(1-\tau)
$$
and this will be used to prove an exponent theorem. 

The case $n=1$ gives the metastable EHP sequence
$$
X \xlongrightarrow{E} \Omega\Sigma X \xlongrightarrow{H} \Omega^\infty\Sigma^\infty X^{\otimes 2}
$$
from which we derive a spectral sequence with $E^1$-page
$$
E^1_{t,m}:=
\begin{cases}
	\pi_tX & m=0 \\
	\pi_{t-m+1}^sX^{\otimes 2} & 0<m\leq n \\
	0 & m>n
\end{cases}
$$
converging to $\pi_{n+t}\Sigma^nX$.

In order to use this spectral sequence to compute homotopy groups of spheres in the metastable category, we first study the spheres $S^0$ and $S^1$. Introducing the bigraded homotopy groups
$$
\pi_{n,j}X := [S^{n,j},X]
$$
we show that the bigraded homotopy groups of the exotic $0$-spheres $S^{0,k}$ are given by
$$
\pi_{n,j}S^{0,k} \cong
\begin{cases}
	\mathbb{Z}/2 & \text{if } n=0 \text{ and } j \nmid k \\
	\mathbb{Z}/2 \times \{0,k/j\} & \text{if } n=0 \text{ and } j \mid k \\
	\pi_{n+1}\mathbb{RP}^\infty_{-1} \oplus \pi_{n+1}(\mathbb{S}/k) & \text{if } n>0
\end{cases}
$$
while the bigraded homotopy groups of $S^1$ are given by
$$
\pi_{n,j}S^1 \cong
\begin{cases}
	0 & \text{if } n=0 \\
	\mathbb{Z} & \text{if } n=1 \\
	\pi_{n+1}\mathbb{Sp}^2 \oplus \pi_{n-1}\mathbb{S}[\sfrac{1}{2}] & \text{if } n>1
\end{cases}
$$

These calculations provide the first explicit examples in the metastable category and serve as input for the spectral sequence above. We then use this information to compute low-dimensional examples and, as an application, prove that no sphere is a loop space, except possibly $S^7$.

The section is organized as follows. 

In the first subsection we prove the canonical resolution, namely the case $n=\infty$, by means of a general result on categories of lifts. 

In the second subsection we prove the general $n$-fold resolution and analyze in detail the cases $n=1$ and $n=2$. 

In the third subsection we study the James--Hopf map $H^\infty$, prove a metastable version of the Kahn--Priddy theorem, and compute the bigraded homotopy groups of $S^{0,k}$. 

In the fourth subsection we refine this analysis, prove the reduced Kahn--Priddy theorem, and compute the bigraded homotopy groups of $S^1$.

\subsection{Canonical resolution}

In this subsection we prove that every object $X$ of $P_2\mathcal{S}_*$ fits into a fiber sequence
$$
X \xlongrightarrow{E^\infty} \Omega^\infty\Sigma^\infty X \xlongrightarrow{H^\infty} \Omega^\infty\Sigma^\infty X^{\otimes 2}_{hC_2}
$$
which we call the \emph{canonical resolution}. Here $E^\infty$ is the stabilization map, namely the unit of the adjunction $\Sigma^\infty \dashv \Omega^\infty$, while $H^\infty$ is the \emph{James--Hopf map}. Since $\Sigma^\infty$ is symmetric monoidal and colimit-preserving, we will freely identify $(\Sigma^\infty X)^{\otimes 2}_{hC_2}$ with $\Sigma^\infty(X^{\otimes 2})_{hC_2}$.

The proof proceeds in three steps. We first establish an analogous fiber sequence in the category of coalgebras. We then upgrade it to a fiber sequence in a category of lifts. Finally, we specialize to the metastable category and use the fact that the forgetful functor coincides with stabilization.

We begin with coalgebras. Let $F$ be an accessible reduced endofunctor of spectra. By lemma \ref{propcoalg}, the forgetful functor
$$
\mathrm{coAlg}\,F \longrightarrow \mathcal{Sp}
$$
creates colimits. In particular, it admits a right adjoint, which we denote by
\[
\begin{tikzcd}
	{\underline{\,\,}: \mathrm{coAlg}\,F} & {\mathcal{Sp}: \mathrm{cofree}}
	\arrow[""{name=0, anchor=center, inner sep=0}, shift left, from=1-1, to=1-2]
	\arrow[""{name=1, anchor=center, inner sep=0}, shift left, from=1-2, to=1-1]
	\arrow["\dashv"{anchor=center, rotate=-90}, draw=none, from=0, to=1]
\end{tikzcd}
\]
and refer to it as the \emph{cofree} functor. With this notation in place, we prove the following.

\begin{lem}
	Let $F$ be an accessible reduced endofunctor of $\mathcal{Sp}$. Then any $F$-coalgebra $X$ fits into a fiber sequence
	$$
	X \longrightarrow \mathrm{cofree}\,\underline{X} \longrightarrow \mathrm{cofree}\,F\underline{X}
	$$
\end{lem}

\begin{proof}
	We will show that for every $F$-coalgebra $T$ there is a fiber sequence of mapping spaces
	$$
	\mathrm{Map}_{\mathrm{coAlg}\,F}(T, X) \longrightarrow \mathrm{Map}_{\mathcal{Sp}}(\underline{T}, \underline{X}) \longrightarrow \mathrm{Map}_{\mathcal{Sp}}(\underline{T}, F\underline{X})
	$$
	Adjunction then identifies this with a fiber sequence
	$$
	\mathrm{Map}_{\mathrm{coAlg}\,F}(T, X) \longrightarrow \mathrm{Map}_{\mathrm{coAlg}\,F}(T, \mathrm{cofree}\,\underline{X}) \longrightarrow \mathrm{Map}_{\mathrm{coAlg}\,F}(T, \mathrm{cofree}\,F\underline{X})
	$$
	and the claim follows by Yoneda. To prove the existence of the first fiber sequence, consider the diagram
\[\begin{tikzcd}
	{\mathrm{Map}_{\mathrm{coAlg}\,F}(T, X)} & {\mathrm{Map}_{\mathrm{Fun}(\Delta^1, \mathcal{Sp})}(L_T, L_X)} & {\mathrm{Map}_{\mathcal{Sp}}(\underline{T}, F\underline{X})} \\
	{\mathrm{Map}_{\mathcal{Sp}}(\underline{T}, \underline{X})} & {\mathrm{Map}_{\mathcal{Sp}}(\underline{T}, \underline{X}) \times \mathrm{Map}_{\mathcal{Sp}}(F\underline{T}, F\underline{X})} & {\mathrm{Map}_{\mathcal{Sp}}(\underline{T}, F\underline{X}) \times \mathrm{Map}_{\mathcal{Sp}}(\underline{T}, F\underline{X})}
	\arrow[from=1-1, to=1-2]
	\arrow[from=1-1, to=2-1]
	\arrow[from=1-2, to=1-3]
	\arrow[from=1-2, to=2-2]
	\arrow["\Delta", from=1-3, to=2-3]
	\arrow[""{name=0, anchor=center, inner sep=0}, "{(1, F)}"', from=2-1, to=2-2]
	\arrow[""{name=1, anchor=center, inner sep=0}, "{L_{X_*} \times L_T^*}"', from=2-2, to=2-3]
	\arrow["\lrcorner"{anchor=center, pos=0.125}, draw=none, from=1-1, to=0]
	\arrow["\lrcorner"{anchor=center, pos=0.125}, draw=none, from=1-2, to=1]
\end{tikzcd}\]
	The left square is a pullback because $\mathrm{coAlg}\,F$ is defined as a pullback, so mapping spaces in $\mathrm{coAlg}\,F$ are computed as pullbacks. The right square is the pullback defining mapping spaces in the arrow category. Therefore the outer square is again a pullback. It follows that
\[\begin{tikzcd}
	{\mathrm{Map}_{\mathrm{coAlg}\,F}(T, X) \simeq \mathrm{eq}\Bigl(\mathrm{Map}_{\mathcal{Sp}}(\underline{T}, \underline{X}) } & {\mathrm{Map}_{\mathcal{Sp}}(\underline{T}, F\underline{X})\Bigr)}
	\arrow["{L_{X*}}"', shift right, from=1-1, to=1-2]
	\arrow["{L_T^* \circ F}", shift left, from=1-1, to=1-2]
\end{tikzcd}\]
	The equalizer can be computed at the level of mapping spectra, and therefore as the fiber
$$
\mathrm{Map}_{\mathrm{coAlg}\,F}(T, X) \simeq \mathrm{fib}\Bigl(\mathrm{Map}_{\mathcal{Sp}}(\underline{T}, \underline{X}) \xlongrightarrow{L_T^* \circ F - L_{X*}} \mathrm{Map}_{\mathcal{Sp}}(\underline{T}, F\underline{X})\Bigr)
$$
This gives the claimed fiber sequence. This description also appears in \cite[prop. II.1.5]{NikolausScholze}.
\end{proof}

\begin{rem}
	\label{Hopfexplicitcoalg}
	The proof above gives an explicit description of the map on homotopy classes induced by
	$$
	\mathrm{cofree}\,\underline X \longrightarrow \mathrm{cofree}\,F\underline X
	$$
	If $T$ and $X$ are $F$-coalgebras, a map $\alpha : \underline T \to \underline X$ is sent to the difference
	$F\alpha \circ L_T - L_X \circ \alpha$.

	Indeed, a map of $F$-coalgebras from $T$ to $X$ is precisely a map $\alpha : \underline T \to \underline X$ such that
\[\begin{tikzcd}
	{\underline T} & {F\underline T} \\
	{ \underline X} & {F\underline X}
	\arrow["{L_T}", from=1-1, to=1-2]
	\arrow["\alpha"', from=1-1, to=2-1]
	\arrow["{F\alpha}", from=1-2, to=2-2]
	\arrow["{L_X}"', from=2-1, to=2-2]
\end{tikzcd}\]
	commutes. In an abelian category one would express this by the equation
	$F\alpha \circ L_T - L_X \circ \alpha = 0$, and the fiber sequence above is the homotopical analogue of that condition.

	In particular, the map $\mathrm{cofree}\,\underline X \to \mathrm{cofree}\,F\underline X$ is not obtained by applying the cofree functor to a map of spectra $\underline X \to F\underline X$. In this sense, it is not induced by a stable map.
\end{rem}

We now pass from coalgebras to categories of lifts. As in the case of coalgebras, the forgetful functor to spectra admits a right adjoint by \ref{proplift}, which we again denote by \emph{cofree}.

\begin{lem}
	Let $1 \to G \leftarrow F$ be a span of natural transformations between accessible reduced endofunctors of $\mathcal{Sp}$. Then any object $X$ of $\mathrm{Lift}(1 \to G \leftarrow F)$ fits into a fiber sequence
	$$
	X \longrightarrow \mathrm{cofree}\,\underline{X} \longrightarrow \mathrm{cofree}\,E\underline{X}
	$$
	where $E$ denotes the fiber of the map $F \to G$.
\end{lem}

\begin{proof}
	As before, it is enough to show that for every object $T$ of $\mathrm{Lift}(1 \to G \leftarrow F)$ there is a fiber sequence of mapping spaces
	$$
	\mathrm{Map}_{\mathrm{Lift}(1 \to G \leftarrow F)}(T, X)
	\longrightarrow
	\mathrm{Map}_{\mathcal{Sp}}(\underline{T}, \underline{X})
	\longrightarrow
	\mathrm{Map}_{\mathcal{Sp}}(\underline{T}, E\underline{X})
	$$
	The claim then follows by adjunction and Yoneda.

	We now compute the first mapping space as a pullback. Since $\mathrm{Lift}(1 \to G \leftarrow F)$ is defined as a pullback, its mapping spaces are given by pullbacks as well:
\[\begin{tikzcd}
	{\mathrm{Map}_{\mathrm{Lift}(1 \to G \leftarrow F)}(T, X)} & {\mathrm{Map}_{\mathrm{coAlg}\,F}(T, X)} \\
	{\mathrm{Map}_\mathcal{Sp}(\underline{T}, \underline{X})} & {\mathrm{Map}_{\mathrm{coAlg}\,G}(T, X)}
	\arrow[from=1-1, to=1-2]
	\arrow[from=1-1, to=2-1]
	\arrow[from=1-2, to=2-2]
	\arrow[""{name=0, anchor=center, inner sep=0}, from=2-1, to=2-2]
	\arrow["\lrcorner"{anchor=center, pos=0.125}, draw=none, from=1-1, to=0]
\end{tikzcd}\]
	Here we are slightly abusing notation, since the two maps on the right also forget part of the structure. It therefore remains to prove that there is a fiber sequence
	$$
\mathrm{Map}_{\mathrm{coAlg}\,F}(T, X)
\longrightarrow
\mathrm{Map}_{\mathrm{coAlg}\,G}(T, X)
\longrightarrow
\mathrm{Map}_{\mathcal{Sp}}(\underline{T}, E\underline{X})
$$
and the result will then follow by composing pullback squares. This fiber sequence follows from the previous lemma by taking vertical fibers in the following diagram of fiber sequences:

\[\begin{tikzcd}
	{\mathrm{Map}_\mathcal{Sp}(\underline{T}, \underline{X})} & {\mathrm{Map}_\mathcal{Sp}(\underline{T}, \underline{X})} & {*} \\
	{\mathrm{Map}_\mathcal{Sp}(\underline{T}, F\underline{X})} & {\mathrm{Map}_\mathcal{Sp}(\underline{T}, G\underline{X})} & {\mathrm{Map}_\mathcal{Sp}(\underline{T}, \Sigma E\underline{X})}
	\arrow[equals, from=1-1, to=1-2]
	\arrow[from=1-1, to=2-1]
	\arrow[from=1-2, to=1-3]
	\arrow[from=1-2, to=2-2]
	\arrow[from=1-3, to=2-3]
	\arrow[from=2-1, to=2-2]
	\arrow[from=2-2, to=2-3]
\end{tikzcd}\]
	The lemma follows.
\end{proof}

\begin{rem}
	\label{Hopfexplicitlift}
	The composite
	$$
	\mathrm{cofree}\,\underline{X} \longrightarrow \mathrm{cofree}\,E\underline{X} \longrightarrow \mathrm{cofree}\,F\underline{X}
	$$
	coincides with the map encoding the $F$-coalgebra structure on $X$. In particular, this composite admits an explicit description on homotopy classes. By contrast, the first map is less explicit. Indeed, the choice of lift determines a nullhomotopy of the composite
	$$
	\mathrm{cofree}\,\underline{X} \longrightarrow \mathrm{cofree}\,F\underline{X} \longrightarrow \mathrm{cofree}\,G\underline{X}
	$$
	and this nullhomotopy in turn determines a lift to $\mathrm{cofree}\,E\underline{X}$.
\end{rem}

We are now ready to prove the canonical resolution in the metastable category.

\begin{lem}
	\label{canonicalres}
	Any object $X$ of $P_2\mathcal{S}_*$ fits into a fiber sequence
	$$
	X \xlongrightarrow{E^\infty} \Omega^\infty\Sigma^\infty X \xlongrightarrow{H^\infty} \Omega^\infty\Sigma^\infty X^{\otimes 2}_{hC_2}
	$$
\end{lem}

\begin{proof}
	Since $P_2\mathcal{S}_*$ is defined as the category of lifts associated to the span
	\[
1 \xlongrightarrow{\Delta} (-^{\otimes 2})^{tC_2}
\xlongleftarrow{\mathrm{can}} (-^{\otimes 2})^{hC_2}
\]
	the previous lemma gives, for any $X$ of $P_2\mathcal{S}_*$, a fiber sequence
	$$
	X \longrightarrow \mathrm{cofree}\,\underline{X} \longrightarrow \mathrm{cofree}\,\underline{X}^{\otimes 2}_{hC_2}
	$$
	Now the forgetful functor $P_2\mathcal{S}_* \to \mathcal{Sp}$ coincides with stabilization by lemma \ref{forgetstabilization}, so its right adjoint is $\Omega^\infty$. Therefore the fiber sequence becomes
	$$
	X \longrightarrow \Omega^\infty\Sigma^\infty X \longrightarrow \Omega^\infty(\Sigma^\infty X)^{\otimes 2}_{hC_2}
	$$
	and we can remove the parentheses since $\Sigma^\infty$ is symmetric monoidal and colimit-preserving by lemma \ref{forgetmonoidal}.
\end{proof}

\begin{rem}
	\label{Hopfexplicitp2}
	Using remarks \ref{Hopfexplicitcoalg} and \ref{Hopfexplicitlift}, the composite of the James--Hopf map with the norm map
	$$
	\Omega^\infty\Sigma^\infty X \xlongrightarrow{H^\infty} \Omega^\infty\Sigma^\infty X^{\otimes 2}_{hC_2} \xlongrightarrow{\mathrm{nm}} \Omega^\infty\Sigma^\infty (X^{\otimes 2})^{hC_2}
	$$
	induces a map
	$$
	[\Sigma^\infty T, \Sigma^\infty X] \longrightarrow [\Sigma^\infty T, \Sigma^\infty (X^{\otimes 2})^{hC_2}]
	$$
	given by
	$$
	\alpha \longmapsto (\alpha \otimes \alpha)^{hC_2} \circ L_T - L_X \circ \alpha
	$$
	where $L_T$ and $L_X$ denote the metastable structures. We will return to this later.
\end{rem}

\begin{rem}
	\label{Hopfgoodwillie}
	Let $X$ be a pointed space. Evaluating the canonical resolution at $\Sigma^\infty_2X$ gives a fiber sequence
	$$
	\Sigma^\infty_2X \xlongrightarrow{E^\infty} \Omega^\infty\Sigma^\infty\Sigma^\infty_2X \xlongrightarrow{H^\infty} \Omega^\infty\Sigma^\infty(\Sigma^\infty_2X)^{\otimes 2}_{hC_2}
	$$
	in $P_2\mathcal{S}_*$. Applying the right adjoint $\Omega^\infty_2$ yields a fiber sequence
	$$
	\Omega^\infty_2\Sigma^\infty_2X \longrightarrow \Omega^\infty_2\Omega^\infty\Sigma^\infty\Sigma^\infty_2X \longrightarrow \Omega^\infty_2\Omega^\infty\Sigma^\infty(\Sigma^\infty_2X)^{\otimes 2}_{hC_2}
	$$
	in $\mathcal{S}_*$. Using the equivalence $\Sigma^\infty \circ \Sigma^\infty_2 \simeq \Sigma^\infty$, this identifies with
	$$
	\Omega^\infty_2\Sigma^\infty_2X \longrightarrow \Omega^\infty\Sigma^\infty X \longrightarrow \Omega^\infty\Sigma^\infty X^{\otimes 2}_{hC_2}
	$$
	On the other hand, $\Sigma^\infty_2$ exhibits $P_2\mathcal{S}_*$ as the 2-excisive approximation of $\mathcal{S}_*$ by \cite[ex.~3.1]{Gijsapprox}, so the unit of the adjunction $\Sigma^\infty_2 \dashv \Omega^\infty_2$ corresponds to the natural transformation $I \to P_2I$. We therefore recover the standard fiber sequence
	$$
	P_2I(X) \longrightarrow P_1I(X) \simeq \Omega^\infty\Sigma^\infty X \longrightarrow BD_2I(X) \simeq \Omega^\infty\Sigma^\infty X^{\otimes 2}_{hC_2}
	$$
	from Goodwillie calculus. In particular, our James--Hopf map $H^\infty$ agrees with the usual James--Hopf map, also called the stable Hopf map.
\end{rem}

We conclude by describing the canonical resolution for metastable spheres.

\begin{ex}
	For any sphere $S^{n, k}$ in $P_2\mathcal{S}_*$ there is a fiber sequence
	$$
	S^{n, k} \longrightarrow \Omega^\infty \mathbb{S}^n \longrightarrow \Omega^\infty \Sigma^{n}\mathbb{RP}^\infty_n
	$$
	This follows from the canonical resolution, since the underlying spectrum of $S^{n,k}$ is $\mathbb{S}^n$ and
	$$
	(S^n)^{\otimes 2}_{hC_2} \simeq \Sigma^n S^{n\rho}_{hC_2} \simeq \Sigma^nRP^\infty_n
	$$
	The dependence on $k$ when $n = 0$ is encoded in the James--Hopf map.
\end{ex}

\subsection{The metastable EHP sequence}

In this subsection we study the fiber of the $n$-fold suspension map in the metastable category. More precisely, we show that every object $X$ of $P_2\mathcal{S}_*$ fits into a fiber sequence
$$
X \xlongrightarrow{E^n} \Omega^n\Sigma^n X \xlongrightarrow{H^n} \Omega^\infty\Sigma^\infty(S(n\rho)_+ \otimes X^{\otimes 2})_{hC_2}
$$

This recovers a result of Milgram. Specializing to $n = 1$, we obtain the metastable \emph{EHP sequence}
$$
X \xlongrightarrow{E} \Omega\Sigma X \xlongrightarrow{H} \Omega^\infty\Sigma^\infty X^{\otimes 2}
$$
which we will later use to construct a spectral sequence for computing homotopy groups in the metastable category, and in particular the homotopy groups of spheres.

Specializing instead to $n = 2$, we obtain a description of the fiber of the double suspension
$$
X \xlongrightarrow{E^2} \Omega^2\Sigma^2 X \xlongrightarrow{H^2} \Omega^\infty\Sigma^\infty X^{\otimes 2}/(1 - \tau)
$$
in terms of the swap map $\tau$. We will use this later to prove an exponent theorem.

To state the main result, recall that $\rho$ denotes the sign representation. We write $S(V)$ for the unit sphere in a representation $V$, and $S^V$ for its one-point compactification. Thus $S(V)$ has dimension one less than $V$, whereas $S^V$ has the same dimension as $V$. These are related by the cofiber sequence $$S(V)_+ \longrightarrow S^0 \longrightarrow S^V$$
We are now ready to prove the main result of this subsection.

\begin{lem}
	\label{intermediateres}
	For any object $X$ of $P_2\mathcal{S}_*$ there is a fiber sequence
	$$
	X \xlongrightarrow{E^n} \Omega^n\Sigma^n X \xlongrightarrow{H^n} \Omega^\infty\Sigma^\infty(S(n\rho)_+ \otimes X^{\otimes 2})_{hC_2}
	$$
\end{lem}

\begin{proof}
	By the canonical resolution, every object $Y$ of $P_2\mathcal{S}_*$ fits into a fiber sequence
	$$
	Y \longrightarrow \Omega^\infty\Sigma^\infty Y \longrightarrow \Omega^\infty\Sigma^\infty Y^{\otimes 2}_{hC_2}
	$$
	Applying this to $Y = \Sigma^n X$ gives a fiber sequence
	$$
	\Sigma^n X \longrightarrow \Omega^\infty\Sigma^{\infty+n} X \longrightarrow \Omega^\infty\Sigma^{\infty+n}(S^{n\rho} \otimes X^{\otimes 2})_{hC_2}
	$$
	where we use the $C_2$-equivariant equivalence $(S^n)^{\otimes 2} \simeq \Sigma^n S^{n\rho}$. Applying $\Omega^n$ yields a fiber sequence
	$$
	\Omega^n\Sigma^n X \longrightarrow \Omega^\infty\Sigma^\infty X \longrightarrow \Omega^\infty\Sigma^\infty(S^{n\rho} \otimes X^{\otimes 2})_{hC_2}
	$$
	Consider the cofiber sequence
	$$
	S(n\rho)_+ \longrightarrow S^0 \longrightarrow S^{n\rho}
	$$
	Smashing with $X^{\otimes 2}$, taking homotopy orbits, and applying $\Omega^\infty\Sigma^\infty$ gives a diagram of fiber sequences
\[\begin{tikzcd}
	{\Omega^\infty\Sigma^\infty X} & {\Omega^\infty\Sigma^\infty X} & {*} \\
	{\Omega^\infty\Sigma^\infty X^{\otimes 2}_{hC_2}} & {\Omega^\infty\Sigma^\infty (S^{n\rho} \otimes X^{\otimes 2})_{hC_2}} & {\Omega^\infty\Sigma^\infty (\Sigma S(n\rho)_+ \otimes X^{\otimes 2})_{hC_2}}
	\arrow[equals, from=1-1, to=1-2]
	\arrow[from=1-1, to=2-1]
	\arrow[from=1-2, to=1-3]
	\arrow[from=1-2, to=2-2]
	\arrow[from=1-3, to=2-3]
	\arrow[from=2-1, to=2-2]
	\arrow[from=2-2, to=2-3]
\end{tikzcd}\]
	The claimed fiber sequence is obtained by taking vertical fibers in this diagram.
\end{proof}

\begin{rem}
	For $n \to \infty$, the lemma recovers the canonical resolution, since $\mathrm{colim}_n S(n\rho) \simeq *$.
\end{rem}

\begin{rem}
	In \cite[thm. 1.11]{Milgram}, Milgram proves that for a space $X$ in the metastable range, the fiber $F_n$ of the map $X \to \Omega^n\Sigma^nX$ is equivalent to
	$$
	F_n \simeq \Omega(S^{n-1}\ltimes_{C_2} X^{\wedge 2})
	$$
	where $C_2$ acts on $S^{n-1}$ by the antipodal action, the orbits are strict, and $\ltimes$ denotes the half-smash product. This follows from our result. Indeed, our description gives
	$$
	F_n \simeq \Omega^{\infty+1}\Sigma^\infty(S(n\rho)_+ \otimes X^{\otimes 2})_{hC_2}
	$$
	and the claim follows since $S(n\rho) \simeq S^{n-1}$ with the antipodal action, $Y_+ \wedge Z \simeq Y \ltimes Z$, and homotopy orbits agree with strict orbits because the action is free. Moreover, since $X$ is in the metastable range, the space $X^{\wedge 2}$ lies in the stable range, so it agrees with its suspension spectrum. 
\end{rem}

\begin{ex}
	\label{msuspsn}
	Taking $X = S^n$ and replacing $n$ by $m$ in the previous lemma gives a fiber sequence
	$$
	S^n \longrightarrow \Omega^m S^{m+n} \longrightarrow \Omega^\infty\Sigma^{n}\mathbb{RP}_n^{m+n-1}
	$$
	Indeed, by the previous lemma the rightmost term is equivalent to
	$$
	\Omega^\infty\Sigma^\infty(S(m\rho)_+ \otimes X^{\otimes 2})_{hC_2} \simeq \Omega^\infty\Sigma^{\infty + n}(S(m\rho)_+ \otimes S^{n\rho})_{hC_2}
	$$
	so it suffices to show that
	$$
	(S(m\rho)_+ \otimes S^{n\rho})_{hC_2} \simeq RP^{m+n-1}_n
	$$
	This follows by passing to orbits in the cofiber sequence
	$$
	S(n\rho)_+ \longrightarrow S((m+n)\rho)_+ \longrightarrow S^{n\rho} \otimes S(m\rho)_+
	$$
	which is a special case of the cofiber sequence
	$$
	S(U)_+ \longrightarrow S(U \oplus V)_+ \longrightarrow S^U \otimes S(V)_+
	$$
	for representations $U$ and $V$.
\end{ex}

As a special case of the previous lemma, we obtain the EHP sequence. In this sense, the following lemma makes precise the slogan that in the metastable range there is always an EHP sequence. 

\begin{lem}
	\label{EHP}
	For any object $X$ of $P_2\mathcal{S}_*$ there is a fiber sequence
	$$
	X \xlongrightarrow{E} \Omega\Sigma X \xlongrightarrow{H} \Omega^\infty\Sigma^\infty X^{\otimes 2}
	$$
\end{lem}

\begin{proof}
	This follows from the previous lemma, since $S(\rho) \simeq C_2$ is a free $C_2$-space.
\end{proof}

\begin{rem}
	Using the identification of $P_2\mathcal{S}_*$ as the $2$-excisive approximation of $\mathcal{S}_*$, exactly as in remark \ref{Hopfgoodwillie}, we obtain a fiber sequence
	$$
	P_2I(X) \longrightarrow \Omega P_2I(\Sigma X) \longrightarrow P_1I(X^{\otimes 2})
	$$
	This fiber sequence appears in the work of Behrens \cite[cor. 2.1.4, $m=1$]{Behrens2010} when $X$ is a sphere. His general statement, for $m \geq 1$, only applies to spheres, whereas the special case $m = 1$ holds for any pointed space.
\end{rem}

We now use the EHP sequence to construct a spectral sequence for computing homotopy groups in the metastable category. Before doing so, we need to define these homotopy groups.

For an object $X$ of $P_2\mathcal{S}_*$, define the $n$th homotopy group of $X$ by $$\pi_nX := [S^n, X]$$
as in any symmetric monoidal category, since $S^n \simeq \Sigma^nS^0 \simeq \Sigma^n\mathbb{1}$. If $X$ is a pointed space, then the homotopy groups of the metastable spectrum $\Sigma^\infty_2X$ coincide with those of the space $P_2I(X)$. Indeed,
$$
\pi_n \Sigma^\infty_2 X = [S^n, \Sigma^\infty_2 X] = [\Sigma^\infty_2S^n, \Sigma^\infty_2 X] = [S^n, \Omega^\infty_2\Sigma^\infty_2 X] = \pi_nP_2I(X)
$$
Note also that $\pi_nX$ is always an abelian group for $n \geq 2$, since $S^n$ is a suspension. On the other hand, $\pi_1X$ is in general a non-abelian group, while $\pi_0X$ is only a set, exactly as in the case of spaces. An example of non-abelian fundamental group is given by \cite[thm. 6.8]{nilpotent}
$$
\pi_1(S^1 \vee S^1) = \pi_1P_2I(S^1 \vee S^1) \cong F_2/\Gamma_3F_2 \cong \mathrm{Heis}(\mathbb{Z})
$$
where $F_2 := \mathbb{Z} * \mathbb{Z}$ is the free group on two generators, $\Gamma_n$ denotes the lower central series, and $\mathrm{Heis}(\mathbb{Z})$ is the discrete Heisenberg group, equivalently the free $2$-step nilpotent group on two generators.

Our main computational tool for homotopy groups is the following spectral sequence.

\begin{lem}
	\label{metSS}
	For any object $X$ of $P_2\mathcal{S}_*$, there is a spectral sequence with $E^1$-page
	$$
	E^1_{t,m}:=
	\begin{cases}
		\pi_tX & m=0 \\
		\pi_{t-m+1}^sX^{\otimes 2} & 0<m\leq n \\
		0 & m>n
	\end{cases}
	$$
	with differentials of bidegree $(-1,-r)$
	\[
	d^r:E^r_{t,m}\longrightarrow E^r_{t-1,m-r}
	\]
	and converging to $\pi_{n+t}\Sigma^nX$.
\end{lem}

\begin{proof}
	Consider the filtration of $\Omega^n\Sigma^nX$ induced by the successive suspension maps
\[\begin{tikzcd}
	X & {\Omega \Sigma X} & {\Omega^2 \Sigma^2 X} & {\Omega^3\Sigma^3X } & \cdots & {\Omega^n\Sigma^nX} \\
	& {\Omega^\infty \Sigma^\infty X^{\otimes 2}} & {\Omega^\infty \Sigma^{\infty+1} X^{\otimes 2}} & {\Omega^\infty \Sigma^{\infty+2} X^{\otimes 2}}
	\arrow["E", from=1-1, to=1-2]
	\arrow["E", from=1-2, to=1-3]
	\arrow["H", from=1-2, to=2-2]
	\arrow["E", from=1-3, to=1-4]
	\arrow["H", from=1-3, to=2-3]
	\arrow["E", from=1-4, to=1-5]
	\arrow["H", from=1-4, to=2-4]
	\arrow["E", from=1-5, to=1-6]
\end{tikzcd}\]
	Every corner in this diagram is a fiber sequence by the EHP sequence. Applying $\pi_0$ to this filtration gives a spectral sequence, and the description of the $E^1$-page is immediate from the fibers.
\end{proof}

Concretely, the first row of this spectral sequence is given by $\pi_*X$, while every higher row is given by $\pi_*^sX^{\otimes 2}$, shifted one step to the right each time. If $X$ is connective, the spectral sequence takes the following form:

\[\begin{tikzcd}[cramped]
	&& 0 & 1 & 2 & \cdots & t & \cdots & \\
	& 0 & {\pi_0X} & {\pi_1X} & {\pi_2X} & \cdots & {\pi_tX} & \cdots \\
	& 1 & {\pi_0^sX^{\otimes 2}} & {\pi_1^sX^{\otimes 2}} & {\pi_2^sX^{\otimes 2}} & \cdots & {\pi_t^sX^{\otimes 2}} & \cdots \\
	& 2 && {\pi_0^sX^{\otimes 2}} & {\pi_1^sX^{\otimes 2}} & \cdots & {\pi_{t-1}^sX^{\otimes 2}} & \cdots \\
	& 3 &&& {\pi_0^sX^{\otimes 2}} & \cdots & {\pi_{t-2}^sX^{\otimes 2}} & \cdots \\
	& \cdots &&&& \cdots & \cdots & \cdots \\
	{} & \implies & {\pi_n\Sigma^nX} & {\pi_{n+1}\Sigma^nX} & {\pi_{n+2}\Sigma^nX} & \cdots & {\pi_{n+t}\Sigma^nX} & \cdots & {}
	\arrow[dashed, from=3-5, to=2-4]
	\arrow[dashed, from=4-4, to=2-3]
	\arrow[dashed, from=5-5, to=3-4]
	\arrow[shift left=5, no head, from=7-1, to=7-9]
\end{tikzcd}\]
If $X$ is connective, we may also let $n \to \infty$ and still obtain a convergent spectral sequence, since in this case there is a vanishing line.

The differentials are easy to describe. Since the filtration is induced by filtering $X^{\otimes 2}_{hC_2}$ by the skeleta of $BC_{2+}$, the differentials between positive rows, that is, between rows with $m \neq 0$, are exactly the differentials in the homotopy orbits spectral sequence for $X^{\otimes 2}_{hC_2}$. In particular, these are stable differentials. The only genuinely metastable differentials are therefore those landing in the $0$th row.

\begin{ex}
	\label{metSSspheres}
	For $X = S^0$, the spectral sequence has $E^1$-page
	$$
	E^1_{t,m}:=
	\begin{cases}
		\pi_tS^0 & m=0 \\
		\pi_{t-m+1}^s & 0<m\leq n \\
		0 & m>n
	\end{cases}
	$$
	and converges to $\pi_{n+t}S^n$.

	At this point, however, we do not yet know the homotopy groups of $S^0$, so we cannot write down the entire $E^1$-page explicitly. We will compute them in the next subsection. The remaining rows are already determined: they coincide with the homotopy orbits spectral sequence, equivalently the Atiyah--Hirzebruch spectral sequence, for the spectrum $\mathbb{RP}^\infty_+$.
\end{ex}

Besides the case $n=1$, another special case of the general description is $n=2$. This gives the following explicit description of the fiber of the double suspension.

\begin{lem}
	\label{doublesuspp2}
	For any object $X$ of $P_2\mathcal{S}_*$ there is a fiber sequence
	$$
	X \xlongrightarrow{E^2} \Omega^2\Sigma^2X \xlongrightarrow{H^2} \Omega^\infty\Sigma^\infty X^{\otimes 2}/(1 - \tau)
	$$
	where $\tau$ denotes the swap map on $X^{\otimes 2}$.
\end{lem}

\begin{proof}
	We first claim that for any $C_2$-space $Y$ there is an equivalence
\[\begin{tikzcd}
	{(S(2\rho)_+\otimes Y)_{hC_2}\simeq \mathrm{coeq}(Y } & {Y)}
	\arrow["1", shift left, from=1-1, to=1-2]
	\arrow["\sigma"', shift right, from=1-1, to=1-2]
\end{tikzcd}\]
	where $\sigma$ denotes the automorphism of $Y$ induced by the $C_2$-action. Once this is proved, the lemma follows immediately, since stable coequalizers are computed by cofibers.

	To prove the claim, note that $S(2\rho)\simeq S^1$ with the antipodal action. As a $C_2$-complex, it has one free $0$-cell $C_{2+}$ and one free $1$-cell $C_{2+}\times D^1$. Its attaching map
	$$
	C_{2+}\sqcup C_{2+}=C_{2+}\times S^0 \longrightarrow C_{2+}
	$$
	is given by the identity on one component and by the swap map on the other. Thus we have a pushout
\[\begin{tikzcd}
	{C_{2+} \sqcup C_{2+}} & {C_{2+} \times S^0} & {C_{2+}} \\
	{C_{2+}} & {C_{2+}\times D^1} & {S(2\rho)_+}
	\arrow[equals, from=1-1, to=1-2]
	\arrow["{(1, \sigma)}", curve={height=-18pt}, from=1-1, to=1-3]
	\arrow["\nabla"', from=1-1, to=2-1]
	\arrow[from=1-2, to=1-3]
	\arrow[from=1-2, to=2-2]
	\arrow[from=1-3, to=2-3]
	\arrow[equals, from=2-1, to=2-2]
	\arrow[from=2-2, to=2-3]
	\arrow["\lrcorner"{anchor=center, pos=0.125, rotate=180}, draw=none, from=2-3, to=1-2]
\end{tikzcd}\]
	Smashing with $Y$ and taking homotopy orbits gives
\[\begin{tikzcd}
	{Y \sqcup Y \simeq ((C_{2+} \sqcup C_{2+}) \otimes Y)_{hC_2}} & {(C_{2+} \otimes Y)_{hC_2} \simeq Y} \\
	{Y \simeq (C_{2+} \otimes Y)_{hC_2}} & {(S(2\rho)_+ \otimes Y)_{hC_2}}
	\arrow[""{name=0, anchor=center, inner sep=0}, "{(1, \sigma)}", from=1-1, to=1-2]
	\arrow["\nabla"', from=1-1, to=2-1]
	\arrow[from=1-2, to=2-2]
	\arrow[from=2-1, to=2-2]
	\arrow["\lrcorner"{anchor=center, pos=0.125, rotate=180}, draw=none, from=2-2, to=0]
\end{tikzcd}\]
	which is precisely the coequalizer diagram of $(1,\sigma)$ on $Y$.
\end{proof}

\begin{ex}
	If $X = S^n$, then the swap map $\tau$ acts on $\mathbb{S}^{2n}$ by the identity when $n$ is even, and by $-1$ when $n$ is odd. Hence, if $n$ is even, we obtain a fiber sequence
	$$
	S^n \longrightarrow \Omega^2S^{n+2} \longrightarrow \Omega^\infty(\mathbb{S}^{2n} \oplus \mathbb{S}^{2n+1})
	$$
	whereas if $n$ is odd, we obtain a fiber sequence
	$$
	S^n \longrightarrow \Omega^2S^{n+2} \longrightarrow \Omega^\infty\mathbb{S}^{2n}/2
	$$
	This agrees with the description of example \ref{msuspsn}, since
	$$
	RP_n^{n+1} \simeq
	\begin{cases}
		S^n \oplus S^{n+1} & n \text{ even} \\
		S^n/2 & n \text{ odd}
	\end{cases}
	$$
\end{ex}

We conclude the subsection by using this fiber sequence to obtain an exponent result.

\begin{lem}
	\label{exponentp2}
	Let $X$ be an object of $P_2\mathcal{S}_*$ such that the swap map on $X^{\otimes 2}$ is given by $-1$. Then multiplication by $4^n$ on $\Omega^{2n}\Sigma^{2n}X$ factors as
\[\begin{tikzcd}
	& X \\
	{\Omega^{2n}\Sigma^{2n}X} & {\Omega^{2n}\Sigma^{2n}X}
	\arrow[from=1-2, to=2-2]
	\arrow[dashed, from=2-1, to=1-2]
	\arrow["{4^n}"', from=2-1, to=2-2]
\end{tikzcd}\]
\end{lem}

\begin{proof}
	We argue by induction on $n$. It is enough to construct the dotted map in the diagram
\[\begin{tikzcd}
	&& X \\
	& {\Omega^{2n-2}\Sigma^{2n-2}X} & {\Omega^{2n-2}\Sigma^{2n-2}X} \\
	{\Omega^{2n}\Sigma^{2n}X} & {\Omega^{2n}\Sigma^{2n}X} & {\Omega^{2n}\Sigma^{2n}X}
	\arrow[from=1-3, to=2-3]
	\arrow[from=2-2, to=1-3]
	\arrow["{4^{n-1}}", from=2-2, to=2-3]
	\arrow[from=2-2, to=3-2]
	\arrow[from=2-3, to=3-3]
	\arrow[dashed, from=3-1, to=2-2]
	\arrow["4"', from=3-1, to=3-2]
	\arrow["{4^{n-1}}"', from=3-2, to=3-3]
\end{tikzcd}\]
	By assumption, the swap map on $X^{\otimes 2}$ is $-1$. The same is therefore true for $(\Sigma^{2n-2}X)^{\otimes 2}$, since the swap map on $(S^{2k})^{\otimes 2}$ is the identity for every $k$. Applying the previous lemma to $\Sigma^{2n-2}X$, we obtain a fiber sequence
	$$
	\Omega^{2n-2}\Sigma^{2n-2}X \xlongrightarrow{E^2} \Omega^{2n}\Sigma^{2n}X \xlongrightarrow{H^2} \Omega^\infty\Sigma^{\infty+2n-2}X^{\otimes 2}/2
	$$

	Thus it suffices to show that the composite
	$$
	\Omega^{2n}\Sigma^{2n}X \xlongrightarrow{4} \Omega^{2n}\Sigma^{2n}X \xlongrightarrow{H^2} \Omega^\infty\Sigma^{\infty+2n-2}X^{\otimes 2}/2
	$$
	is null. Since the target is an infinite loop space, this is equivalent to asking that the composite
	$$
	\Omega^{2n}\Sigma^{2n}X \xlongrightarrow{H^2} \Omega^\infty\Sigma^{\infty+2n-2}X^{\otimes 2}/2 \xlongrightarrow{4} \Omega^\infty\Sigma^{\infty+2n-2}X^{\otimes 2}/2
	$$
	be null. But for any spectrum $Y$, multiplication by $4$ on $Y/2$ is null. This gives the dotted arrow.
\end{proof}

We thus obtain a metastable analogue of James' exponent theorem \cite{EHP}; see also \cite[cor. 4.5.3]{unstablemethods}.

\begin{lem}
	Multiplication by $4^n$ on $\Omega^{2n}S^{2n+1}$ factors as
\[\begin{tikzcd}
	& {S^1} \\
	{\Omega^{2n}S^{2n+1}} & {\Omega^{2n}S^{2n+1}}
	\arrow[from=1-2, to=2-2]
	\arrow[dashed, from=2-1, to=1-2]
	\arrow["{4^n}"', from=2-1, to=2-2]
\end{tikzcd}\]
\end{lem}

\begin{proof}
	This is the previous lemma applied to $X = S^1$, since the swap map on $\mathbb{S}^1 \otimes \mathbb{S}^1$ is $-1$.
\end{proof}

\begin{rem}
	Unlike in the classical case, this does not imply that multiplication by $4^n$ is null on $\Omega^{2n+2}S^{2n+1}$. The reason is that in the metastable category one has $\Omega^2S^1 \not\simeq *$, so in particular the base case $n=0$ already fails.
\end{rem}

\subsection{The 0-sphere}

In this subsection we compute the homotopy groups of $S^0$, and more generally of the exotic spheres $S^{0,k}$. Our main result is the equivalence
$$
\Omega S^{0,k} \simeq \Omega^{\infty+2}(\mathbb{RP}^\infty_{-1} \oplus \mathbb{S}/k)
$$
From this we obtain all positive homotopy groups and deduce that the exotic spheres $S^{0,k}$ are pairwise distinct. We then determine all maps
$$
S^{0,j} \longrightarrow S^{0,k}
$$
showing that there are two such maps if $j \nmid k$, and four if $j \mid k$. In particular, $\pi_0S^0$ has four elements: the null map, the identity, a nontrivial automorphism, and a nontrivial map lifting the null map.

The main tool is the James--Hopf map $H^\infty$. While this paper was being written, a paper of Klein appeared \cite{kleinnew} in which some of the properties we need are established explicitly. We had been using these properties implicitly throughout, and it seems preferable to isolate and organize them here. This also provides a different perspective on Klein's result.

We begin by fixing some notation.

For a spectrum $Y$ with $C_2$-action, we write $i \colon Y^{hC_2} \to Y$ for the forgetful map, which may also be obtained by taking fixed points of the diagonal map $Y \to Y \times Y = C_{2+} \otimes Y$. Similarly, we write $\pi \colon Y \to Y_{hC_2}$ for the projection map, obtained by passing to orbits of the fold map $C_{2+} \otimes Y = Y \vee Y \to Y$. We denote by $\mathrm{tr} \colon Y_{hC_2} \to Y$ the transfer map, defined as the composite $i \circ \mathrm{nm}$, where $\mathrm{nm}$ is the norm map. Equivalently, $\mathrm{tr}$ is obtained by passing to orbits of the diagonal map $Y \to C_{2+} \otimes Y$, as in the diagram
\[\begin{tikzcd}
	{Y_{hC_2}} & {(C_{2+}\otimes Y)_{hC_2}} \\
	{Y^{hC_2}} & {(C_{2+}\otimes Y)^{hC_2}}
	\arrow[from=1-1, to=1-2]
	\arrow["{\mathrm{nm}}"', from=1-1, to=2-1]
	\arrow["{\mathrm{nm}}", equals, from=1-2, to=2-2]
	\arrow["i"', from=2-1, to=2-2]
\end{tikzcd}\]
Recall also that the composite $\mathrm{tr} \circ \pi \colon Y \to Y$ is given by $1+\sigma$, where $\sigma$ denotes the automorphism of $Y$ induced by the $C_2$-action. Indeed, $\mathrm{tr} \circ \pi = i \circ \mathrm{nm} \circ \pi = 1+\sigma$ by the defining property of the norm map. In particular, if the action is trivial, then $\mathrm{tr} \circ \pi$ is multiplication by $2$.

For an object $X$ of $P_2\mathcal{S}_*$, we write $L_X \colon X \to (X^{\otimes 2})^{hC_2}$ for the metastable structure, and $\Delta_X := i \circ L_X \colon X \to X^{\otimes 2}$ for the induced diagonal map.

With this notation in place, we can state and prove the main results of Klein \cite[thm. A, cor. C]{kleinnew}. 

\begin{lem}
	Let $X$ and $Y$ be objects of $P_2\mathcal{S}_*$. Then the James--Hopf map
	$$
	H^\infty \colon [\Sigma^\infty X, \Sigma^\infty Y] \longrightarrow [\Sigma^\infty X, \Sigma^\infty Y^{\otimes 2}_{hC_2}]
	$$
	satisfies the following properties:
	\begin{enumerate}
		\item $H^\infty(\alpha) = 0$ if and only if $\alpha \in \mathrm{im}\,E^\infty$
		\item $\mathrm{nm}\,H^\infty(\alpha) = (\alpha \otimes \alpha)^{hC_2} \circ L_X - L_Y \circ \alpha$
		\item $\mathrm{tr}\,H^\infty(\alpha) = (\alpha \otimes \alpha) \circ \Delta_X - \Delta_Y \circ \alpha$
		\item $H^\infty(\alpha + \beta) = H^\infty(\alpha) + H^\infty(\beta) + \pi \circ (\alpha \otimes \beta) \circ \Delta_X$
		\item $H^\infty(\beta \circ \alpha) = H^\infty(\beta) \circ \alpha + (\beta \otimes \beta)_{hC_2} \circ H^\infty(\alpha)$
	\end{enumerate}
\end{lem}

\begin{proof}
	We verify the five assertions in order.
	\begin{enumerate}
		\item The first claim is immediate from the canonical fiber sequence
		$$
		X \xlongrightarrow{E^\infty} \Omega^\infty\Sigma^\infty X \xlongrightarrow{H^\infty} \Omega^\infty\Sigma^\infty X^{\otimes 2}_{hC_2}
		$$

		\item Consider the composite
		$$
		\Omega^\infty\Sigma^\infty X \xlongrightarrow{H^\infty} \Omega^\infty\Sigma^\infty X^{\otimes 2}_{hC_2} \xlongrightarrow{\mathrm{nm}} \Omega^\infty\Sigma^\infty (X^{\otimes 2})^{hC_2}
		$$
		By remark \ref{Hopfexplicitp2}, this induces on homotopy classes the map
		$$
		\alpha \longmapsto (\alpha \otimes \alpha)^{hC_2} \circ L_X - L_Y \circ \alpha
		$$
		which is exactly the desired formula.

		\item This follows immediately from (2). Indeed, using $\mathrm{tr} = i \circ \mathrm{nm}$ and $\Delta = i \circ L$, together with the naturality identity
		$$
		i \circ (\alpha \otimes \alpha)^{hC_2} \simeq (\alpha \otimes \alpha) \circ i
		$$
		we obtain
		$$
		\mathrm{tr}\,H^\infty(\alpha)
		= i \circ \mathrm{nm}\,H^\infty(\alpha)
		= i \circ (\alpha \otimes \alpha)^{hC_2} \circ L_X - i \circ L_Y \circ \alpha
		= (\alpha \otimes \alpha) \circ \Delta_X - \Delta_Y \circ \alpha
		$$
		as claimed.
		
		\item By (2), we compute
		\begin{align*}
			\mathrm{nm}\,H^\infty(\alpha + \beta)
			&= ((\alpha + \beta)^{\otimes 2})^{hC_2} \circ L_X - L_Y \circ (\alpha + \beta) \\
			&= (\alpha^{\otimes 2})^{hC_2} \circ L_X + (\beta^{\otimes 2})^{hC_2} \circ L_X + (\alpha \otimes \beta + \beta \otimes \alpha)^{hC_2} \circ L_X - L_Y \circ \alpha - L_Y \circ \beta \\
			&= \mathrm{nm}\,H^\infty(\alpha) + \mathrm{nm}\,H^\infty(\beta) + (\alpha \otimes \beta + \beta \otimes \alpha)^{hC_2} \circ L_X
		\end{align*}
		It therefore remains to identify the last term.
		
		Note that the map $\alpha \otimes \beta + \beta \otimes \alpha$ is given by the composite
		$$
		X^{\otimes 2} \xlongrightarrow{(\alpha \otimes \beta,\beta \otimes \alpha)} C_{2+} \otimes Y^{\otimes 2} \xlongrightarrow{\nabla} Y^{\otimes 2}
		$$
		Applying the norm map to this, we obtain the diagram
\[\begin{tikzcd}
	{(X^{\otimes 2})_{hC_2}} && {Y^{\otimes 2}} & {(Y^{\otimes 2})_{hC_2}} \\
	{(X^{\otimes 2})^{hC_2}} && {Y^{\otimes 2}} & {(Y^{\otimes 2})^{hC_2}}
	\arrow["{(\alpha\otimes\beta, \beta\otimes\alpha)_{hC_2}}", from=1-1, to=1-3]
	\arrow["{\mathrm{nm}}"', from=1-1, to=2-1]
	\arrow["\pi", from=1-3, to=1-4]
	\arrow["{\mathrm{nm}}"', equals, from=1-3, to=2-3]
	\arrow["{\mathrm{nm}}"', from=1-4, to=2-4]
	\arrow["{(\alpha\otimes\beta, \beta\otimes\alpha)^{hC_2}}"', from=2-1, to=2-3]
	\arrow[from=2-3, to=2-4]
\end{tikzcd}\]
		Hence
		$$
		(\alpha \otimes \beta + \beta \otimes \alpha)^{hC_2} \simeq \mathrm{nm} \circ \pi \circ (\alpha \otimes \beta,\beta \otimes \alpha)^{hC_2}
		$$
		To conclude, consider the diagram
\[\begin{tikzcd}
	{(X^{\otimes 2})^{hC_2}} && {(C_{2+} \otimes Y^{\otimes 2})^{hC_2}} & {Y^{\otimes 2}} \\
	{X^{\otimes 2}} && {C_{2+} \otimes Y^{\otimes 2}}
	\arrow["{(\alpha\otimes\beta, \beta\otimes\alpha)^{hC_2}}", from=1-1, to=1-3]
	\arrow["i"', from=1-1, to=2-1]
	\arrow[equals, from=1-3, to=1-4]
	\arrow["i"', from=1-3, to=2-3]
	\arrow["{(\alpha\otimes\beta, \beta\otimes\alpha)}"', from=2-1, to=2-3]
	\arrow["{\pi_1}"', from=2-3, to=1-4]
\end{tikzcd}\]
		which implies
		$$
		(\alpha \otimes \beta,\beta \otimes \alpha)^{hC_2} \circ L_X \simeq (\alpha \otimes \beta) \circ i \circ L_X = (\alpha \otimes \beta) \circ \Delta_X
		$$
		Putting everything together, we obtain
		$$
		\mathrm{nm}\,H^\infty(\alpha + \beta)
		=
		\mathrm{nm}\,H^\infty(\alpha) + \mathrm{nm}\,H^\infty(\beta) + \mathrm{nm} \circ \pi \circ (\alpha \otimes \beta) \circ \Delta_X
		$$
Since the two maps have the same image under the norm map, it remains only to compare the specified nullhomotopies after passage to the Tate power. These nullhomotopies are obtained functorially from the lift data, and the same calculation as above applies after passage to the Tate power. Hence the two lifts coincide.

Equivalently, one can avoid the bookkeeping of the homotopies by passing to the section model. By lemma \ref{p2equiv}, there is an equivalence
\[
P_2\mathcal{S}_* \simeq \operatorname{Sect}\bigl((-^{\otimes 2})^{C_2}\to 1\bigr)
\]
Under this equivalence, $H^\infty$ identifies with the corresponding map for sections, and the formula above is exactly the identity proved in \cite[thm. A(ii)]{kleinnew}.

	\item Let $\alpha \colon \Sigma^\infty X \to \Sigma^\infty Y$ and $\beta \colon \Sigma^\infty Y \to \Sigma^\infty Z$. Applying the norm map to $H^\infty(\beta \circ \alpha)$, we compute
\begin{align*}
	\mathrm{nm}\, H^\infty(\beta \circ \alpha)
	&= ((\beta \circ \alpha) \otimes (\beta \circ \alpha))^{hC_2} \circ L_X - L_Z \circ (\beta \circ \alpha) \\
	&= (\beta \otimes \beta)^{hC_2} \circ (\alpha \otimes \alpha)^{hC_2} \circ L_X - L_Z \circ \beta \circ \alpha \\
	&= (\beta \otimes \beta)^{hC_2} \circ (\mathrm{nm}\, H^\infty(\alpha) + L_Y \circ \alpha) - L_Z \circ \beta \circ \alpha \\
	&= (\beta \otimes \beta)^{hC_2} \circ \mathrm{nm}\, H^\infty(\alpha) + (\beta \otimes \beta)^{hC_2} \circ L_Y \circ \alpha - L_Z \circ \beta \circ \alpha \\
	&= (\beta \otimes \beta)^{hC_2} \circ \mathrm{nm}\, H^\infty(\alpha) + \mathrm{nm}\,H^\infty(\beta) \circ \alpha \\
	&= \mathrm{nm}\big((\beta \otimes \beta)_{hC_2} \circ H^\infty(\alpha) + H^\infty(\beta) \circ \alpha\big)
\end{align*}
Since the two maps have the same image under the norm map, it remains only to compare the specified nullhomotopies after passage to the Tate power. These nullhomotopies are obtained functorially from the lift data, and the same calculation as above applies after passage to the Tate power. Hence the two lifts coincide.

Equivalently, one can avoid the bookkeeping of the homotopies by passing to the section model using lemma \ref{p2equiv}, as in (4), and following the proof of \cite[thm. A(iv)]{kleinnew}. \qedhere
\end{enumerate}
\end{proof}

Now that we have established the basic properties of the James--Hopf map, we can begin the computation of the homotopy groups of $S^0$. The following is a stronger refinement of the classical Kahn--Priddy theorem \cite{KP}.

\begin{lem}
	If $X$ is a suspension in $P_2\mathcal{S}_*$, then the James--Hopf map
	$$
	H^\infty \colon [\Sigma^\infty X, \mathbb{S}] \longrightarrow [\Sigma^\infty X, \mathbb{RP}^\infty_+]
	$$
	associated to the $0$-sphere $S^0 = S^{0,1}$ is a split injection.
\end{lem}

\begin{proof}
	By (3) of the previous lemma, we have
	$$
	\mathrm{tr}\,H^\infty(\alpha) = (\alpha \otimes \alpha) \circ \Delta_X - \Delta_{S^0} \circ \alpha
	$$
	Since $X$ is a suspension, the diagonal map $\Delta_X$ is null by lemma \ref{suspensionp2}. On the other hand, $\Delta_{S^0}$ is the identity. Therefore
	$$
	\mathrm{tr}\,H^\infty(\alpha) = -\alpha
	$$
	This exhibits $-\mathrm{tr}$ as a retraction to $H^\infty$, so $H^\infty$ is a split injection.
\end{proof}

The previous lemma immediately upgrades to a statement at the level of spaces, before passing to homotopy groups.

\begin{lem}
	The composite 
	$$\Omega^{\infty + 1} \mathbb{S} \xlongrightarrow{\Omega H^\infty} \Omega^{\infty+1}\mathbb{RP}^\infty_+ \xlongrightarrow{\Omega \mathrm{tr}} \Omega^{\infty+1} \mathbb{S}$$
	is an equivalence.
\end{lem}

We can now ask what happens for the other exotic $0$-spheres. Denote by $H_k^\infty$ the James--Hopf map associated to $S^{0,k}$. If $X$ is again a suspension, then the same argument as above gives
$$
\mathrm{tr}\,H_k^\infty(\alpha)
=
(\alpha \otimes \alpha) \circ \Delta_X - \Delta_{S^{0,k}} \circ \alpha
=
-k\alpha
$$
Hence $H_k^\infty(\alpha)=0$ implies $k\alpha=0$. The converse also holds, as we now show.

To state the result, recall that the fiber of the transfer map $\mathrm{tr}$ is the spectrum $\mathbb{RP}^\infty_{-1}$, defined by
$$
\mathbb{RP}^\infty_{-1}:=\mathbb{S}^{-\rho}_{hC_2}
$$
This follows by passing to homotopy orbits in the cofiber sequence
$$
\mathbb{S}^{-\rho} \longrightarrow \mathbb{S} \xlongrightarrow{\Delta} C_{2+}\otimes \mathbb{S}
$$
which is the Spanier--Whitehead dual of the classical cofiber sequence
$$
C_{2+} \simeq S(\rho)_+ \xlongrightarrow{\nabla} S^0 \longrightarrow S^\rho
$$
See also \cite[ex.~6.3]{kuhnoverview}. The spectrum $\mathbb{RP}^\infty_{-1}$ is not a suspension spectrum.

\begin{lem}
	\label{modelS0P2}
	There is an equivalence
	$$
	\Omega S^{0,k} \simeq \Omega^{\infty+2}(\mathbb{RP}^\infty_{-1} \oplus \mathbb{S}/k)
	$$
	In particular,
	$$
	\Omega S^0 \simeq \Omega^{\infty+2}\mathbb{RP}^\infty_{-1}
	$$
\end{lem}

\begin{proof}
	As above, point (3) of the lemma and the vanishing $\Delta_{\Sigma X} = 0$ imply that
	$$
	\Omega(\mathrm{tr} \circ H^\infty) \simeq - \Omega^{\infty+1}\Delta_{S^{0,k}} = -k
	$$
	Iterating pullbacks, we obtain the diagram
\[\begin{tikzcd}
	{\Omega^{\infty+2}\mathbb{S}/k} & {\Omega^{\infty+1}\mathbb{RP}^\infty_{-1}} & {*} \\
	{\Omega^{\infty+1}\mathbb{S}} & {\Omega^{\infty+1}\mathbb{RP}^\infty_+} & {\Omega^{\infty+1}\mathbb{S}}
	\arrow[from=1-1, to=1-2]
	\arrow[from=1-1, to=2-1]
	\arrow["\lrcorner"{anchor=center, pos=0.125}, draw=none, from=1-1, to=2-2]
	\arrow[from=1-2, to=1-3]
	\arrow[from=1-2, to=2-2]
	\arrow["\lrcorner"{anchor=center, pos=0.125}, draw=none, from=1-2, to=2-3]
	\arrow[from=1-3, to=2-3]
	\arrow["{\Omega H^\infty}"', from=2-1, to=2-2]
	\arrow["{\Omega^{\infty+1}\Delta \simeq -k}"', curve={height=24pt}, from=2-1, to=2-3]
	\arrow["{\Omega\mathrm{tr}}"', from=2-2, to=2-3]
\end{tikzcd}\]
	The horizontal fiber of the left-hand pullback square therefore fits into a fiber sequence
	$$
	\Omega S^{0,k} \longrightarrow \Omega^{\infty+2}\mathbb{S}/k \longrightarrow \Omega^{\infty+1}\mathbb{RP}^\infty_{-1}
	$$
	Since $k$ is odd and $\mathbb{RP}^\infty_{-1}$ is $2$-local, the rightmost map is null. It follows that the fiber sequence splits, and hence
	$$
	\Omega S^{0,k} \simeq \Omega^{\infty+2}(\mathbb{S}/k \oplus \mathbb{RP}^\infty_{-1})
	$$
	This proves the claim.
\end{proof}

\begin{rem}
	In particular, if $j \neq k$ are positive odd integers, then $S^{0,j} \not\simeq S^{0,k}$.
\end{rem}

The previous lemma gives an explicit description of $\pi_nS^{0,k}$ for $n \geq 1$. It remains to determine $\pi_0$. For completeness, we will in fact compute all maps $S^{0,j} \to S^{0,k}$, and therefore introduce the bigraded homotopy groups
$$
\pi_{n,j}X := [S^{n,j},X]
$$
This refinement is only relevant in degree $n=0$, since $S^{n,j} \simeq S^n$ for every $n>0$. Note also that the bigrading is invisible on infinite loop spaces, since
$$
\pi_{n,j}\Omega^\infty X \cong [\Sigma^\infty S^{n,j},X] \cong [\mathbb{S}^n,X] = \pi_nX
$$
for every spectrum $X$.

\begin{lem}
	\label{homgroupsS0}
	The bigraded homotopy groups of $S^{0,k}$ are given by
	$$
	\pi_{n,j}S^{0,k} \cong
	\begin{cases}
		\mathbb{Z}/2 & \text{if } n = 0 \text{ and } j \nmid k \\
		\mathbb{Z}/2 \times \{0,k/j\} & \text{if } n = 0 \text{ and } j \mid k \\
		\pi_{n+1}\mathbb{RP}^\infty_{-1} \oplus \pi_{n+1}(\mathbb{S}/k) & \text{if } n > 0
	\end{cases}
	$$
	More precisely, the underlying map of spectra $\mathbb{S} \to \mathbb{S}$ induced by a map $S^{0,j} \to S^{0,k}$ has degree either $0$ or $k/j$, and each such map admits exactly two lifts to $P_2\mathcal S_*$.
\end{lem}

\begin{proof}
	For $n>0$, the formula follows from the previous lemma, since $S^{n,j} \simeq S^n$ for every $j$. It therefore remains to determine $\pi_{0,j}S^{0,k}$.

Recall that the canonical resolution for $S^{0,k}$ is
$$
S^{0,k} \xlongrightarrow{E^\infty} \Omega^\infty\mathbb{S} \xlongrightarrow{H_k^\infty} \Omega^\infty\mathbb{RP}^\infty_+
$$
Mapping out of $S^{0,j}$ yields an exact sequence
$$
0 \longrightarrow \pi_1\mathbb{S} \longrightarrow \pi_1\mathbb{RP}^\infty_+ \longrightarrow \pi_{0,j}S^{0,k} \longrightarrow \pi_0\mathbb{S} \xlongrightarrow{\pi_{0,j}H_k^\infty} \pi_0\mathbb{RP}^\infty_+
$$
The first map is injective by the previous lemma. Since $\pi_1\mathbb{S} \cong \mathbb{Z}/2$ and $\pi_1\mathbb{RP}^\infty_+ \cong \mathbb{Z}/2 \oplus \mathbb{Z}/2$, it follows that $\pi_{0,j}S^{0,k}$ fits into an extension
$$
0 \longrightarrow \mathbb{Z}/2 \longrightarrow \pi_{0,j}S^{0,k} \longrightarrow \ker(\pi_{0,j}H_k^\infty) \longrightarrow 0
$$
Thus it remains to determine the kernel of $\pi_{0,j}H_k^\infty$. To compute this kernel, consider the diagram
\[\begin{tikzcd}
	& {\Omega^\infty \mathbb{S}} & \\
	{\Omega^\infty \mathbb{S}} & {\Omega^\infty \mathbb{RP}^\infty_+} & {\Omega^\infty \mathbb{S}}
	\arrow["{\Omega^\infty\pi}"', from=1-2, to=2-2]
	\arrow["2", from=1-2, to=2-3]
	\arrow["{H^\infty_k}"', from=2-1, to=2-2]
	\arrow["{\Omega^\infty\mathrm{tr}}"', from=2-2, to=2-3]
\end{tikzcd}\]
Passing to $\pi_{0, j}$, we obtain

\[\begin{tikzcd}
	& {\mathbb{Z}} & \\
	{\mathbb{Z}} & {\mathbb{Z}} & {\mathbb{Z}}
	\arrow[equals, from=1-2, to=2-2]
	\arrow["2", from=1-2, to=2-3]
	\arrow["{\pi_{0,j}H^\infty_k}"', from=2-1, to=2-2]
	\arrow["{\mathrm{tr}}"', from=2-2, to=2-3]
\end{tikzcd}\]
In particular,
$$
2\pi_{0,j}H_k^\infty(\alpha)
=
\mathrm{tr}\,\pi_{0,j}H_k^\infty(\alpha)
=
(\alpha \otimes \alpha) \circ \Delta_{S^{0,j}} - \Delta_{S^{0,k}} \circ \alpha
=
\alpha(j\alpha-k)
$$
Thus, $\pi_{0,j}H_k^\infty(\alpha)=0$ if and only if either $\alpha=0$ or $j\alpha=k$. In particular,
$$
\ker(\pi_{0,j}H_k^\infty)=
\begin{cases}
\{0\} & \text{if } j \nmid k \\
\{0,k/j\} & \text{if } j \mid k
\end{cases}
$$
Since $\pi_{0,j}S^{0,k}$ fits into an extension
$$
0 \longrightarrow \mathbb{Z}/2 \longrightarrow \pi_{0,j}S^{0,k} \longrightarrow \ker(\pi_{0,j}H_k^\infty) \longrightarrow 0
$$
it follows that each element in the kernel admits exactly two lifts to $\pi_{0,j}S^{0,k}$. This gives the claimed description of $\pi_{0,j}S^{0,k}$, and in particular shows that the underlying map of spectra has degree either $0$ or $k/j$.
\end{proof}

\begin{rem}
	In particular, $\pi_0S^0$ has four elements: the null map, the identity, a nontrivial automorphism, and a nontrivial lift of the null map. It follows that
	$$
	\pi_1\mathrm{Pic}(P_2\mathcal{S}_*) \cong \pi_0\mathrm{Aut}(\mathbb{1}) \cong \mathbb{Z}/2
	$$
\end{rem}

\begin{ques}
	Does this computation extend to an equivalence
	$$
	\mathrm{Pic}(P_2\mathcal{S}_*) \simeq \Omega^\infty \mathbb{RP}^\infty_{-1}
	$$
\end{ques}

\begin{rem}
	This result also shows that, although $\Omega S^{0,k}$ is an infinite loop space, the object $S^{0,k}$ itself is never a loop space. Indeed, if $S^{0,k} \simeq \Omega X$, then $\pi_{0,j}S^{0,k} \cong \pi_1X$ would be independent of $j$, contrary to the calculation above.
\end{rem}

In particular, the previous lemma shows that there is always a map $S^{0,j} \to S^{0,jk}$ lifting the degree $k$ map on the sphere spectrum. Since stabilization is conservative, it follows that $S^{0,j} \simeq S^{0,jk}$ after inverting $k$. As $k$ ranges over all odd integers, we conclude that all exotic spheres become equivalent to $S^0$ after $2$-localization.

\subsection{The 1-sphere}

In this subsection we compute the homotopy groups of $S^1$. Our main result is the equivalence
$$
\Omega^2S^1 \simeq \Omega^{\infty+3}\mathbb{Sp}^2 \times \Omega^{\infty+1}\mathbb{S}[\sfrac{1}{2}]
$$
where $\mathbb{Sp}^2$ denotes the $2$-local symmetric square spectrum. From this we determine all homotopy groups except $\pi_0$ and $\pi_1$, which we show to be $0$ and $\mathbb{Z}$ respectively. We then use this computation as the starting point for the calculation of homotopy groups of spheres in the metastable category via the spectral sequence associated to the EHP sequence. As an application, we prove that no sphere is a loop space, except possibly $S^7$.

To prove this, we refine our analysis of the James--Hopf map. We have seen that the composite
$$
\Omega^\infty \mathbb{S} \xlongrightarrow{H^\infty} \Omega^\infty\mathbb{RP}^\infty_+ \xlongrightarrow{\mathrm{tr}} \Omega^\infty \mathbb{S}
$$
becomes an equivalence after looping once. At first sight this may seem unsurprising, since there is an obvious inclusion of spectra $\mathbb{S} \hookrightarrow \mathbb{RP}^\infty_+$. The point, however, is that the James--Hopf map does not identify $\Omega^\infty\mathbb{S}$ with this summand. Rather, most of its image lies in the reduced part $\Omega^\infty\mathbb{RP}^\infty$. We now make this precise.

Define the reduced James--Hopf map $\tilde{H}^\infty$ to be the composite
$$
\Omega^\infty \mathbb{S} \xlongrightarrow{H^\infty} \Omega^\infty\mathbb{RP}^\infty_+ \twoheadrightarrow \Omega^\infty\mathbb{RP}^\infty
$$
and let $\psi$ denote the composite
$$
\Omega^\infty \mathbb{S} \xlongrightarrow{H^\infty} \Omega^\infty\mathbb{RP}^\infty_+ \twoheadrightarrow \Omega^\infty \mathbb{S}
$$
Thus
$$
H^\infty = (\psi,\tilde{H}^\infty) \colon \Omega^\infty \mathbb{S} \longrightarrow \Omega^\infty \mathbb{S} \times \Omega^\infty\mathbb{RP}^\infty \simeq \Omega^\infty\mathbb{RP}^\infty_+
$$
Similarly, define the reduced transfer $\tilde{\mathrm{tr}}$ to be the composite
$$
\mathbb{RP}^\infty \hookrightarrow \mathbb{RP}^\infty_+ \xrightarrow{\mathrm{tr}} \mathbb{S}
$$

The other composite, namely
$$
\mathbb{S}\hookrightarrow \mathbb{RP}^\infty_+ \xrightarrow{\mathrm{tr}} \mathbb{S}
$$
corresponds to multiplication by $2$. Hence the transfer is given by
$$
\mathrm{tr} = 2 + \tilde{\mathrm{tr}} \colon \mathbb{S} \oplus \mathbb{RP}^\infty \simeq \mathbb{RP}^\infty_+ \to \mathbb{S}
$$
Therefore
$$
\mathrm{tr} \circ H^\infty = 2\psi + \tilde{\mathrm{tr}} \circ \tilde{H}^\infty
$$

We are now ready to prove the reduced Kahn--Priddy theorem.

\begin{lem}
	If $X$ is a suspension in $P_2\mathcal{S}_*$, then the kernel of the reduced James--Hopf map
	$$
	\tilde{H}^\infty \colon [\Sigma^\infty X, \mathbb{S}] \longrightarrow [\Sigma^\infty X, \mathbb{RP}^\infty]
	$$
	consists precisely of the infinitely $2$-divisible elements.
\end{lem}

\begin{proof}
	Every infinitely $2$-divisible element lies in the kernel, since the codomain is $2$-complete. Conversely, let $\alpha$ lie in the kernel of $\tilde{H}^\infty$. Then
	$$
	\mathrm{tr}(H^\infty(\alpha)) = 2\psi(\alpha)
	$$
	Since $X$ is a suspension, the Kahn--Priddy theorem gives $\mathrm{tr} \circ H^\infty = -1$. Hence
	$$
	\alpha = -2\psi(\alpha) = 4\psi^2(\alpha) = -8\psi^3(\alpha) = \cdots
	$$
	which shows that $\alpha$ is divisible by every power of $2$.
\end{proof}

The previous lemma immediately upgrades to a statement at the level of spaces, before passing to homotopy groups.

\begin{lem}
	The composite
	$$
	\Omega^{\infty+1}\mathbb{S} \xlongrightarrow{\Omega \tilde{H}^\infty} \Omega^{\infty+1}\mathbb{RP}^\infty \xlongrightarrow{\Omega \tilde{\mathrm{tr}}} \Omega^{\infty+1}\mathbb{S}
	$$
	is a $2$-local equivalence.
\end{lem}

\begin{proof}
	From the identity
	$$
	\mathrm{tr} \circ H^\infty = 2\psi + \tilde{\mathrm{tr}} \circ \tilde{H}^\infty
	$$
	and the Kahn--Priddy theorem, we obtain
	$$
	\Omega \tilde{\mathrm{tr}} \circ \Omega \tilde{H}^\infty = -1 - 2\Omega\psi
	$$
	It therefore suffices to show that $-1-2\Omega\psi$ is a $2$-local equivalence. For $n \geq 1$, each homotopy group $\pi_n\mathbb{S}$ is a finite abelian group, and after localizing at $2$, an endomorphism of the form $-1+2f$ is automatically an automorphism, since it reduces to $-1$ modulo $2$. Therefore $-1-2\Omega\psi$ induces an isomorphism on all $2$-local homotopy groups, and hence is a $2$-local equivalence.
\end{proof}

We now identify the reduced James--Hopf map with the James--Hopf map for $S^1$. More precisely, $\tilde{H}^\infty$ coincides with $\Omega H^\infty_{S^1}$, as follows from the diagram of canonical resolutions below.

\[\begin{tikzcd}
	{S^0} & {\Omega^\infty\mathbb{S}} & {\Omega^\infty\mathbb{RP}_+} \\
	{\Omega S^1} & {\Omega^\infty\mathbb{S}} & {\Omega^\infty\mathbb{RP}}
	\arrow[from=1-1, to=1-2]
	\arrow[from=1-1, to=2-1]
	\arrow["{H^\infty_{S^0}}", from=1-2, to=1-3]
	\arrow[equals, from=1-2, to=2-2]
	\arrow[two heads, from=1-3, to=2-3]
	\arrow[from=2-1, to=2-2]
	\arrow["{\Omega H^\infty_{S^1}}"', from=2-2, to=2-3]
\end{tikzcd}\]
Moreover, recall that the cofiber of the reduced transfer $\tilde{\mathrm{tr}}$ is the symmetric square spectrum $\mathbb{Sp}^2$. This is made explicit in \cite{KPWelcher}.

With this notation in place, we can prove the following.

\begin{lem}
	There is a $2$-local equivalence
	$$
	\Omega^2S^1 \simeq \Omega^{\infty+3}\mathbb{Sp}^2
	$$
\end{lem}

\begin{proof}
	The composite
	$$
	\Omega^{\infty+1}\mathbb{S} \xlongrightarrow{\Omega^2 H^\infty_{S^1}} \Omega^{\infty+1}\mathbb{RP}^\infty \xlongrightarrow{\Omega \tilde{\mathrm{tr}}} \Omega^{\infty+1}\mathbb{S}
	$$
	is a $2$-local equivalence by the previous lemma, and the cofiber of $\tilde{\mathrm{tr}}$ is $\mathbb{Sp}^2$. Hence we obtain a diagram of pullbacks
\[\begin{tikzcd}
	{*} & {\Omega^{\infty+2}\mathbb{Sp}^2} & {*} \\
	{\Omega^{\infty+1}\mathbb{S}} & {\Omega^{\infty+1}\mathbb{RP}} & {\Omega^{\infty+1}\mathbb{S}}
	\arrow[from=1-1, to=1-2]
	\arrow[from=1-1, to=2-1]
	\arrow["\lrcorner"{anchor=center, pos=0.125}, draw=none, from=1-1, to=2-2]
	\arrow[from=1-2, to=1-3]
	\arrow[from=1-2, to=2-2]
	\arrow["\lrcorner"{anchor=center, pos=0.125}, draw=none, from=1-2, to=2-3]
	\arrow[from=1-3, to=2-3]
	\arrow["{\Omega^2 H^\infty_{S^1}}"', from=2-1, to=2-2]
	\arrow["{\tilde{\mathrm{tr}}}"', from=2-2, to=2-3]
\end{tikzcd}\]
	Taking horizontal fibers of the left-hand square, we obtain a fiber sequence
$$
\Omega^2S^1 \longrightarrow * \longrightarrow \Omega^{\infty+2}\mathbb{Sp}^2
$$
and hence a $2$-local equivalence
$$
\Omega^2S^1 \simeq \Omega^{\infty+3}\mathbb{Sp}^2
$$
as claimed.
\end{proof}

\begin{rem}
	Applying $\Omega^\infty_2 \colon P_2\mathcal{S}_* \to \mathcal{S}_*$ to the equivalence above, as in remark \ref{Hopfgoodwillie}, we obtain a $2$-local equivalence
	$$
	\Omega^2P_2I(S^1) \simeq \Omega^{\infty+3}\mathbb{Sp}^2
	$$
	In a companion paper \cite[thm. C]{NervoCircle} we show that
	$$
	\Omega P_2I(S^1) \simeq \mathbb{Z} \times \Omega^{\infty+2}\mathbb{Sp}^2
	$$
	At present we do not know how to lift this equivalence to the metastable category, that is, we do not know whether
	$$
	\Omega S^1 \simeq \mathbb{Z} \times \Omega^{\infty+2}\mathbb{Sp}^2
	$$
	The difficulty is that, unlike in spaces, there is no evident way to construct a map $\Omega^\infty\Sigma \mathbb{Z} \to S^1$ in $P_2\mathcal{S}_*$.
\end{rem}

Before turning to the integral computation, we use the $2$-local description of $S^1$ obtained above to recover another description of the homotopy groups of $S^0$. This will provide a useful complement to the analysis of the previous subsection, and will allow us to express the homotopy groups of $S^0$ in terms of those of $S^1$ together with the stable homotopy groups.

\begin{lem}
	\label{S0split}
	There is a $2$-local equivalence
	$$
	\Omega S^0 \simeq \Omega^2 S^1 \times \Omega^{\infty+2}\mathbb{S} \simeq \Omega^{\infty+3}\mathbb{Sp}^2 \times \Omega^{\infty+2}\mathbb{S} \simeq \Omega^{\infty+2}\mathbb{RP}^\infty
	$$
\end{lem}

\begin{proof}
	Consider the metastable EHP sequence
	$$
	S^0 \longrightarrow \Omega S^1 \longrightarrow \Omega^\infty \mathbb{S}
	$$
	We claim that the map $\Omega S^0 \to \Omega^2 S^1$ admits a retraction $r$. This will imply the result, since the composite
	$$
	\Omega^2S^1 \times \Omega^{\infty+2}\mathbb{S} \xlongrightarrow{r \times \Omega\delta} \Omega S^0 \times \Omega S^0 \xlongrightarrow{m} \Omega S^0
	$$
	is an equivalence; see, for example, \cite[lem. 3.12]{Sanath}.

	By the previous lemma, we have a section
	$$
	\Omega^2 S^1 \longrightarrow \Omega^{\infty+2}\mathbb{RP}^\infty
	$$
	Hence we can form the composite
	$$
	\Omega^2 S^1 \longrightarrow \Omega^{\infty+2}\mathbb{RP}^\infty \longrightarrow \Omega^{\infty+2}\mathbb{RP}^\infty_+ \longrightarrow \Omega S^0
	$$
	which provides the desired retraction, as shown in the diagram below.
\[\begin{tikzcd}
	{\Omega^2S^1} & {\Omega^{\infty+2}\mathbb{RP}^\infty} & {\Omega^{\infty+2}\mathbb{RP}^\infty_+} & {\Omega S^0} \\
	&& {\Omega^{\infty+2}\mathbb{RP}^\infty} & {\Omega^2 S^1}
	\arrow[from=1-1, to=1-2]
	\arrow["\simeq"{description}, curve={height=30pt}, from=1-1, to=2-4]
	\arrow[from=1-2, to=1-3]
	\arrow[equals, from=1-2, to=2-3]
	\arrow[from=1-3, to=1-4]
	\arrow[from=1-3, to=2-3]
	\arrow[from=1-4, to=2-4]
	\arrow[from=2-3, to=2-4]
\end{tikzcd}\]
	This proves the first equivalence. The remaining two are immediate from the previous lemma.
\end{proof}

\begin{rem}
	This also shows that
	$$
	\Omega^{\infty+2}\mathbb{RP}^\infty_{-1} \simeq \Omega^{\infty+2}\mathbb{RP}^\infty
	$$
	This is already suggested by the Adams spectral sequence, and it should hold with fewer loops, although we do not know a reference. Of course, such an equivalence cannot be induced by a map of spectra.
\end{rem}

\begin{rem}
	The main point of this lemma is not the precise form of the splitting, but rather that, apart from $\pi_1$, the homotopy groups of $S^1$ inject into those of $S^0$, and that the homotopy groups of $S^0$ can in turn be recovered from those of $S^1$ together with the stable homotopy groups.
\end{rem}

Unlike $S^0$, the circle $S^1$ is not itself $2$-local. We therefore complete the picture integrally. This is easy, since after inverting $2$ we have the following.

\begin{lem}
	After inverting $2$, there is an equivalence
	$$
	S^1 \simeq \Omega^\infty \mathbb{S}^1
	$$
\end{lem}

\begin{proof}
	It follows from the canonical resolution
	$$
	S^1 \longrightarrow \Omega^\infty \mathbb{S}^1 \longrightarrow \Omega^\infty\Sigma\mathbb{RP}^\infty
	$$
	and the fact that $\mathbb{RP}^\infty[\sfrac{1}{2}]$ is trivial.
\end{proof}

We thus obtain an integral description of $\Omega^2S^1$. Although the symmetric square spectrum is not itself $2$-local, we continue to write $\mathbb{Sp}^2$ for its $2$-localization, in order to avoid heavier notation.

\begin{lem}
	\label{modelS1P2}
	There is an equivalence
	$$
	\Omega^2S^1 \simeq \Omega^{\infty+3}\mathbb{Sp}^2 \times \Omega^{\infty+1}\mathbb{S}[\sfrac{1}{2}]
	$$
\end{lem}

\begin{proof}
	It follows from the $2$-local and $2$-inverted descriptions above together with the fracture square. It therefore suffices to show that both factors are rationally trivial. Consider the cofiber sequence
	$$
	\mathbb{S} \longrightarrow \mathbb{Sp}^2 \longrightarrow \mathbb{RP}^\infty
	$$
	Passing to connective covers gives
	$$
	\tau_{>0}\mathbb{S} \longrightarrow \tau_{>0}\mathbb{Sp}^2 \longrightarrow \mathbb{RP}^\infty
	$$
	Since $\mathbb{RP}^\infty$ is rationally trivial, we obtain
	$$
	\mathbb{Q} \otimes \tau_{>0}\mathbb{Sp}^2 \simeq \mathbb{Q} \otimes \tau_{>0}\mathbb{S} \simeq *
	$$
	This proves the claim.
\end{proof}

This determines the homotopy groups $\pi_nS^1$ for $n>1$, so it remains to compute $\pi_0$ and $\pi_1$.

\begin{lem}
	\label{homgroupsS1}
	The bigraded homotopy groups of $S^1$ are given by
	$$
	\pi_{n,j}S^1 \cong
	\begin{cases}
		0 & \text{if } n = 0 \\
		\mathbb{Z} & \text{if } n = 1 \\
		\pi_{n+1}\mathbb{Sp}^2 \oplus \pi_{n-1}\mathbb{S}[\sfrac{1}{2}] & \text{if } n > 1
	\end{cases}
	$$
\end{lem}

\begin{proof}
	Consider the canonical resolution
	$$
	S^1 \longrightarrow \Omega^\infty \mathbb{S}^1 \longrightarrow \Omega^\infty\Sigma\mathbb{RP}^\infty
	$$
	This induces a long exact sequence
	$$
	0 \to \pi_2\mathbb{S}^1 \to \pi_2\Sigma\mathbb{RP}^\infty \to \pi_1S^1 \to \pi_1\mathbb{S}^1 \to \pi_1\Sigma\mathbb{RP}^\infty \to \pi_{0,j}S^1 \to \pi_0\mathbb{S}^1 \to \pi_0\Sigma\mathbb{RP}^\infty
	$$
	where the first map is injective by reduced Kahn--Priddy, since the groups involved are $2$-local. This reduces to
	$$
	0 \to \mathbb{Z}/2 \to \mathbb{Z}/2 \to \pi_1S^1 \to \mathbb{Z} \to 0 \to \pi_{0,j}S^1 \to 0 \to 0
	$$
	and therefore $\pi_{0,j}S^1 \cong 0$ and $\pi_1S^1 \cong \mathbb{Z}$.
\end{proof}

\begin{rem}
	\label{random}
	This can also be seen in another way. Since the underlying spectrum of $S^1$ is $\mathbb{S}^1$, the object $S^1$ is connective. Moreover, \cite[thm. 6.8]{nilpotent} gives
	$$
	\pi_1S^1 \cong \pi_1P_2I(S^1) \cong \mathbb{Z}/\Gamma_3\mathbb{Z} = \mathbb{Z}
	$$
\end{rem}

With this information, we can now write down the $E_1$-page of the spectral sequence computing the $2$-local homotopy groups of spheres in $P_2\mathcal{S}_*$; see \ref{metSS} and \ref{metSSspheres}. Since the homotopy groups of $S^1$ can already be computed independently, we begin the spectral sequence at $S^1$ rather than at $S^0$. This is also more economical, since by lemma \ref{S0split} the homotopy groups of $S^0$ can be recovered from those of $S^1$ together with the stable homotopy groups. Thus the $E_1$-page is given by

$$
E^1_{t,m} :=
\begin{cases}
	\mathbb{Z} & m = 0,\ t = 0 \\
	\pi_{t+2}\mathbb{Sp}^2 & m = 0,\ t > 0 \\
	\pi_{t-m}^s & 0 < m < n \\
	0 & m > n
\end{cases}
$$
and converges to $\pi_{n+t}S^n$.

Apart from the zeroth row, this spectral sequence agrees with the Atiyah--Hirzebruch spectral sequence for $\mathbb{RP}^\infty$. The differentials landing in the zeroth row are induced by the reduced transfer. In particular, every class in the zeroth row except the copy of $\mathbb{Z}$ must eventually be hit by a differential. A companion paper computes and displays this spectral sequence in detail \cite[sec. 3.2]{NervoCircle}. The resulting 2-local homotopy groups are summarized in the following table.

\begin{center}
\newcommand{\green}[1]{\textcolor{green!60!black}{#1}}
\[
\setlength{\tabcolsep}{7pt}
\begin{tabular}{c|c|c|c|c|c|c|c}

 & $S^0$ & $\Omega S^1$ & $\Omega^2 S^2$ & $\Omega^3 S^3$ & $\Omega^4 S^4$ & $\Omega^5 S^5$ & $\Omega^6 S^6$ \\
\midrule
$\pi_0$
& $\bullet\bullet$
& $\green{\square}$
& $\green{\square}$
& $\green{\square}$
& $\green{\square}$
& $\green{\square}$
& $\green{\square}$
\\

$\pi_1$
& $\bullet$
&
& $\green{\square}$
& $\green{\bullet}$
& $\green{\bullet}$
& $\green{\bullet}$
& $\green{\bullet}$
\\

$\pi_2$
& $\bullet_3$
&
& $\green{\bullet}$
& $\green{\bullet}$
& $\green{\bullet}$
& $\green{\bullet}$
& $\green{\bullet}$
\\

$\pi_3$
& $\bullet$
& $\bullet$
& $\green{\bullet}$
& $\green{\bullet_2}$
& $\green{\square\ \bullet_2}$
& $\green{\bullet_3}$
& $\green{\bullet_3}$
\\

$\pi_4$
&
&
& $\green{\bullet_2}$
& $\green{\bullet}$
& $\green{\bullet\ \bullet}$
& $\green{\bullet}$
& 
\\

$\pi_5$
& $\bullet$
&
&
& $\green{\bullet}$
& $\green{\bullet\bullet}$
& $\green{\bullet}$
& $\green{\square}$
\\

$\pi_6$
& $\bullet\ \bullet_4$
& $\bullet$
&
&
& $\green{\bullet_3}$
& $\green{\bullet}$
& $\green{\bullet}$
\\

$\pi_7$
& $\bullet\ \bullet\ \bullet$
& $\bullet$
&
&
&
& $\green{\bullet}$
& $\green{\bullet_2}$
\\

$\pi_8$
& $\bullet\ \bullet\ \bullet\ \bullet$
& $\bullet$
& $\bullet_3$
& $\bullet$
& $\green{\bullet}$
& $\green{\bullet}$
& $\green{\bullet\bullet_3}$
\\

$\pi_9$
& $\bullet_3\ \bullet$
& $\bullet_3$
& $\bullet\ \bullet_2$
& $\bullet\ \bullet\ \bullet$
& $\green{\bullet\ \bullet\ \bullet}$
& $\green{\bullet\ \bullet\ \bullet}$
& $\green{\bullet\bullet\bullet}$
\\

$\pi_{10}$
& $\bullet_3$
&
& $\bullet\ \bullet$
& $\bullet\bullet_2$
& $\green{\bullet\ \bullet_2\ \bullet_3}$
& $\green{\bullet\bullet_3}$
& $\green{\bullet\bullet_3}$
\\

$\pi_{11}$
&
&
& $\bullet$
& $\bullet\bullet\bullet\bullet$
& $\bullet\bullet\bullet\bullet\bullet\bullet$
& $\green{\bullet\bullet\bullet_3}$
& $\green{\bullet_2\bullet_3}$
\\

$\pi_{12}$
&
&
& $\bullet_3$
& $\bullet\ \bullet$
& $\bullet\ \bullet\ \bullet\ \bullet\ \bullet$
& $\green{\bullet\bullet\bullet}$
& $\green{\bullet_4}$
\\

$\pi_{13}$
& $\bullet\bullet$
&
&
& $\bullet$
& $\bullet\bullet$
& $\green{\bullet\ \bullet}$
& $\green{\bullet}$
\\

$\pi_{14}$
& $\bullet \bullet \bullet_5$
& $\bullet$
&
&
& $\bullet_3$
& $\bullet$
& $\green{\bullet\bullet_2}$
\\

$\pi_{15}$
& $\bullet\bullet\bullet\bullet$
& $\bullet\ \bullet$
& $\bullet\ \bullet\ \bullet$
& $\bullet$
& $\bullet$
& $\bullet\ \bullet$
& $\green{\bullet\bullet_2}$
\\

$\pi_{16}$
& $\bullet\bullet\bullet\bullet\bullet\bullet$
& $\bullet\ \bullet$
& $\bullet\ \bullet\ \bullet_5$
& $\bullet\ \bullet\ \bullet$
& $\bullet\ \bullet$
& $\bullet\ \bullet$
& $\green{\bullet\bullet\bullet_3}$
\\

$\pi_{17}$
& $\bullet_3\bullet_3\ \bullet$
& $\bullet_3$
& $\bullet\ \bullet_2$
& $\bullet\ \bullet\ \bullet\ \bullet$
& $\bullet\ \bullet\ \bullet\ \bullet\ \bullet$
& $\bullet\bullet\bullet$
& $\bullet\bullet\bullet$
\\
\end{tabular}
\]
{\footnotesize The $2$-local homotopy groups of spheres in the metastable category}
\end{center}

We have highlighted in green the homotopy groups that agree with the unstable ones, that is, those for which the map
$$
\pi_kS^n \longrightarrow \pi_kP_2I(S^n)
$$
is an equivalence. By the generalized Freudenthal theorem proved in the companion paper \cite[thm. A]{NervoCircle}, this holds precisely for $k < 4n-1$. We emphasize, however, that the table does not determine the precise group structure of the black dots, since we do not resolve the $2$-extensions in the spectral sequence.

We conclude the subsection with the following application.

\begin{lem}
	\label{noloopspaces}
	No metastable sphere is a loop space except, possibly, $S^7$.
\end{lem}

\begin{proof}
	Whenever an object $X$ is a loop space, the map $X \to \Omega \Sigma X$ admits a retraction. Indeed, if $X \simeq \Omega Y$, then the composite
	$$
	\Omega Y \xlongrightarrow{\eta_{\Omega Y}} \Omega \Sigma \Omega Y \xlongrightarrow{\Omega \epsilon_Y} \Omega Y
	$$
	is an equivalence by the triangle identities. In particular, the induced map
	$$
	\pi_kX \longrightarrow \pi_{k+1}\Sigma X
	$$
	is injective for every $k$.

	Now the map $\pi_{2n-1}S^n \to \pi_{2n}S^{n+1}$ agrees with the unstable suspension map, since we are in the convergent range. By the classical Hopf invariant one theorem \cite[thm.~A]{AdamsAtiyah}, this map fails to be injective unless $n=0,1,3,7$. These are therefore the only spheres that can possibly be loop spaces.

	It remains to rule out $S^0$, $S^1$, and $S^3$. From the table above we see that the map $\pi_1S^0 \to \pi_2S^1$ is not injective, and similarly the maps $\pi_7S^1 \to \pi_8S^2$ and $\pi_{19}S^3 \to \pi_{20}S^4$ are not injective. Hence $S^0$, $S^1$, and $S^3$ are not loop spaces. The only remaining possibility is $S^7$, which lies beyond the range of our computation.
\end{proof}

\section{Miscellanea}

In this section we collect a number of constructions and questions that lie somewhat outside the main development of the paper, but are nevertheless closely related to it.

In the first subsection, we introduce a generalized Mahowald invariant for maps of spectra. As an application, we obtain a corresponding generalization of Jones' theorem.

In the second subsection, we obtain the Hilton--Milnor splitting in the metastable category. We use it to compute the homotopy groups of $S^1 \vee S^1$.

The final two subsections are more speculative in nature.

In the third subsection, we discuss possible odd-primary analogues of the metastable category and explain why one of the candidate constructions appears to be the most promising.

In the final subsection, we outline a possible Adams-type spectral sequence for the metastable category, relating it to earlier ideas of Mahowald and Milgram.

\subsection{Mahowald invariant}

In this subsection we define a natural generalization of the Mahowald invariant \cite{Mahowaldinv} and relate it to the metastable category. Roughly speaking, this generalized Mahowald invariant can be viewed as a kind of twisted, maximally desuspended metastable structure. In particular, we will prove a generalized version of Jones' theorem \cite{Jonesthm}. The original theorem states that the Mahowald invariant of a map $\mathbb{S}^t \to \mathbb{S}$ raises the stem by at least $t$. Our generalization is the following: if $X$ admits a metastable structure, then the Mahowald invariant of a map $\Sigma^t X \to Y$ raises the stem by at least $t$.

We begin by recalling the classical definition.

Let $\alpha \colon \mathbb{S}^t \to \mathbb{S}^\wedge_2$. Classically, the Mahowald invariant $M(\alpha)$ of $\alpha$ is the coset of completions of the diagram
\[\begin{tikzcd}
	{\mathbb{S}^t} & {\mathbb{S}^{-N}} \\
	{\mathbb{S}^\wedge_2 \simeq \Sigma\mathbb{RP}^\infty_{-\infty}} & { \Sigma \mathbb{RP}^\infty_{-N-1}}
	\arrow["{M(\alpha)}", dashed, from=1-1, to=1-2]
	\arrow["\alpha"', from=1-1, to=2-1]
	\arrow[from=1-2, to=2-2]
	\arrow[from=2-1, to=2-2]
\end{tikzcd}\]
where $N$ is chosen to be minimal such that the composite along the lower path is nontrivial. The identification in the lower-left corner is Lin's theorem, which is equivalent to the case $p = 2$ of the Segal conjecture, since
$$
\mathbb{S}^\wedge_2 \simeq (\mathbb{S}^{\otimes 2})^{tC_2} \simeq P^1(\Sigma(-^{\otimes 2})_{hC_2})(\mathbb{S}) \simeq \mathrm{lim}_n\;\Sigma^{n+1}((\mathbb{S}^{-n})^{\otimes 2})_{hC_2} \simeq \mathrm{lim}_n\;\Sigma\mathbb{RP}^\infty_{-n} =: \Sigma\mathbb{RP}^\infty_{-\infty}
$$

The bottom horizontal map is induced by
$$
(Y^{\otimes 2})^{tC_2} \simeq \Sigma^{n+1}\big((\Sigma^{-(n+1)}Y)^{\otimes 2}\big)^{tC_2} \longrightarrow \Sigma^{n+2}\big((\Sigma^{-(n+1)}Y)^{\otimes 2}\big)_{hC_2} \simeq \Sigma(\Sigma^{-(n+1)\rho}Y^{\otimes 2})_{hC_2}
$$
in the case $Y = \mathbb{S}$. This suggests the following general definition.

\begin{definition}
	\label{newmahowald}
	Let $\alpha \colon X \to Y$ be a map of spectra. The \textit{Mahowald invariant} $M(\alpha)$ of $\alpha$ is the coset of completions of the diagram
\[\begin{tikzcd}
	X && {\Sigma^{-N} Y^{\otimes 2}} \\
	Y & {(Y^{\otimes 2})^{tC_2}} & { \Sigma (\Sigma^{-(N+1)\rho}Y^{\otimes 2})_{hC_2}}
	\arrow["{M(\alpha)}", dashed, from=1-1, to=1-3]
	\arrow["\alpha"', from=1-1, to=2-1]
	\arrow["{\Sigma\pi}", from=1-3, to=2-3]
	\arrow["\Delta"', from=2-1, to=2-2]
	\arrow[from=2-2, to=2-3]
\end{tikzcd}\]
where $N$ is chosen to be minimal such that the composite along the lower path is nontrivial. We write $|M(\alpha)|$ for this integer $N$.
\end{definition}

\begin{rem}
	Our notation $|M(\alpha)|$ differs from the classical notation in the literature, where $|M(\alpha)|$ denotes the stem of the map. If $\alpha$ is a map $\mathbb{S}^t \to \mathbb{S}$, then classically $|M(\alpha)| = t + N$. By contrast, our $|M(\alpha)|$ measures the shift, that is, the increase in the stem rather than the stem itself.
\end{rem}

\begin{rem}
	Note that the generalized Mahowald invariant defines a multivalued function
	$$
	M \colon [X, Y] \longrightarrow \bigoplus_n[\Sigma^n X, Y^{\otimes 2}]
	$$
	In contrast with the classical case, the codomain changes, since $Y^{\otimes 2}$ is not equivalent to $Y$ in general.
\end{rem}

The following lemma is immediate from the definition.

\begin{lem}
	Let $\alpha \colon X \to Y$ and $\beta \colon Y \to Z$ be maps of spectra. Then
	$$
	|M(\beta \circ \alpha)| \geq \max(|M(\alpha)|, |M(\beta)|)
	$$
\end{lem}

\begin{proof}
	This follows from naturality, which gives the diagram
\[\begin{tikzcd}
	X && \\
	Y & {(Y^{\otimes 2})^{tC_2}} & { \Sigma (\Sigma^{-n\rho}Y^{\otimes 2})_{hC_2}} \\
	Z & {(Z^{\otimes 2})^{tC_2}} & { \Sigma (\Sigma^{-n\rho}Z^{\otimes 2})_{hC_2}}
	\arrow["\alpha"', from=1-1, to=2-1]
	\arrow["\Delta"', from=2-1, to=2-2]
	\arrow["\beta"', from=2-1, to=3-1]
	\arrow[from=2-2, to=2-3]
	\arrow[from=2-2, to=3-2]
	\arrow[from=2-3, to=3-3]
	\arrow["\Delta"', from=3-1, to=3-2]
	\arrow[from=3-2, to=3-3]
\end{tikzcd}\]
	If for some $n$ the composite
	$$
	X \xrightarrow{\alpha} Y \xrightarrow{\Delta} (Y^{\otimes 2})^{tC_2} \longrightarrow \Sigma(\Sigma^{-n\rho}Y^{\otimes 2})_{hC_2}
	$$
	is trivial, then so is the corresponding composite for $\beta \circ \alpha$. Hence $|M(\beta \circ \alpha)| \geq |M(\alpha)|$. The same argument applied to $\beta$ gives $|M(\beta \circ \alpha)| \geq |M(\beta)|$, and the claim follows.
\end{proof}

To bring this closer to the main theme of the paper, we now introduce a more refined version of the Mahowald invariant.

\begin{definition}
	Let $\alpha \colon X \to Y$ be a map of spectra. The \textit{lift invariant} $L(\alpha)$ of $\alpha$ is the coset of completions of the diagram
\[\begin{tikzcd}
	&& {(\Sigma^{-N\rho}Y^{\otimes 2})^{hC_2}} \\
	X & Y & {(Y^{\otimes 2})^{tC_2}}
	\arrow["{\mathrm{can}}", from=1-3, to=2-3]
	\arrow["{L(\alpha)}", dashed, from=2-1, to=1-3]
	\arrow["\alpha"', from=2-1, to=2-2]
	\arrow["\Delta"', from=2-2, to=2-3]
\end{tikzcd}\]
	where $N$ is chosen to be maximal such that the lift exists. We write $|L(\alpha)|$ for this integer $N$.
\end{definition}

\begin{rem}
	For $\alpha = 1_X$ and $N = 0$, this recovers our notion of metastable structure.
\end{rem}

Although $M(\alpha)$ and $L(\alpha)$ take values in different sets, they determine the same integer. In particular, for every map $\alpha \colon X \to Y$, one has
$$
|M(\alpha)| = |L(\alpha)|
$$

If we denote this common integer by $N$ and write
$$
i \colon (\Sigma^{-N\rho}Y^{\otimes 2})^{hC_2} \longrightarrow \Sigma^{-N}Y^{\otimes 2}
$$
for the forgetful map, then the following lemma identifies the Mahowald invariant as the image of the lift invariant under $i$.

\begin{lem}
	Let $\alpha \colon X \to Y$ be a map of spectra. Then $M(\alpha) = i_*L(\alpha)$.
\end{lem}

\begin{proof}
	We claim that for any spectrum $Z$ with a $C_2$-action there is a pullback square
\[\begin{tikzcd}
	{Z^{hC_2}} & {Z \simeq \Sigma(\Sigma^{-1}Z)} \\
	{Z^{tC_2}} & {\Sigma(\Sigma^{-\rho}Z)_{hC_2}}
	\arrow["i", from=1-1, to=1-2]
	\arrow["{\mathrm{can}}"', from=1-1, to=2-1]
	\arrow["\lrcorner"{anchor=center, pos=0.125}, draw=none, from=1-1, to=2-2]
	\arrow["{\Sigma\pi}", from=1-2, to=2-2]
	\arrow[from=2-1, to=2-2]
\end{tikzcd}\]
	Applying this with $Z = \Sigma^{-N\rho}Y^{\otimes 2}$, the result follows immediately. Indeed, a choice of Mahowald invariant is precisely a map into the span on the right, while a choice of lift invariant is precisely a map into the pullback. This is summarized by the diagram
\[\begin{tikzcd}
	X && \\
	& {(\Sigma^{-N\rho}Y^{\otimes 2})^{hC_2}} & {\Sigma^{-N\rho}Y^{\otimes 2}} \\
	& {Y \simeq (\Sigma^{-N\rho}Y^{\otimes 2})^{tC_2}} & {\Sigma(\Sigma^{-(N+1)\rho}Y^{\otimes 2})_{hC_2}}
	\arrow["{L(\alpha)}"{description}, dashed, from=1-1, to=2-2]
	\arrow["{M(\alpha)}"{description}, dashed, from=1-1, to=2-3]
	\arrow["{\Delta \circ \alpha}"', from=1-1, to=3-2]
	\arrow["i", from=2-2, to=2-3]
	\arrow["{\mathrm{can}}"', from=2-2, to=3-2]
	\arrow["{\Sigma\pi}", from=2-3, to=3-3]
	\arrow[""{name=0, anchor=center, inner sep=0}, from=3-2, to=3-3]
	\arrow["\lrcorner"{anchor=center, pos=0.125}, draw=none, from=2-2, to=0]
\end{tikzcd}\]
	It therefore remains to prove that the first square is a pullback. For this, we compute the fiber of the right vertical map. Recall the cofiber sequence
	$$
	S(\rho)_+ = C_{2+} \longrightarrow S^0 \longrightarrow S^\rho
	$$
	Smashing with $\Sigma^{-\rho}Z$ gives a cofiber sequence
	$$
	C_{2+} \otimes \Sigma^{-\rho}Z \longrightarrow \Sigma^{-\rho}Z \longrightarrow Z
	$$
	and passing to homotopy orbits yields
	$$
	\Sigma^{-1}Z \xlongrightarrow{\pi} (\Sigma^{-\rho}Z)_{hC_2} \longrightarrow Z_{hC_2}
	$$
	Hence $Z_{hC_2} \simeq \mathrm{fib}(\Sigma \pi)$, which proves the claim.
\end{proof}

We now make the relation between the lift invariant, the Mahowald invariant, and metastable structures more explicit.

\begin{lem}
	Let $X$ be a spectrum, and let $N$ be the maximal integer such that $\Sigma^{-N}X$ admits a metastable structure. Then
	$$
	L(1_X) = \mathcal{M}_{\Sigma^{-N}X} = \{\text{metastable structures on }\Sigma^{-N}X\}
	$$
	and
	$$
	M(1_X) = \{\text{diagonals arising from metastable structures on }\Sigma^{-N}X\}
	$$
\end{lem}

\begin{rem}
	If $X$ admits a metastable structure, then $N$ is the maximal number of times that $X$ can be desuspended in $P_2\mathcal{S}_*$.
\end{rem}

\begin{proof}
	The second statement follows from the previous lemma, since composing a metastable structure $Y \to (Y^{\otimes 2})^{hC_2}$ with the forgetful map $(Y^{\otimes 2})^{hC_2} \to Y^{\otimes 2}$ gives exactly the diagonal arising from that metastable structure.

	For the first statement, note that the set of metastable structures on $\Sigma^{-N}X$ is the set of completions of the diagram
\[\begin{tikzcd}
	& {((\Sigma^{-N}X)^{\otimes 2})^{hC_2}} & {\Sigma^{-N}(\Sigma^{-N\rho}X^{\otimes 2})^{hC_2}} \\
	{\Sigma^{-N}X} & {((\Sigma^{-N}X)^{\otimes 2})^{tC_2}} & {\Sigma^{-N}(X^{\otimes 2})^{tC_2}}
	\arrow[equals, from=1-2, to=1-3]
	\arrow["{\mathrm{can}}", from=1-2, to=2-2]
	\arrow["{\mathrm{can}}", from=1-3, to=2-3]
	\arrow[dashed, from=2-1, to=1-2]
	\arrow["\Delta"', from=2-1, to=2-2]
	\arrow[equals, from=2-2, to=2-3]
\end{tikzcd}\]
	After removing $\Sigma^{-N}$ throughout, this is exactly the diagram defining the lift invariant of $1_{\Sigma^{-N}X}$.
\end{proof}

\begin{ex}
	We have seen that $\mathbb{S}^n$ admits a metastable structure only for $n \geq 0$. Hence
	$$
	L(1_{\mathbb{S}^n}) = \mathcal{M}_{\mathbb{S}} = \{L + k\,(\mathrm{nm}\circ \pi) \mid k\in\mathbb{Z}\} \subseteq [\mathbb{S}, \mathbb{S}^{hC_2}]
	$$
	where $L \colon \mathbb{S} \to \mathbb{S}^{hC_2}$ denotes the spacelike metastable structure on $\mathbb{S}$ and $\pi \colon \mathbb{S} \to \mathbb{RP}_+$ the inclusion. Moreover,
	$$
	M(1_{\mathbb{S}^n}) = \{2k + 1 \mid k\in\mathbb{Z}\} \subseteq [\mathbb{S}, \mathbb{S}]
	$$
	and the map $i$ sends $L + k\,(\mathrm{nm}\circ \pi)$ to $2k+1$. See lemma \ref{exotics0} and the preceding discussion.
\end{ex}

We now obtain the promised generalization of Jones' theorem.

\begin{lem}
	\label{newjones}
	Let $\alpha \colon X \to Y$ be a map of spectra. If $X$ or $Y$ is an $n$-fold suspension in $P_2\mathcal{S}_*$, then $|M(\alpha)| \geq n$.
\end{lem}

\begin{proof}
	By the previous discussion, if $X$ or $Y$ is an $n$-fold suspension in $P_2\mathcal{S}_*$, then either $|M(1_X)| \geq n$ or $|M(1_Y)| \geq n$. Since
	$$
	\alpha = 1_Y \circ \alpha \circ 1_X
	$$
	the lemma on composition implies that $|M(\alpha)| \geq n$, as required.
\end{proof}

\begin{ex}
	As a special case of the previous lemma, if $X$ or $Y$ admits a metastable structure, then $|M(\alpha)| \geq 0$. In general, however, there is no reason for this number to be positive. For example, $|M(1_{\mathbb{S}/2})| = -1$, since $\mathbb{S}/2$ does not admit a metastable structure, whereas its suspension does.
\end{ex}

\begin{ex}
	Consider a map $\alpha \colon \mathbb{S}^t \to \mathbb{S}$. Since $\mathbb{S}^t$ is a $t$-fold suspension in $P_2\mathcal{S}_*$, it follows that $|M(\alpha)| \geq t$. Hence every element of $M(\alpha)$ lies in stem
	$$
	t + |M(\alpha)| \geq 2t
	$$
	which recovers the classical Jones theorem \cite{Jonesthm}.
\end{ex}

\subsection{Hilton--Milnor splitting}

In this subsection we obtain an analogue of the Hilton--Milnor splitting \cite{HiltonMilnor1, HiltonMilnor2} in the metastable category. This also recovers a result of \cite{GijsLukas}. We then use it to compute the homotopy groups of $S^1 \vee S^1$.

We begin by comparing the canonical resolution of $X_1 \vee X_2$ with the product of the canonical resolutions of $X_1$ and $X_2$.

\begin{lem}
	For any $X_1, X_2$ in $P_2\mathcal{S}_*$, there is a fiber sequence
	$$
	X_1 \vee X_2 \longrightarrow X_1 \times X_2 \longrightarrow \Omega^\infty\Sigma^\infty (X_1 \otimes X_2)
	$$
\end{lem}

\begin{proof}
	The canonical resolution of $X_1 \vee X_2$ is
$$
X_1 \vee X_2 \longrightarrow \Omega^\infty \Sigma^\infty (X_1 \vee X_2) \longrightarrow \Omega^\infty \Sigma^\infty (X_1 \vee X_2)^{\otimes 2}_{hC_2} \simeq \Omega^\infty \Sigma^\infty \bigl((X_1)^{\otimes 2}_{hC_2} \vee (X_2)^{\otimes 2}_{hC_2} \vee X_1 \otimes X_2\bigr)
$$
Since $\Sigma^\infty$ preserves colimits and is symmetric monoidal, this becomes
$$
X_1 \vee X_2 \longrightarrow \Omega^\infty(\Sigma^\infty X_1 \oplus \Sigma^\infty X_2) \longrightarrow \Omega^\infty \bigl((\Sigma^\infty X_1)^{\otimes 2}_{hC_2} \oplus (\Sigma^\infty X_2)^{\otimes 2}_{hC_2} \oplus \Sigma^\infty (X_1 \otimes X_2)\bigr)
$$
	On the other hand, taking the product of the canonical resolutions of $X_1$ and $X_2$ gives a fiber sequence
$$
X_1 \times X_2 \longrightarrow \Omega^\infty\Sigma^\infty X_1 \times \Omega^\infty\Sigma^\infty X_2 \longrightarrow \Omega^\infty (\Sigma^\infty X_1)^{\otimes 2}_{hC_2} \times \Omega^\infty (\Sigma^\infty X_2)^{\otimes 2}_{hC_2}
$$
This is in turn equivalent to
$$
X_1 \times X_2 \longrightarrow \Omega^\infty(\Sigma^\infty X_1 \oplus \Sigma^\infty X_2) \longrightarrow \Omega^\infty \bigl((\Sigma^\infty X_1)^{\otimes 2}_{hC_2} \oplus (\Sigma^\infty X_2)^{\otimes 2}_{hC_2}\bigr)
$$
	The result now follows by taking vertical fibers in the following diagram of fiber sequences.
	\[
	\adjustbox{scale=0.9,center}{
\begin{tikzcd}
	{\Omega^\infty(\Sigma^\infty X_1 \oplus \Sigma^\infty X_2)} & {\Omega^\infty(\Sigma^\infty X_1 \oplus \Sigma^\infty X_2)} & {*} \\
	{ \Omega^\infty ((\Sigma^\infty X_1)^{\otimes 2}_{hC_2} \oplus (\Sigma^\infty X_2)^{\otimes 2}_{hC_2} \oplus \Sigma^\infty (X_1 \otimes X_2))} & { \Omega^\infty ((\Sigma^\infty X_1)^{\otimes 2}_{hC_2} \oplus (\Sigma^\infty X_2)^{\otimes 2}_{hC_2})} & { \Omega^\infty\Sigma^{\infty+1} (X_1 \otimes X_2)}
	\arrow[equals, from=1-1, to=1-2]
	\arrow[from=1-1, to=2-1]
	\arrow[from=1-2, to=1-3]
	\arrow[from=1-2, to=2-2]
	\arrow[from=1-3, to=2-3]
	\arrow[from=2-1, to=2-2]
	\arrow[from=2-2, to=2-3]
\end{tikzcd}}\]
	This proves the lemma.
\end{proof}

Looping the previous fiber sequence yields the metastable analogue of the Hilton--Milnor splitting.

\begin{lem}
	\label{hiltonmilnor}
	For any $X_1, X_2$ in $P_2\mathcal{S}_*$, there is an equivalence
	$$
	\Omega (X_1 \vee X_2) \simeq \Omega X_1 \times \Omega X_2 \times \Omega^{\infty + 2}\Sigma^\infty(X_1 \otimes X_2)
	$$
\end{lem}

\begin{proof}
	By the previous lemma, there is a fiber sequence
	$$
	\Omega^{\infty + 2}\Sigma^\infty(X_1 \otimes X_2) \longrightarrow \Omega (X_1 \vee X_2) \longrightarrow \Omega X_1 \times \Omega X_2
	$$
	The right-hand map admits a section, given by
	$$
	\Omega X_1 \times \Omega X_2 \xlongrightarrow{\Omega i_1 \times \Omega i_2} \Omega (X_1 \vee X_2) \times \Omega (X_1 \vee X_2) \xlongrightarrow{m} \Omega (X_1 \vee X_2)
	$$
	where $i_j$ are the insertion maps into the coproduct and $m$ is the multiplication map induced by the loop space structure; see, for example, \cite[proof of thm. 3.1.]{Sanath}. This proves the lemma.
\end{proof}

Passing to spaces, we recover the corresponding splitting for the $2$-excisive approximation.

\begin{rem}
	Using the identification of $P_2\mathcal{S}_*$ with the $2$-excisive approximation of $\mathcal{S}_*$, exactly as in remark \ref{Hopfgoodwillie}, for any $X_1, X_2$ in $\mathcal{S}_*$ we obtain the splitting
	$$
	\Omega P_2I(X_1 \vee X_2) \simeq \Omega P_2I(X_1) \times \Omega P_2I(X_2) \times \Omega^2 P_1I(X_1 \wedge X_2)
	$$
	Applied to suspensions, this gives
	$$
	\Omega P_2I(\Sigma X_1 \vee \Sigma X_2) \simeq \Omega P_2I(\Sigma X_1) \times \Omega P_2I(\Sigma X_2) \times \Omega^2 P_1I(\Sigma X_1 \wedge \Sigma X_2)
	$$
	Since
	$$
	\Omega^2 P_1I(\Sigma X_1 \wedge \Sigma X_2) \simeq \Omega P_1I(\Sigma (X_1 \wedge X_2))
	$$
	this recovers a result of Brantner--Heuts \cite[thm. 1.2, $n=2$]{GijsLukas}. One may wonder why this gives a stronger statement, namely a splitting without suspending the factors. This reflects a general feature of calculus: polynomial functors are determined by their values on suspensions. Therefore, if two polynomial functors agree on suspensions, then they agree in general. We leave the details to the reader.
\end{rem}

We conclude this subsection by using the Hilton--Milnor splitting to compute the homotopy groups of $S^1 \vee S^1$.

\begin{ex}
	\label{homgroupsS1vS1}
	From the Hilton--Milnor splitting we obtain
	$$
	\Omega (S^1 \vee S^1) \simeq \Omega S^1 \times \Omega S^1 \times \Omega^\infty \mathbb{S}
	$$
	We already know from \cite[thm. 6.8]{nilpotent} that
	$$
	\pi_1(S^1 \vee S^1) \cong F_2/\Gamma_3F_2 \cong \mathrm{Heis}(\mathbb{Z})
	$$
	and, using the description of $S^1$ from lemma \ref{homgroupsS1}, we obtain
	$$
	\pi_{n, j}(S^1 \vee S^1) \cong
	\begin{cases}
		0 & n = 0 \\
		\mathrm{Heis}(\mathbb{Z}) & n = 1 \\
		\bigl(\pi_{n+1}\mathbb{Sp}^2 \oplus \pi_{n-1}\mathbb{S}[\sfrac{1}{2}]\bigr)^{\oplus 2} \oplus \pi_{n-1}\mathbb{S} & n > 1
	\end{cases}
	$$
	This example also shows that the Hilton--Milnor splitting cannot be upgraded to an equivalence of loop spaces, since
	$$
	\pi_1(S^1 \vee S^1) \cong \mathrm{Heis}(\mathbb{Z}) \not\cong \mathbb{Z}^3 \cong \pi_1S^1 \times \pi_1S^1 \times \pi_0\mathbb{S}
	$$
\end{ex}

\subsection{Odd primes}

Our definition of the metastable category implicitly singles out the prime $p = 2$. It is therefore natural to ask what happens at an odd prime, and to what extent the theory developed in this paper admits a $p$-local analogue. There are three natural candidates.

The first and most natural candidate is the category $P_p\mathcal{S}_*$. From the point of view of calculus, this is clearly the correct analogue, since it fits the philosophy of approximating the category of pointed spaces by polynomial approximations. The difficulty is that this category is already substantially more complicated to construct and study. Indeed, each $P_m\mathcal{S}_*$ is defined inductively in terms of $P_{m-1}\mathcal{S}_*$, so understanding $P_p\mathcal{S}_*$ requires understanding the entire tower $P_m\mathcal{S}_*$ for $m \leq p$. This complexity does not disappear even after working $p$-locally.

A second possibility is to define the category as spectra equipped with lifts
\[\begin{tikzcd}
	& {(X^{\otimes p})^{hC_p}} \\
	X & {(X^{\otimes p})^{tC_p}}
	\arrow["{\mathrm{can}}", from=1-2, to=2-2]
	\arrow[dashed, from=2-1, to=1-2]
	\arrow["\Delta"', from=2-1, to=2-2]
\end{tikzcd}\]
This construction still receives a colimit-preserving functor from $\mathcal{S}_*$, and hence admits a right adjoint. However, it does not seem to produce a meaningful approximation to spaces: the unit of this adjunction is not even a $p$-local $p$-excisive equivalence on basic examples such as spheres. For this reason, this candidate appears to be too weak to be of real interest.

The third candidate is to define the category as spectra equipped with lifts
\[\begin{tikzcd}
	& {((X^{\otimes p})^{hC_p})^{h\mathbb{F}_p^\times}} \\
	X & {((X^{\otimes p})^{tC_p})^{h\mathbb{F}_p^\times}}
	\arrow["{\mathrm{can}}", from=1-2, to=2-2]
	\arrow[dashed, from=2-1, to=1-2]
	\arrow["\Delta"', from=2-1, to=2-2]
\end{tikzcd}\]
This category is still different from the $p$-excisive approximation of spaces, but it appears to be the most promising odd-primary analogue. At least on odd-dimensional spheres, the unit of the resulting adjunction agrees with the $p$-local $p$-excisive approximation, even though this already fails on even-dimensional spheres.

Most of the arguments should then go through with only minor changes. In particular, one expects the existence of a colimit-preserving functor from spaces, stabilization and costabilization statements analogous to those at the prime $2$, a symmetric monoidal structure, exotic $S^0$-spheres, and a canonical resolution together with the corresponding intermediate sequences. Likewise, the basic computational outputs of the theory should remain available: one should still be able to analyze the cases of $S^0$ and $S^1$, and to construct odd-primary analogues of the spectral sequence and exponent results.

The main difference is that one must now work everywhere with the reduced regular representation of $C_p$, rather than with the sign representation of $C_2$. This makes the geometry less transparent, and one therefore has to be more careful precisely at the points where special properties of the sign representation were used in an essential way. In particular, the analogues of the single and double suspension results would have to be revisited, and similarly for the derivation of the EHP sequences and related calculations.

In summary, the first candidate is conceptually the right one but presently too difficult, the second is formally available but seems too weak to be interesting, and the third appears to retain enough structure to support a meaningful odd-primary theory.

\subsection{Adams spectral sequence}

It would be very desirable to develop an analogue of the Adams spectral sequence for the metastable category. The basic idea is already visible in the work of Mahowald \cite{Metastable}, in the case of spheres, and in Milgram's work \cite{Milgram}, for more general spaces. One would like to lift the Hopf map
$$
[\Sigma^\infty T, \Sigma^\infty X] \xlongrightarrow{H^\infty} [\Sigma^\infty T, \Sigma^\infty X^{\otimes 2}_{hC_2}]
$$
to a corresponding map on Adams spectral sequences
$$
\mathrm{Ext}_{\mathcal{A}_*}(\mathbb{F}_{2*}T, \mathbb{F}_{2*}X) \longrightarrow \mathrm{Ext}_{\mathcal{A}_*}(\mathbb{F}_{2*}T, \mathbb{F}_{2*}X^{\otimes 2}_{hC_2})
$$
At this level, however, the picture should be understood mainly as a guiding principle. In the sphere case, Mahowald's work fits this picture closely, whereas for general spaces Milgram's construction is more elaborate and is not simply given by a direct map of Adams spectral sequences of the form above.

In many cases, this map should still be explicit enough to be computable. Mahowald and Milgram use it for direct calculations, but one may hope for a more conceptual interpretation: the kernel and cokernel should assemble into a bigraded $\mathbb{F}_2$-vector space that behaves like the $E_2$-page of a spectral sequence computing $[T, X]$ in the metastable category. It would be very desirable to make this precise by constructing such a spectral sequence internally in the metastable category, and only afterwards recovering the Mahowald--Milgram calculations from it.

There is also a very natural candidate for such a construction, which works in any pointed presentable category that stabilizes to spectra. Let $C$ be such a category. The adjunction $\Sigma^\infty \dashv \Omega^\infty$ gives rise to the monad
$$
\Omega^\infty(\mathbb{F}_2 \otimes \Sigma^\infty -)
$$
on $C$. This in turn produces a cosimplicial resolution, and hence a Bousfield--Kan spectral sequence. One would then have to show that the resulting totalization agrees with $2$-completion for a suitable class of objects.

One expects the $E_2$-page to admit a corresponding algebraic description. Namely, the same adjunction gives rise to the comonad
$$
\mathbb{F}_2 \otimes \Sigma^\infty \Omega^\infty
$$
on $\mathbb{F}_2\text{-}\mathrm{Mod}$, and hence to a comonad $\Gamma$ on the homotopy category, which may be identified with the category of graded $\mathbb{F}_2$-modules. The $E_2$-page should then be given by an Ext-group in the category of $\Gamma$-coalgebras.

This recovers the classical stable and unstable Adams spectral sequences. Indeed, for $C = \mathcal{Sp}$ the comonad is simply $\Gamma(V) := \mathcal{A}_* \otimes V$, and $\Gamma$-coalgebras are precisely $\mathcal{A}_*$-comodules. For $C = \mathcal{S}_*$ the comonad is the free unstable coalgebra functor on graded $\mathbb{F}_2$-vector spaces, and $\Gamma$-coalgebras are unstable $\mathcal{A}_*$-coalgebras.

The natural question is then: what is $\Gamma$ when $C = P_2\mathcal{S}_*$? Answering this seems closely related to computing the $\mathbb{F}_2$-homology of the metastable Eilenberg--MacLane spaces $K(\mathbb{F}_2, n) := \Omega^\infty \Sigma^n \mathbb{F}_2$, which we have not attempted here. More generally, we do not yet understand the cofree functor $\Sigma^\infty \Omega^\infty$, of which this is a special case.

\begin{ques}
	Can one compute the $\mathbb{F}_2$-homology of the metastable Eilenberg--MacLane spaces $\Omega^\infty \Sigma^n \mathbb{F}_2$?
\end{ques}

The ideal outcome would be a usable algebraic description of the $E_2$-page, together with an effective way to implement it computationally. We leave this to future generations.

\sloppy
\bibliographystyle{alpha}
\bibliography{bib}

\end{document}